\documentclass[lettersize,journal]{IEEEtran}

\usepackage{amsmath,amsfonts}
\usepackage{array}
\usepackage[caption=false,font=normalsize,labelfont=rm,textfont=rm]{subfig}
\usepackage{textcomp}
\usepackage{stfloats}
\usepackage{url}
\usepackage{verbatim}
\usepackage{graphicx}
\usepackage{cite}
\usepackage{amsthm}
\usepackage{algorithmicx,algorithm}
\usepackage[noend]{algpseudocode}
\usepackage{amssymb}
\usepackage{bm}
\usepackage{float}
\usepackage{color,xcolor}
\usepackage{balance}

\def\BibTeX{{\rm B\kern-.05em{\sc i\kern-.025em b}\kern-.08em
		T\kern-.1667em\lower.7ex\hbox{E}\kern-.125emX}}

\newtheorem{theorem}{Theorem}

\newtheorem{remark}{Remark}
\newtheorem{lemma}{Lemma}
\newtheorem{assumption}{Assumption}
\newtheorem{definition}{Definition}

\newtheorem{proposition}{Proposition}

\begin{document}

\title{A Novel Privacy Enhancement Scheme with Dynamic Quantization for Federated Learning\\
}

\author{Yifan Wang, \IEEEmembership{Student Member, IEEE}, Xianghui Cao, \IEEEmembership{Senior Member, IEEE},\\ Shi Jin, \IEEEmembership{Fellow, IEEE}, and Mo-Yuen Chow, \IEEEmembership{Fellow, IEEE}
\thanks{Y. Wang and X. Cao are with the School of Automation,
Southeast University, Nanjing 210096, China (e-mail: \{evan,xhcao\}@seu.edu.cn).}
\thanks{S. Jin is with the National Mobile Communications Research Laboratory,
 Southeast University, Nanjing 210096, China (e-mail: jinshi@seu.edu.cn).}
\thanks{M.-Y. Chow is with the UM–SJTU Joint Institute Shanghai Jiao Tong
University, Shanghai 200240, China (e-mail: moyuen.chow@sjtu.edu.cn).}
\thanks{This work has been submitted to the IEEE for possible publication. Copyright may be transferred without notice, after which this version may no longer be accessible.}
}

\maketitle

\begin{abstract}
Federated learning (FL) has been widely regarded as a promising paradigm for privacy preservation of raw data in machine learning. Although data privacy in FL is locally protected to some extent, it is still a desideratum to enhance privacy and alleviate communication overhead caused by repetitively transmitting model parameters. Typically, these challenges are addressed separately, or jointly via a unified scheme that consists of noise-injected privacy mechanism and communication compression, which may lead to model corruption due to the introduced composite noise. In this work, we propose a novel model-splitting privacy-preserving FL (MSP-FL) scheme to achieve private FL with precise accuracy guarantee. Based upon MSP-FL, we further propose a model-splitting privacy-preserving FL with dynamic quantization (MSPDQ-FL) to mitigate the communication overhead, which incorporates a shrinking quantization interval to reduce the quantization error. We provide privacy and convergence analysis for both MSP-FL and MSPDQ-FL under non-i.i.d. dataset, partial clients participation and finite quantization level. Numerical results are presented to validate the superiority of the proposed schemes.
\end{abstract}



\section{Introduction}\label{sec:introduction}
\IEEEPARstart{E}{mpowered} by immense amount of data, deep learning has become a breakthrough for a broad range of applications, adequately unleashing the power of data \cite{fl_overview1}. Deep learning is usually implemented in a centralized manner, where a powerful server has access to the entire data, risking leakage of privacy-sensitive data. Meanwhile, the communication overhead substantially increases, owing to the collection of massive raw data at the server.
In response to these challenges, federated learning (FL) has emerged as a viable paradigm to prevent data leakage and reduce communication overhead \cite{advances_in_FL}. FL enables the independent training of distributed client models on local dataset, orchestrated by the central server \cite{fl_overview2}.

Typically, FL is designed by a decentralized parameter-server framework with a powerful server mastering the aggregation process, yet it is still desired to enhance the privacy of clients. Specifically, although data are not explicitly shared among clients, data privacy is still at risk of being divulged as adversaries are potentially able to reconstruct raw data of a target client based on the exchanged model parameters \cite{leakage}. There are many methods deal with extracting sensitive information from trained models, e.g., model-inversion attack, membership-inference attack and model-rebuild attack \cite{idlg,geiping2020inverting,reconstruct}.

To cope with the privacy concern, some approaches in \cite{zhaojiaqi_pvd-fl_tifs_2022,Thapa_splitfed_aaai_2022} enhanced the privacy in FL by homomorphic encryption and splitting learning, however, resulted in severe computational burden and poor learning efficiency, respectively.
An intuitive consideration involves integrating artificial noise into FL, namely differential private (DP)-based FL, known for its verifiable privacy guarantee \cite{The_algorithmic_foundations_of_differential_privacy_book}. In accordance with DP technique, local differential privacy (LDP) framework is widely adopted, wherein each local model parameter is perturbed by Laplacian or Gaussian noise before uploading \cite{truex_2020_ldp,jinshi}. Especially, the authors in \cite{kangwei_fl+dp_tifs_2020} were the first to theoretically analyze the convergence property of DP-based FL.
To mitigate the adverse impact of DP noise, the authors in \cite{yuanxin_Amplitude-Varying_Perturbation_for_Balancing_Privacy_and_Utility_tifs_2023} proposed a novel DP perturbation mechanism employing amplitude-varying noise.
Alternatively, in the context of average consensus, temporally or spatially correlated noise-based methods were devised for privacy preservation while maintaining accuracy \cite{Privacy-Preserving-Average-Consensus-moyilin-tac,wyq,yf}. For instance,
a state decomposition method in \cite{wyq} was proposed by decomposing each local state into two substates. However, such an approach is inapplicable to preserve privacy in decentralized FL by directly decomposing the model parameter according to certain public decomposition rule.

Regarding communication overhead, it becomes more urgent in privacy-preserving FL owing to the sacrifice of communication resource for privacy enhancement \cite{kangwei_fl+dp_tifs_2020}. In order to compress the communication, prior works usually turn to reducing the volume of transmitted model updates via sparsification or quantization \cite{uveq,pmlr-162-wang-Communication-Efficient_Adaptive_Federated_Learning,cuisg_arxiv_Performance_optimization_for_variable_bitwidth_federated}, aggregating with partially received information \cite{sayed_TSP_Federated_Learning_Under_Importance_Sampling_2022} and event-triggered communication \cite{Asymptotic-Analysis-of-Federated-Learning-Under-Event-Triggered-Communication-TSP-2023-KTH}. In this paper, we mainly focus on tools from quantization theory to tackle communication overhead under finite communication bandwidth. It is worth noting that some recent works proposed attractive designs for quantization to further improve the communication efficiency \cite{LAQ_2022_Globecom,aq3}.

The aforementioned studies separately considered privacy concern and communication overhead in FL. Yet the privacy enhancement brought by LDP actually sacrifices the convergence accuracy of the model parameter, and in turn requires more communication cost for the interactions between the server and clients. Thereby, it naturally emerges a necessity to consider both challenges of privacy enhancement and communication compression for FL. Recently, a handful of studies have jointly tackled privacy enhancement and communication efficiency in FL. The authors in \cite{Shuffled-Model-of-Differential-Privacy-in-Federated-Learning_2021_pmlr_Girgis} were the first to provide trade-off analysis on privacy, convergence and communication for convex loss functions. Later, in \cite{NIPS2022_SoteriaFL:A-Unified-Framework-for-Private-Federated}, a unified FL framework was proposed by exploiting general communication compression scheme and LDP mechanism for nonconvex optimization setting, together with concrete analysis on trade-offs among privacy, utility and communication.
To avoid injecting excessive noise, noise from subtractive dithered quantization (SDQ) was utilized together with a low-power and dedicated privacy-preserving noise to build up LDP mechanism, hence, obtaining a less distorted model via such a joint design \cite{joint}.
Nonetheless, these studies inevitably introduced persistent noise either from the LDP mechanism or lossy compression, compromising the model accuracy more or less. To the best of our knowledge, it still remains an open problem how to realize both privacy guarantee and communication efficiency in FL with exact model accuracy and provable convergence.

Driven by the aforementioned considerations, this paper emphasizes the necessity to simultaneously design privacy enhancement and compressed communication in FL. Concretely, instead of using LDP mechanism to enhance privacy, we firstly develop the state decomposition approach in FL coined model-splitting privacy-preserving FL (MSP-FL). Under MSP-FL, each local model parameter is split into two types of submodel parameters, wherein one of the submodel parameters is appointed as the visible submodel and the others are invisible. At each selected client, only the visible submodel parameter is uploaded to the channel. Note that, to avoid the observability of the remaining invisible submodel parameters, the number of invisible local submodel parameters is kept privately at each client. To proceed, on the basis of MSP-FL, we propose a model-splitting privacy-preserving with dynamic quantization FL (MSPDQ-FL), which is equipped with dynamic quantization intervals in channels to further reduce the volume of transmitted submodels. Given that MSP-FL ensures precise convergence accuracy if the number of communication rounds after each learning round is sufficiently large, hence, the privacy and quantization mechanisms can be independently designed in MSPDQ-FL. This is exactly one of the major differences between our work and \cite{NIPS2022_SoteriaFL:A-Unified-Framework-for-Private-Federated,joint}. The main contributions of this paper are summarized as follows:
\begin{itemize}
  \item We propose an MSP-FL for privacy enhancement in FL. In contrast to traditional LDP-based FL, MSP-FL can achieve exact model accuracy in expectation via correlated noise and specific interaction protocol designs for submodels. Meanwhile, a new privacy notion is used to measure the privacy degree due to the incompletely randomized characteristic of model splitting mechanism.
  \item To further reduce the communication overhead, we present a novel scheme coined MSPDQ-FL. By assigning dynamic quantization intervals, MSPDQ-FL asymptotically reduces the quantization errors during the transmissions.
  \item We rigorously analyze the convergence performance of both MSP-FL and MSPDQ-FL. The results theoretically prove that the model parameter errors stemming from privacy enhancement and quantization can be eliminated, as long as the number of communication rounds $K_t$ following each learning round $t$ exceed a lower bound.
  \item From the perspective of differential privacy, we explore the sufficient condition of model splitting rule for differential privacy noise. In addition, we prove that the proposed quantization scheme achieves at least $(0,\delta)$-differential privacy under provable convergence accuracy.
\end{itemize}

\textit{Organization:} Section \uppercase\expandafter {\romannumeral2} introduces the system model and threat model. Section \uppercase\expandafter {\romannumeral3} presents MSP-FL along with theoretical analysis of privacy and convergence properties. Section \uppercase\expandafter {\romannumeral4} further extends MSP-FL to MSPDQ-FL by incorporating dynamic quantization. Discussions are presented in Section \uppercase\expandafter{\romannumeral5}. Numerical results are shown in Section \uppercase\expandafter {\romannumeral6}. Finally, Section \uppercase\expandafter {\romannumeral7} concludes the paper.

\textit{Notations:} Throughout this paper, we introduce the following conventions: Sets are represented by upper calligraphy letters; vectors are represented by bold lowercase letters; matrices are denoted by bold uppercase letters; scalars are written as regular lowercase or uppercase letters.
$\mathbb{R}^{d}$, $^{*}\mathbb{R}^d$ and $\mathbb{R}^{n\times d}$ denote the sets of real vectors of dimension $d$, hyperreal vectors of dimension $d$, and real matrices of dimension $n\times d$, respectively. The identity matrix and the vector of ones are represented by $\mathbf{I}$ and $\mathbf{1}$ with proper dimensions, respectively. The superscript $\top$ denotes the transpose of a vector or matrix. Let $\text{col}(w_1,...,w_n)$ be a column vector of $\{w_i\in\mathbb{R}|i=1,...,n\}$. $[i]$ is the index number of $i$ in a set or sequence. For a matrix $\mathbf{U}$, $\lambda_{k,\mathbf{U}}$ denotes the $k$-th largest eigenvalue of $\mathbf{U}$, and $\lambda_{\min,\mathbf{U}}$ is the smallest one. Denote by $\|\cdot\|$ the Euclidean norm (or Frobenius norm) of a vector (or a matrix). The stochastic expectation is denoted by $\mathbb{E}\{\cdot\}$.

\section{Preliminaries}
This section reviews a basic framework of FL and presents the considered problem together with threat model.

\subsection{System Model}
In FL, a server and a set of clients $\mathcal{N}=\{1,...,N\}$ collaboratively train a model, parameterized by $\mathbf{w}\in \mathbb{R}^d$, by solving the following decentralized learning problem
\begin{equation}\label{learning problem}
  \min\limits_{\mathbf{w}\in\mathbb{R}^d}\quad F(\mathbf{w})\triangleq \sum_{i=1}^{N}p_iF_i(\mathbf{w}),
\end{equation}
where $F_i(\mathbf{w})$ is the local loss function of $i$-th client which evaluates the fitness between the global model parameter\footnote{In the remaining paper, the `model parameter' is called `model' for brevity.} $\mathbf{w}$ and the local dataset $\mathcal{D}_i$. ${p}_i={|\mathcal{D}_i|}/{\sum_{i=1}^{N}}|\mathcal{D}_i|$, where $|\mathcal{D}_i|$ denotes the size of $\mathcal{D}_i$. Therewith, one can conclude that the global loss function $F(\mathbf{w})$ evaluates how well the global model fits the overall data of all the clients. Given a mini-batch sample $\bm{\zeta}_i$ with data size $s_i$, the local loss function can be further represented as
$
  F_i(\mathbf{w},\bm{\zeta}^i)=\frac{1}{s_i}\sum_{j=1}^{s_i}f_i(\mathbf{w},\zeta^{i,j}),
$
where $f_i(\mathbf{w},\zeta^{i,j})$ is the local loss function with respect to the $j$-th selected data sample in $\bm{\zeta}^i$. If the mini-batch sample is the whole local dataset, i.e., $s_i=|\mathcal{D}_i|$, it has $F_i(\mathbf{w},\bm{\zeta}^i)=F_i(\mathbf{w})$.

For the learning problem in \eqref{learning problem}, we make the following assumptions that are commonly employed in the context of distributed learning \cite{niid,uveq,joint,Asymptotic-Analysis-of-Federated-Learning-Under-Event-Triggered-Communication-TSP-2023-KTH}.

\begin{assumption}
  (Assumptions on loss functions and gradients)

  1) \textit{(Convexity) Each loss function $F_i$ is $\mu$-strongly convex, i.e., for all $\mathbf{v}$ and $\mathbf{w}$, it has
        $F_i(\mathbf{v})-F_i(\mathbf{w})-\nabla F_i(\mathbf{w})^{\top}(\mathbf{v}-\mathbf{w})\geq \frac{\mu}{2}\|\mathbf{v}-\mathbf{w}\|^2$.}

2) \textit{(Smoothness) Each loss function $F_i$ is $L$-Lipschitz smooth, i.e., for all $\mathbf{v}$ and $\mathbf{w}$, it has
        $F_i(\mathbf{v})-F_i(\mathbf{w})-\nabla F_i(\mathbf{w})^{\top}(\mathbf{v}-\mathbf{w})\leq \frac{L}{2}\|\mathbf{v}-\mathbf{w}\|^2$.}

3) \textit{(Bounded variance of stochastic gradients) For each client, the variance of the stochastic gradient on any mini-batch data is bounded by a positive constant $\sigma_i$, i.e.,
    $\mathbb{E}\left\{\|\nabla F_i(\mathbf{w}_{t}^i,\bm{\zeta}^i)-\nabla F_i(\mathbf{w}_{t}^i)\|^2\right\}\leq\sigma_i$, $\forall i\in\mathcal{N}$ and $\forall t$.}

4) \textit{(Uniformly bounded stochastic gradients) There exists a positive constant $G>0$, such that the expected squared norm of the stochastic gradients are uniformly bounded, i.e.,
    $\mathbb{E}\left\{\|\nabla F_i(\mathbf{w}_{t}^i,\bm{\zeta}^i)\|^2\right\}\leq G$, $\forall i\in\mathcal{N}$ and $\forall t$.}
\end{assumption}

In this work, we consider a general FL system model that follows FedAvg with only partial clients participating in model uploading after each learning round and with non-i.i.d. data distribution$\footnote{The degree of non-i.i.d., namely heterogeneity, is quantified as $\Gamma= F^*-\sum_{i=1}^{N}p_iF_i^*$, where $F^*$ and $F_i^*$ denote the minimum values of global loss function and local loss function, respectively \cite{niid}.}$ \cite{fedavg}. Specifically, the server and clients execute the following procedures at each learning round $t$:

(1) \textbf{Server broadcasts global model to partial clients.} The server randomly selects a set of clients, denoted as $\mathcal{S}_{t}$ with $|\mathcal{S}_{t}|=M$, and broadcasts the latest global model $\mathbf{w}_{t-1}$ to the selected clients.

(2) \textbf{Clients update local models.} Based on the received global model, each client $i\in\mathcal{S}_{t}$ updates its local model by iteratively running the mini-batch stochastic gradient descent (SGD) method as follows:
\begin{subequations}\label{local training}
    \begin{eqnarray}
       \hspace{-17pt} \mathbf{w}_{t,0}^i\hspace{-7pt}&=&\hspace{-7pt}\mathbf{w}_{t-1},\\
       \hspace{-17pt} \mathbf{w}_{t,e+1}^i \hspace{-7pt}&=&\hspace{-7pt} \mathbf{w}_{t,e}^i-\eta_{t}\nabla F_i(\mathbf{w}_{t,e}^i,\mathbf{\zeta}_e^i), \forall e\in\{0,...,E-1\},\\
       \hspace{-17pt} \mathbf{w}_{t}^i\hspace{-7pt}&=&\hspace{-7pt}\mathbf{w}_{t,E}^i,
    \end{eqnarray}
\end{subequations}
where $E$ is the number of SGD updates during each learning round. $\eta_{t}$ is a time-varying learning rate to be designed. $\mathbf{w}_{t}^i$ is the updated local model for upload.

(3) \textbf{Clients upload local models for global aggregation.} Clients in $\mathcal{S}_{t}$ upload local models to the server. Then, the server generates a new global model by average aggregation. The learning process terminates after a total of $T$ rounds.

\subsection{Threat Model}
We assume that there exists a trusted authority in the considered FL scenario besides the clients and the central server, where the trusted authority provides trusted and limited information exchanges among legitimate clients. Meanwhile, adversaries are considered as follows:

\textit{Honest-but-curious adversaries} follow the stipulated protocol without falsifying the shared information. Meanwhile, adversaries attempt at least two classes of attacks: i) \emph{data-oriented attacks}-such as model inversion, membership inference, and gradient/parameter-leakage (e.g., iDLG)—which aim to reconstruct an individual example or test its presence in a client’s dataset and can even recover labels or images from a single step \cite{reconstruct,idlg}; (ii) \emph{model-oriented attacks}-such as model and functionality stealing attacks which try to infer a client's local model parameters or create a knock-off model that should be indistinguishable to the adversary \cite{modelprivacy_attack1,modelprivacy_attack2}. Both types of attacks pose a direct and substantial threat to model privacy. We denote them as a set $\mathcal{A}_h\triangleq\mathcal{A}_c\bigcup \mathcal{A}_s$, where $\mathcal{A}_c$ denotes the set of honest-but-curious clients and $\mathcal{A}_s$ denotes the server. Note that the information is shared with each other in $\mathcal{A}_h$.

The considered threat model that consists of both honest-but-curious server and clients is commonly seen in the field of FL \cite{validity2}. Meanwhile, such a threat model exist in many practical scenarios that are in the form of non-cooperative game problem \cite{validity6,validity7}. For instance, in the problem of grid-to-vehicle energy exchange between a smart grid and plug-in electric vehicle groups (PEVG), each PEVG aims to maximize the benefit from consumed energy or minimize the cost, while the grid aims to maximize the revenue by selling energy surplus to PEVGs. Therefore, the curious grid tends to control the nature of energy consumption of PEVGs. Most PEVGs tend to protect their own privacy and are curious about others' privacy information including energy consumption in order to maximize their own utility. However, some PEVGs might collaboratively share inferred information with the grid, one possibility being that they are stakeholders in the market (e.g., an energy trading center and its affiliated units). In extreme cases, e.g., when all clients are compromised, the honest-but-curious server may share inferences with all clients. Regarding the implementation of privacy inference sharing, it can be directly achieved through the existing communication channels between the server and (honest-but-curious) clients. For privacy inference sharing among (honest-but-curious) clients, the server can support the sharing by acting as a relay.

In some other works, external eavesdroppers are also considered as a type of adversaries, which wiretap all transmitted data and attempt to estimate local models by certain estimating algorithm. Usually, eavesdroppers are more disruptive than honest-but-curious adversaries, because they can estimate local models by certain algorithms with the accumulated information. However, owing to the existence of the honest-but-curious server, external eavesdroppers actually access less information than honest-but-curious adversaries. Hence, we mainly consider the privacy preservation in FL against adversaries in the set $\mathcal{A}_h$ in this paper.

\section{MSP-FL: Model-Splitting Privacy-preserving Federated Learning}
In this section, we first present MSP-FL, and then provide analysis of its privacy and convergence properties.

\subsection{The Proposed Privacy-preserving Scheme: MSP-FL}
To simultaneously make provision for privacy enhancement and learning accuracy, the key idea of MSP-FL which is inspired by \cite{wyq}, is stochastically splitting each local model into two types of submodels, i.e., a visible and invisible submodels, and uploading the visible ones to the server. Compared to FedAvg, MSP-FL brings about subsequent $K_t$ communication rounds between the server and clients after each learning round $t$ instead of one round communication.

The overall illustration of MSP-FL is depicted in Fig. 1.
{Concrete details about model splitting mechanism\footnote{{Model splitting in our work differs from the split learning explored in other research such as \cite{splitting_learning}. In our proposed model splitting, the model parameters are randomly partitioned into submodel parameters, while preserving the dimensionality during aggregation. In contrast, split learning aims to divide the machine learning model (e.g., a deep neural network) into at least two submodels, typically involving a few layers of the entire network, resulting in dimensionality reduction. As such, each client only trains a submodel that consists of a subset of layers, while the remaining layers are trained at the server. In the remaining paper, the `submodel parameter' is called `submodel' for brevity, which actually differs from `submodel' in split learning.}} are stated as follows: Let's denote the index of an arbitrarily selected client as $i$. At each learning round $t$, once entirely finishing training its local model $\mathbf{w}_{t}^i$ as described in \eqref{local training}, it splits $\mathbf{w}_{t}^i$ into $m_i+1$ submodels where $m_i\geq 1$ is unknown to any client $j\in\{\mathcal{S}_t\setminus i\}$ and the server, and randomly assigns initial values for submodels. Without loss of generality, we denote one of submodels as $\mathbf{w}_{t}^{i,\alpha}\in\mathbb{R}^d$, and the others as $\mathbf{w}_{t}^{i,\beta_n}\in\mathbb{R}^d$, $n\in\{1,...,m_i\}$. The criteria for randomly selecting initial values of submodels is: The initial value of $\mathbf{w}_{t}^{i,\alpha}$, denoted by $\mathbf{w}_{t}^{i,\alpha}[0]$ is randomly drawn from a uniform distribution over the interval $[\varepsilon \mathbf{w}_{t}^i,(2-\varepsilon) \mathbf{w}_{t}^i]$ where $\varepsilon\in [0,1)$ is a splitting factor. Initial values of other submodels, denoted by $\mathbf{w}_{t}^{i,\beta_n}[0]$, $n\in\{1,...,m_i\}$, are not required to be selected in this manner but satisfy $\sum_{n=1}^{m_i}\mathbf{w}_{t}^{i,\beta_n}[0]=(1+m_i)\mathbf{w}_{t}^{i}-\mathbf{w}_{t}^{i,\alpha}[0]$.
After such a model splitting procedure, client $i$ uploads $\mathbf{w}_{t}^{i,\alpha}[0]$ to the server, while $\mathbf{w}_{t}^{i,\beta_n}[0]$, $n\in\{1,...,m_i\}$, are kept private and only interact internally with $\mathbf{w}_{t}^{i,\alpha}[0]$. It is precisely for this fact that we refer to $\mathbf{w}_{t}^{i,\alpha}$ as the visible submodel and $\mathbf{w}_{t}^{i,\beta_n}$, $n\in\{1,...,m_i\}$, as invisible submodels for any $i\in\mathcal{S}_t$. The reason for such a selection is discussed in Remark 1.} In the subsequent communication rounds between the server and the selected clients after the learning round $t$, the $i$-th client downloads the latest global model from the server, which is an average aggregation of visible submodels (i.e., $\frac{1}{M}\sum_{i\in\mathcal{S}_{t}}\mathbf{w}_{t}^{i,\alpha}[k]$), to update the current local submodels.
\begin{figure}[t]
   \centering
   \includegraphics[scale=0.22]{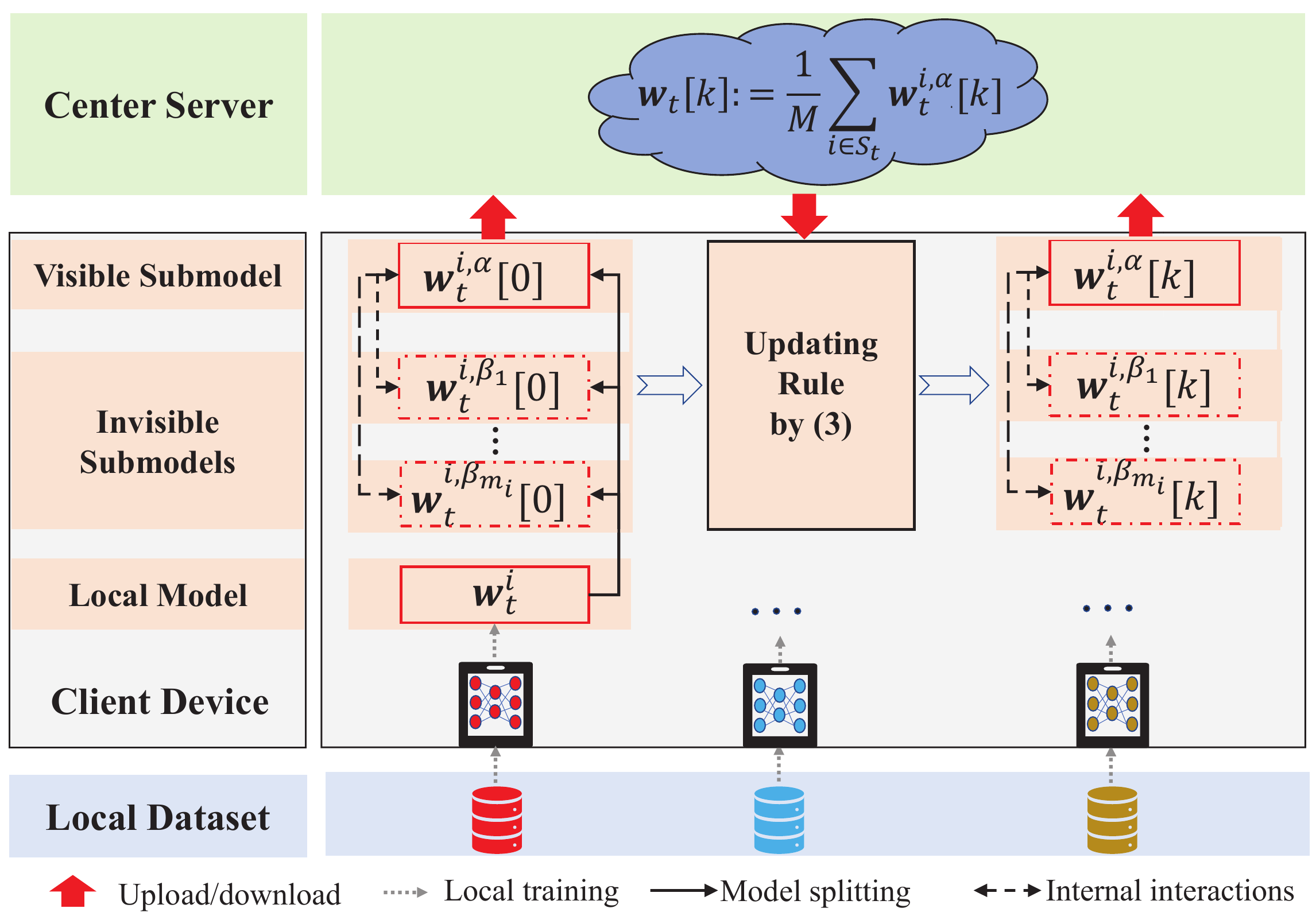}
   \caption{Illustration of MSP-FL}\label{fig-MSP-FL}
 \end{figure}

In such a manner, one can see that only the visible part, which acts the role of the original local model, is exposed to the server or an eavesdropper either in the uplink or downlink channels, though the invisible part internally interacts with it.

Under the proposed MSP-FL, the updating rule of the selected $i$-th client after $t$-th learning round is given as follows:
\begin{subequations}\label{decomposition1}
    \begin{eqnarray}
     \hspace{-10pt}  \mathbf{w}_{t}^{i,\alpha}[k+1] \hspace{-7pt}&=&\hspace{-7pt} \mathbf{w}_{t}^{i,\alpha}[k]+ \epsilon\left(\mathbf{w}_{t}[k]-\mathbf{w}_{t}^{i,\alpha}[k]\right)\nonumber \\
       &&\hspace{-7pt}+\sum_{n=1}^{m_i}a_{i,\alpha\beta_n}[k]\left(\mathbf{w}_{t}^{i,\beta_n}[k]-\mathbf{w}_{t}^{i,\alpha}[k]\right),\label{decomposition1a} \\
     \hspace{-10pt}  \mathbf{w}_{t}^{i,\beta_n}[k+1] \hspace{-7pt}&=& \hspace{-7pt} \mathbf{w}_{t}^{i,\beta_n}[k]+ a_{i,\alpha\beta_n}[k]\left(\mathbf{w}_{t}^{i,\alpha}[k]-\mathbf{w}_{t}^{i,\beta_n}[k]\right),\nonumber\\
     \hspace{-7pt}&&\hspace{-7pt}\qquad\qquad\qquad\qquad n=1,2,...,m_i, \label{decomposition1b}
    \end{eqnarray}
\end{subequations}
where $\epsilon$ is a design parameter. $a_{i,\alpha\beta_n}[k]$ is a weight between the visible submodel and the $n$-th invisible submodel of client $i$. $\mathbf{w}_{t}[k]$ is the global model at the $k$-th communication round.
Correspondingly, the current $(k+1)$-th aggregation at the server side is given by
\begin{eqnarray}\label{aggregation2}
   \mathbf{w}_{t}[k+1]\hspace{-7pt}&=&\hspace{-7pt}\frac{1}{M}\sum_{i\in\mathcal{ S}_{t}}\mathbf{w}_{t}^{i,\alpha}[k+1].
\end{eqnarray}

From \eqref{decomposition1} and \eqref{aggregation2}, one can observe that
\begin{eqnarray}\label{sum}
   \hspace{-19pt} &&\hspace{-7pt} \sum_{i\in\mathcal{S}_{t}}  \left( \mathbf{w}_{t}^{i,\alpha}[k+1]+\sum_{n=1}^{m_i}\mathbf{w}_{t}^{i,\beta_n}[k+1]\right)\nonumber\\
\hspace{-19pt} &=&\hspace{-7pt}\sum_{i\in\mathcal{S}_{t}} \left( \mathbf{w}_{t}^{i,\alpha}[k]+ \sum_{n=1}^{m_i} \mathbf{w}_{t}^{i,\beta_n}[k]+ \epsilon\left(\mathbf{w}_{t}[k]-\mathbf{w}_{t}^{i,\alpha}[k]\right)\right).
\end{eqnarray}
Clearly, by substituting \eqref{aggregation2} into \eqref{sum}, the sum of visible part and invisible part over the selected clients is time-invariant with respect to $k$ after the current learning round. Meanwhile, this together with the asymptotic convergence result, that is, it guarantees to reach average consensus under dynamics \eqref{decomposition1} with appropriate step sizes and weights, yields
$
  \lim_{k\rightarrow\infty}\mathbf{w}_{t}^{i,\alpha}[k]=\lim_{k\rightarrow\infty}\mathbf{w}_{t}^{i,\beta_n}[k], \forall i\in\mathcal{S}_{t}, n\in\{1,...,m_i\}
$. That is to say, privacy-utility trade-off can be avoided via dedicated designs for injected noises and interaction protocol as in \eqref{decomposition1} for submodels. However, in practical settings of FL, both the visible and invisible parts are unable to achieve completely consensus state due to the restricted communication resources. Thereby, following a finite and variable number of communication rounds denoted by $K_t$ that depends on $t$, the global model is determined as an aggregated value of the last communication round, i.e.,
\begin{equation}\label{aggregation3}
  \mathbf{w}_{t}=\frac{1}{M}\sum_{i\in\mathcal{ S}_{t}}\mathbf{w}_{t}^{i,\alpha}[K_t].
\end{equation}

Additionally, the clients selection scheme in this work is in line with \cite{participation,niid}. Specifically, the clients contained in $\mathcal{S}_{t}$ are randomly selected with replacement for $M$ times according to the independent distribution probabilities $p_1,...,p_N$. Then, it has the following lemma about unbiased sampling scheme.

\begin{lemma}
  Under the above random clients selection scheme, the partial aggregation after $t$-th learning round as shown in \eqref{aggregation3} is an unbiased estimation of the full aggregation $\tilde{\mathbf{w}}_{t}=\sum_{i=1}^{N}p_i\mathbf{w}_{t}^{i,\alpha}[K_t]$ with regard to any $\mathcal{S}_t$, i.e., $\mathbb{E}\{\mathbf{w}_{t}\}=\tilde{\mathbf{w}}_{t}$.
\end{lemma}

\begin{proof}
   The proof of Lemma 1 follows the line of proof of Lemma 4 in \cite{niid}. Denote a fixed sequence as $\{\mathbf{w}_{t}^{i,\alpha}[K_t]|i=1,...,N\}$, and a multiset $\mathcal{S}_{t}=\{i_1,...,i_M\}\subset \mathcal{N}$ is sampled at sampling time $t$ where each sample $\mathbf{w}_{t}^{i,\alpha}[K_t]$ is selected with probability $q_i$. Meanwhile, we require each sampling distribution to be identical. Then, we have the following observation:
   \begin{eqnarray}\label{lemma1_7}
\hspace{-10pt}\mathbb{E}\Big\{\frac{1}{M}\sum_{j\in\mathcal{S}_{t}}\mathbf{w}_{t}^{j,\alpha}[K_t]\Big\}\hspace{-7pt}&=&\hspace{-7pt}\mathbb{E}\Big\{\frac{1}{M}\sum_{j=1}^{M}\mathbf{w}_{t}^{i_j,\alpha}[K_t]\Big\}\nonumber\\
\hspace{-7pt}&=&\hspace{-7pt}\mathbb{E}\left\{\mathbf{w}_{t}^{i_j,\alpha}[K_t]\right\}, \quad \forall i_j\in\mathcal{S}_t,
   \end{eqnarray}
where the last equality is obtained by the identical sampling distribution. Moreover, recalling that $\mathbb{E}\{\mathbf{w}_{t}^{i_j,\alpha}[K_t]\}=\sum_{i=1}^{N}q_i\mathbf{w}_t^{i,\alpha}[K_t]$. By selecting the probability $q_i$ as $p_i$, we complete the proof.
\end{proof}

\begin{remark}
  Note that, after the entire training process at each learning round $t$, the initial value of each visible submodel, i.e., $\mathbf{w}_{t}^{i,\alpha}[0]$, is randomly drawn from the uniform distribution over the interval $[\varepsilon \mathbf{w}_{t}^i,(1+m_i-\varepsilon) \mathbf{w}_{t}^i]$. The reason is that, as such, it follows with two key properties that are shown in Lemma 8: (1) unbiased stochastic splitting; (2) bounded aggregation error of visible submodels. For those invisible submodels, we can also randomly select their initial values in a similar criterion as long as the condition $\sum_{n=1}^{m_i}\mathbf{w}_{t}^{i,\beta_n}[0]=(1+m_i)\mathbf{w}_{t}^{i}-\mathbf{w}_{t}^{i,\alpha}[0]$ still holds.
\end{remark}

\subsection{The Property of Privacy Preservation}
In this subsection, we will present analysis of MSP-FL for its privacy guarantee in the presence of adversaries.

For the sake of subsequent analysis, some notations are given as follows. Let us rewrite the term $ \epsilon(\mathbf{w}_{t}[k]-\mathbf{w}_{t}^{i,\alpha}[k])$ in \eqref{decomposition1a} as $\epsilon  \sum_{j\in\mathcal{S}_t}\alpha_j[k](\mathbf{w}_t^{j,\alpha}[k]-\mathbf{w}_t^{i,\alpha}[k])$, where $\alpha_j[k]=\frac{1}{M}$ for all $j\in\mathcal{S}_t$ and $k\in\{0,1,...,K_t-1\}$. Denote the information set accessible to $\mathcal{A}_h$ at the $k$-th communication round after the $t$-th learning round as $\mathcal{I}_{t}^{\mathcal{A}_h}[k]=\{\mathcal{I}_{t}^j[k],\mathcal{I}_{t}^{\mathcal{A}_s}[k]|j\in\mathcal{A}_c\}$, where
$\mathcal{I}_{t}^j[k]=\{\mathbf{w}_{t}^{j,\alpha}[k],\mathbf{w}_{t}^{j,\beta_n}[k],\alpha_j[k],a_{j,\alpha\beta_n}[k],\mathbf{w}_{t}[k]|n\in\{1,...,m_j\}\}$
and $\mathcal{I}_{t}^{\mathcal{A}_s}[k]=\{\mathbf{w}_{t}^{i,\alpha}[k]|i\in\mathcal{S}_{t}\}$. 

To see the performance of privacy preservation, we first show the nondeterminacy of local models.

\begin{lemma}\label{lemma2}
  Under MSP-FL, for any $i\in\mathcal{S}_{t}$, if $m_i$ is inaccessible to $\mathcal{A}_h$, then the local model $\mathbf{w}_{t}^i$ is indeterminate to any adversary in $\mathcal{A}_h$ in an asymptotic sense.
\end{lemma}
\begin{proof}
Let $\mathbf{z}_{t}^i[k]=\mathbf{w}_{t}^{i,\alpha}[k]+\sum_{n=1}^{m_i}\mathbf{w}_{t}^{i,\beta_n}[k]$. By \eqref{decomposition1a} and \eqref{decomposition1b}, one has
\begin{equation*}
  \mathbf{z}_{t}^i[k+1]=\mathbf{z}_{t}^i[k]+\epsilon(\mathbf{w}_{t}[k]-\mathbf{w}_{t}^{i,\alpha}[k]).
\end{equation*}
By summing the above equation for all $k$ together, we have
\begin{eqnarray}\label{z}
  \mathbf{z}_{t}^i[0]\hspace{-7pt}&=&\hspace{-7pt}\lim\limits_{k\rightarrow\infty}\bigg[\mathbf{z}_{t}^i[k]\hspace{0pt}-\hspace{0pt}\epsilon\sum_{k_1=0}^{k-1}\Big({{\mathbf{w}_{t}[k_1]}\hspace{0pt}-\hspace{0pt}\mathbf{w}_{t}^{i,\alpha}[k_1]}\Big) \bigg].
\end{eqnarray}
One can observe that any adversary in $\mathcal{A}_h$ can not exactly infer the value of $\lim_{k\rightarrow\infty}\mathbf{z}_{t}^i[k]$, due to the unknown number of local invisible submodels. Consequently, $\mathbf{w}_{t}^i$ is not uniquely estimated by any adversary in $\mathcal{A}_h$ by $\mathbf{z}_t^i[0]$ from \eqref{z}.
\end{proof}

\begin{remark}
  For any $i\in\mathcal{S}_{t}$, if $m_i$ is fixed and known to $\mathcal{A}_h$, then MSP-FL degrades to a direct application of the state decomposition mechanism in \cite{wyq}. In this circumstance, $\mathbf{w}_{t}^i$ is uniquely determined by the adversary set $\mathcal{A}_h$ by \eqref{z} in an asymptotic sense because
  $\lim_{k\rightarrow\infty}\mathbf{z}_{t}^i[k]=\lim_{k\rightarrow\infty}(1+m_i)\mathbf{w}_{t}^{i,\alpha}[k]$. Even though $k=K_t<\infty$ in FL resulting in an incomplete consensus among local submodels, an adversary can still derive a sufficiently approximate range of $\mathbf{w}_{t}^i$ from an estimated $\mathbf{z}_t^i[K_t]$. The observation reveals that direct application of the state decomposition mechanism in \cite{wyq} is unserviceable to preserve the privacy in FL. An intuitive explanation is that, in a decentralized network, it can be deemed that any legitimate client $j$ of client $i$ is directly connected to $\mathcal{A}_h$, resulting in the term ${\mathbf{w}_{t}[k_1]}-\mathbf{w}_{t}^{i,\alpha}[k_1]$ in \eqref{z} losing its capability to help preserve the initial value $\mathbf{z}_{t}^i[0]$. Once the convergence value of $\lim_{k\rightarrow\infty}\mathbf{z}_{t}^i[k]$ or its approximate range can be easily obtained, then $\mathbf{z}_{t}^i[0]$ is estimated with some probability. 
  In our work, we show that one of the cruxes to alter the determinacy of local model in such a decentralized network by \eqref{z} is to hide the local splitting rule, specifically, to simply hide the number of local invisible submodels. This is one of the differences between this work and others following the line of state decomposition mechanism, e.g., \cite{wyq}.
\end{remark}

Now, we show the indistinguishability of the information set accessible to adversaries in $\mathcal{A}_h$. Define $\mathcal{I}_{t}^{\mathcal{A}_h}(0,K_t)$ as the union of information set until the time instant $k=K_t$ after the $t$-th learning round. For any feasible $\mathcal{I}_{t}^{\mathcal{A}_h}(0,K_t)$, denote $\mathcal{S}(\mathcal{I}_{t}^{\mathcal{A}_h}(0,K_t),i)$ as the set of all feasible local models $\mathbf{w}_{t}^i$ at client $i$ that a set of initial submodels $\{\mathbf{w}_{t}^{l,\alpha}[0],\mathbf{w}_{t}^{l,\beta_n}[0]|l\in\mathcal{S}_{t},n\in\{1,...,m_l\}\}$ and a set of weight sequences $\{a_{l,\alpha\beta_n}[k],\alpha_l[k]|l\in \mathcal{S}_{t}, n\in\{1,...,m_l\}, k\in\{0,1,...,K_t-1\}\}$ exist such that the set of sequences $\{a_{j,\alpha\beta_n}[k_1],\alpha_j[k_1],\mathbf{w}_{t}^{i,\alpha}[k_2],\mathbf{w}_{t}^{j,\beta_n}[k_2]| j\in\mathcal{A}_c, i\in\mathcal{S}_{t}, n\in\{1,...,m_j\}, k_1\in\{0,1,...,K_t-1\}, k_2\in\{0,1,...,K_t\}\}$ generated by MSP-FL is the same as that in the adversary information set $\mathcal{I}_{t}^{\mathcal{A}_h}(0,K_t)$.

The intrinsic meaning of set $\mathcal{S}(\mathcal{I}_{t}^{\mathcal{A}_h}(0,K_t),i)$ is that it comprises all possible current local models of client $i$ that yields the same $\mathcal{I}_{t}^{\mathcal{A}_h}(0,K_t)$. Meanwhile, we define the diameter of $\mathcal{S}(\mathcal{I}_{t}^{\mathcal{A}_h}(0,K_t),i)$ as
\begin{equation}\label{diam}
  \hspace{-6pt}\text{Diam}(\mathcal{I}_{t}^{\mathcal{A}_h}(0,K_t))=\sup\limits_{\{\mathbf{w}_{t}^{i},\overline{\mathbf{w}}_{t}^i\}\in\mathcal{S}(\mathcal{I}_{t}^{\mathcal{A}_h}(0,K_t),i)}\|\mathbf{w}_{t}^{i}-\overline{\mathbf{w}}_{t}^i\|.
\end{equation}

In light of \eqref{diam}, we introduce the considered model-oriented privacy notion to evaluate the privacy degree of MSP-FL. Similar privacy notions can be found in literatures \cite{l_diversity,auto_2020_diamter,chenxiaomeng_pp_tac,least_information,notionpaper3}.

\begin{definition}
  The privacy of the local model of client $i$ after $t$-th learning round is said to be insensitive to any adversary in $\mathcal{A}_h$ if $\text{Diam}(\mathcal{I}_{t}^{\mathcal{A}_h}(0,K_t))=\infty$ for any feasible $\mathcal{I}_{t}^{\mathcal{A}_h}(0,K_t)$.
\end{definition}

{The privacy notion in Definition 1 is extended from the concept of $\ell$-diversity \cite{l_diversity}. It measures the diversity of sensitive information of a selected client $i$ (i.e., the local model $\mathbf{w}_{t}^i$ as well as its associated data), which is represented by the diameter of the set $\mathcal{S}(\mathcal{I}_{t}^{\mathcal{A}_h}(0,K_t),i)$. In FL, we treat the continuous-valued model parameter vector $\mathbf{w}_t^i$, for $i\in\mathcal{S}_t$ after any learning round $t$, as the sensitive information and measure its diversity by $\text{Diam}(\mathcal{I}_{t}^{\mathcal{A}_h}(0,K_t))$. 
}

\begin{remark}
     The proposed model splitting mechanism is still in line with noise-adding approaches, wherein the added noise for assigning initial values of submodels are spatially correlated and zero-sum. A distinctive design is the subsequent multiple rounds of interaction between local submodels and the server as described in \eqref{decomposition1a} and \eqref{decomposition1b}. This process enables the asymptotic elimination of spatially correlated noise, thereby avoiding the trade-off between learning accuracy and privacy. However, regarding the dynamics of this interaction protocol, there is no stochasticity. Moreover, we do not require any specific statistic model for the noise resulting from model splitting. These features render probability-based privacy definitions, e,g., differential privacy, unsuitable for evaluating the privacy preservation provided by our proposed method, and necessitate the introduction of the privacy notion in Definition 1 which indicates the indistinguishability in a definite sense between any two adjacent model parameters $\forall i\in\mathcal{S}_t$.
\end{remark}

\begin{remark}
  Under Definition 1, an extreme situation of privacy leakage is that $\text{Diam}(\mathcal{I}_{t}^{\mathcal{A}_h}(0,K_t))=0$, implying that an adversary is able to uniquely determine a model parameter $\mathbf{w}_t^i$ via its accessible information set at learning round $t$. Such an extreme situation of privacy leakage is possible if the rules of model splitting and interactions among local submodels are known to an adversary, e.g., the hidden splitting number $m_i$ is disclosed.
\end{remark}

In the literature on privacy-preserving federated learning, there are various privacy definitions, wherein differential privacy stands as the prominent position. 
Preliminaries of standard data-oriented differential privacy and model-oriented differential privacy are given below.
\begin{definition}\label{adjacent}
      (Adjacent datasets \cite{The_algorithmic_foundations_of_differential_privacy_book}) Two datasets at client $i$ denoted as $\mathcal{D}_i^{(1)}=\{\mathcal{D}_{i,j}^{(1)}\}_{j=1}^n$ and $\mathcal{D}_i^{(2)}=\{\mathcal{D}_{i,j}^{(2)}\}_{j=1}^n$ are said to be adjacent if there exists some $j_0\in\{1,2,...,n\}$ such that $\mathcal{D}_{i,j}^{(1)}=\mathcal{D}_{i,j}^{(2)}$ for any $j\neq j_0$ and $\mathcal{D}_{i,j_0}^{(1)}\neq\mathcal{D}_{i,j_0}^{(2)}$.
\end{definition}
\begin{definition}
 (data-oriented $\epsilon_{DP}$-differential privacy \cite{The_algorithmic_foundations_of_differential_privacy_book}) Given $\epsilon_{DP}>0$, a randomized mechanism $\mathcal{M}:\mathcal{X}\rightarrow \mathcal{R}$ with domain $\mathcal{X}$ and range $\mathcal{R}$ satisfies $\epsilon_{DP}$-differential privacy, if for any observation $\mathcal{S}\subseteq\mathcal{R}$ and for any pair of adjacent datasets $\mathcal{D}_i^{(1)}$ and $\mathcal{D}_i^{(2)}\in\mathcal{X}$, it has
    \begin{equation}\label{DP}
      Pr(\mathcal{M}(\mathcal{D}_i^{(1)})\in\mathcal{S})\leq e^{\epsilon_{DP}}Pr(\mathcal{M}(\mathcal{D}_i^{(2)})\in\mathcal{S}).
    \end{equation}
\end{definition}
\begin{definition}\label{def3}
         ($\sigma$-adjacent model parameters\footnote{The concept of adjacent model parameters follows from $\sigma$-adjacency defined for state variables in dynamic systems \cite{hjp_tsp_dp}.}) Given $\sigma\in\mathbb{R}_{+}$, the model parameters ${x}$ and ${y}$ are $\sigma$-adjacent if, for some $i_0\in \{1,2,...,d\}$,
            \begin{equation*}
              |x_i - y_i| \leq 
              \begin{cases} 
                \sigma, & \text{if } i = i_0 \\
                0, & \text{otherwise}
              \end{cases}
            \end{equation*}
           for $i\in \{1,2,...,d\}$, where $x,y\in\mathbb{R}^d$.
        \end{definition}
    \begin{definition}\label{def4}
    (model-oriented $\epsilon_{DP}$-differential privacy) Given $\epsilon_{DP}>0$, a randomized mechanism $\mathcal{M}:\mathbb{R}^d\rightarrow \mathcal{R}$ with domain $\mathcal{X}$ and range $\mathcal{R}$ satisfies $\epsilon_{DP}$-differential privacy, if for any observation $\mathcal{S}\subseteq\mathcal{R}$ and for any pair of $\sigma$-adjacent model parameters $\mathbf{w}_i^{(1)}$ and $\mathbf{w}_i^{(2)}\in\mathbb{R}^d$, it has
        \begin{equation*}
          Pr(\mathcal{M}(\mathbf{w}_i^{(1)})\in\mathcal{S})\leq e^{\epsilon_{DP}}Pr(\mathcal{M}(\mathbf{w}_i^{(2)})\in\mathcal{S}).
        \end{equation*}
    \end{definition}
In what follows, we discuss the connections between the privacy notion in Definition 1 and differential privacy notion. 
\begin{remark}
  Compared to DP, Definition 1 gives only prerequisites while falling short of established privacy guarantees with probabilistic resolution. Even in earlier works that introduce this type of privacy notion, such as \cite{auto_2020_diamter,wyq}, provable privacy guarantees akin to those in DP remain absent. The intrinsic meaning of DP is that an adversary cannot distinguish between any two adjacent inputs with certain probability based on the output of a randomized mechanism. In this regard, Definition 1 is similar to DP as it also evaluates the indistinguishability of the inputs (i.e., model parameters) of a privacy-preserving mechanism when the public information accessible to the adversary is always in an identical set $\mathcal{I}_t^{\mathcal{A}_h}(0,K_t)$ (i.e., the resulting public information accessible to the adversary is always the same). Infinite diameter, i.e., $\text{Diam}({\mathcal{I}_t^{\mathcal{A}_h}}(0,K_t))=\infty$, ensures that the indistinguishability of model privacy holds for any two adjacent model parameters $\mathbf{w}_t^{i,(1)}$ and $\mathbf{w}_t^{i,(2)}$ for $i\in\mathcal{S}_t$. Hence, to some extent, the concept of infinite diameter $\text{Diam}({\mathcal{I}_t^{\mathcal{A}_h}}(0,K_t))=\infty$ in Definition 1 is similar to ``any pair of adjacent datasets" in standard data-oriented DP notion. Major differences between these two privacy notions are:
  \begin{enumerate}
    \item [1)] Different privacy objects: The input of a randomized mechanism in differential privacy is dataset $\mathcal{D}$, while the privacy notion used in this paper defines the model parameters as the direct privacy object. Many methods including model-inversion and membership-inference attacks deal with extracting sensitive information from trained models. In certain neural networks with fully-connected layers, the input data to the network can be reconstructed uniquely from the network's output \cite{geiping2020inverting}. Hence, model-oriented privacy notions like Definition 1 and Definition 5 are also of importance; 
    \item [2)] Probability of indistinguishability: In differential privacy, the outputs of a randomized mechanism over any two adjacent datasets are nearly statistically identical. While this event is in a definite sense as a condition of privacy preservation in Definition 1. Moreover, due to the ingenious design of model splitting mechanism in handling the injected noise, the privacy-utility trade-off is avoided. Therefore, there is no parameter similar to a privacy budget in this privacy notion.
  \end{enumerate}
    Based on the above discussion, we know that a sufficient relationship whereby Definition 1 implies standard data-oriented differential privacy exists. The essential requirement is the existence of a ``neighboring" mapping between the parameter solutions corresponding to different datasets (e.g., the modification of a single data sample leads to at most a change of $\Delta$ in the model parameters with probability $1$).
\end{remark}

In the sequel, we show that the privacy of any local model is protected against any honest-but-curious adversary in $\mathcal{A}_h$.

\begin{theorem}\label{theorem1}
  Under MSP-FL, for any $i\in\mathcal{S}_{t}$, if the number of invisible submodels (i.e., $m_i$) is inaccessible to $\mathcal{A}_h$, and there exists at least one legitimate client $j$ that shares mutual trust with $i$ via the trusted authority, then any adversary in $\mathcal{A}_h$ is unable to infer $\mathbf{w}_{t}^i$ with any guaranteed accuracy.
\end{theorem}

\begin{proof}
Denote $\{\mathbf{w}_{t}^{l,\alpha}[0],\mathbf{w}_{t}^{l,\beta_n}[0],a_{l,\alpha\beta_n}[k],\alpha_l[k]|l\in\mathcal{S}_{t},n\in\{1,...,m_l\},k\in\{0,1,...,K_t-1\}\}$ as an arbitrary set consisting of sequences of weights and initial local submodels that satisfies $\mathcal{I}_{t}^{\mathcal{A}_h}(0,K_t)$. According to the splitting rule, it has $(1+m_i)\mathbf{w}_{t}^i={\mathbf{w}_{t}^{i,\alpha}[0]+\sum_{n=1}^{m_i}\mathbf{w}_{t}^{i,\beta_n}[0]}$, $\forall i\in\mathcal{S}_t$. Meanwhile, let $\overline{\mathbf{w}}_{t}^{i}=\mathbf{w}_{t}^{i}+\mathbf{e}$ where $\mathbf{e}\in ^{*}\mathbb{R}^d$ is an arbitrary hyperreal vector.

As one of the conditions of privacy preservation, we have proved in Lemma 2 that any adversary in $\mathcal{A}_h$ is unable to infer $\mathbf{w}_{t}^{i}$ by \eqref{z} based on the current accessible information. Then, to prove the privacy preservation for any selected client $i$ against the set $\mathcal{A}_h$, it suffices to show the insensitivity or indistinguishability of information. That is, under any different local model $\mathbf{\overline{w}}_{t}^i$, there exists a set of sequences $\{\overline{\mathbf{w}}_{t}^{l,\alpha}[k_1],\overline{\mathbf{w}}_{t}^{l,\beta_n}[k_1],\overline{a}_{l,\alpha\beta_n}[k_2],\overline{\alpha}_l[k_2],|{l\in\mathcal{S}_{t}},k_1\in\{0,1,...,K_t\},k_2\in\{0,1,...,K_t-1\},n\in\{1,...,m_l\}\}$ such that $\overline{\mathbf{w}}_{t}^i\in \mathcal{S}(\mathcal{I}_{t}^{\mathcal{A}_h}(0,K_t),i)$. In other words, the information set $\mathcal{I}_{t}^{\mathcal{A}_h}(0,K_t)$ accessible to honest-but-curious adversary set  $\mathcal{A}_h$ is the same under any different $\overline{\mathbf{w}}_{t}^i$.

We show that if there exists at least one legitimate client $j\notin \mathcal{A}_h$ that cooperates with client $i$ via the trusted authority, and by properly selecting initial values of submodels and weights, especially, $\{\overline{\mathbf{w}}_{t}^{i,\alpha}[0],\overline{\mathbf{w}}_{t}^{i,\beta_n}[0],\overline{a}_{i,\alpha\beta_n}[0],\overline{\alpha}_i[0]|n\in\{1,...,m_i\}\}$
 and $\{\overline{\mathbf{w}}_{t}^{j,\alpha}[0],\overline{\mathbf{w}}_{t}^{j,\beta_n}[0],\overline{a}_{j,\alpha\beta_n}[0],\overline{\alpha}_j[0]|n\in\{1,...,m_j\}\}$, then it has $\overline{\mathbf{w}}_{t}^i\in \mathcal{S}(\mathcal{I}_{t}^{\mathcal{A}_h}(0,K_t),i)$. Specifically, initial values of local models and submodels are given by:
 \begin{small}
\begin{eqnarray}\label{bar1}
   \overline{\mathbf{w}}_{t}^{i}\hspace{-7pt}&=&\hspace{-7pt}\mathbf{w}_{t}^i+\mathbf{e}; \; \overline{\mathbf{w}}_{t}^{j}=\mathbf{w}_{t}^{j}-\mathbf{e}, \nonumber\\
    \overline{\mathbf{w}}_{t}^{i,\alpha}[0]\hspace{-7pt}&=&\hspace{-7pt}\mathbf{w}_{t}^{i,\alpha}[0];\;   \overline{\mathbf{w}}_{t}^{i,\beta_n}[0]={\mathbf{w}}_{t}^{i,\beta_n}[0], n\in\{1,...,m_i\}\backslash p, \nonumber\\
   \overline{\mathbf{w}}_{t}^{j,\alpha}[0]\hspace{-7pt}&=&\hspace{-7pt}\mathbf{w}_{t}^{j,\alpha}[0];\;   \overline{\mathbf{w}}_{t}^{j,\beta_n}[0]={\mathbf{w}}_{t}^{j,\beta_n}[0], n\in\{1,...,m_j\}\backslash q, \nonumber\\
   \overline{\mathbf{w}}_{t}^{i,\beta_p}[0]\hspace{-7pt}&=&\hspace{-7pt} (1+m_i)\overline{\mathbf{w}}_{t}^i-\overline{\mathbf{w}}_{t}^{i,\alpha}[0]-\sum_{n\in\{1,...,m_i\}\backslash p} \overline{\mathbf{w}}_{t}^{i,\beta_n}[0],\nonumber\\
   \overline{\mathbf{w}}_{t}^{j,\beta_q}[0]\hspace{-7pt}&=&\hspace{-7pt} (1+m_j)\overline{\mathbf{w}}_{t}^j-\overline{\mathbf{w}}_{t}^{j,\alpha}[0]-\sum_{n\in\{1,...,m_j\}\backslash q} \overline{\mathbf{w}}_{t}^{j,\beta_n}[0],\nonumber\\
    \overline{\mathbf{w}}_{t}^l\hspace{-7pt}&=&\hspace{-7pt}\mathbf{w}_{t}^l;\; \overline{\mathbf{w}}_{t}^{l,\alpha}[0]=\mathbf{w}_{t}^{l,\alpha}[0];\; \overline{\mathbf{w}}_{t}^{l,\beta_n}[0]=\mathbf{w}_{t}^{l,\beta_n}[0],\nonumber\\ \hspace{-7pt}&&\hspace{-7pt} \qquad\qquad\qquad\forall l\in\mathcal{S}_{t}\backslash \{i,j\}, n\in\{1,...,m_l\},
\end{eqnarray}
\end{small}%
and the settings of weights are as follows:
\begin{eqnarray*}
  \overline{a}_{i,\alpha\beta_p}[0] \hspace{-7pt}&=&\hspace{-7pt} \frac{-(1+m_i)\mathbf{e}+a_{i,\alpha\beta_p}[0](\mathbf{w}_{t}^{i,\alpha}[0]-\mathbf{w}_{t}^{i,\beta_p}[0])}{\overline{\mathbf{w}}_{t}^{i,\alpha}[0]-\overline{\mathbf{w}}_{t}^{i,\beta_p}[0]}  ,\nonumber\\
  \overline{a}_{i,\alpha\beta_n}[0] \hspace{-7pt}&=&\hspace{-7pt} {a}_{i,\alpha\beta_n}[0], \quad \forall n\in\{1,...,m_i\}\backslash p,\nonumber\\
  \overline{a}_{j,\alpha\beta_q}[0] \hspace{-7pt}&=&\hspace{-7pt} \frac{(1+m_j)\mathbf{e}+{a}_{j,\alpha\beta_q}[0](\mathbf{w}_{t}^{j,\alpha}[0]-\mathbf{w}_{t}^{j,\beta_q}[0])}{\overline{\mathbf{w}}_{t}^{j,\alpha}[0]-\overline{\mathbf{w}}_{t}^{j,\beta_q}[0]}  ,\nonumber\\
  \overline{a}_{j,\alpha\beta_n}[0] \hspace{-7pt}&=&\hspace{-7pt} {a}_{j,\alpha\beta_n}[0], \quad \forall n\in\{1,...,m_j\}\backslash q,\nonumber\\
   \overline{\alpha}_i[0]\hspace{-7pt}&=&\hspace{-7pt}\frac{(1+m_i)\mathbf{e}+{\epsilon}{\alpha_i}(\mathbf{w}_{t}^{j,\alpha}[0]-\mathbf{w}_{t}^{i,\alpha}[0])}{\epsilon(\overline{\mathbf{w}}_{t}^{j,\alpha}[0]-\overline{\mathbf{w}}_{t}^{i,\alpha}[0])},\nonumber\\
     \overline{\alpha}_j[0]\hspace{-7pt}&=&\hspace{-7pt}\frac{-(1+m_j)\mathbf{e}+{\epsilon}{\alpha_j}(\mathbf{w}_{t}^{i,\alpha}[0]-\mathbf{w}_{t}^{j,\alpha}[0])}{\epsilon (\overline{\mathbf{w}}_{t}^{i,\alpha}[0]-\overline{\mathbf{w}}_{t}^{j,\alpha}[0])},
\end{eqnarray*}
\begin{eqnarray}\label{bar2}
   \overline{a}_{i,\alpha\beta_n}[k]\hspace{-7pt}&=&\hspace{-7pt}{a}_{i,\alpha\beta_n}[k],\quad \forall n\in\{1,...,m_i\}, k\in\{1,...,K_t-1\},\nonumber\\
   \overline{a}_{j,\alpha\beta_n}[k]\hspace{-7pt}&=&\hspace{-7pt}{a}_{j,\alpha\beta_n}[k],\quad \forall n\in\{1,...,m_j\}, k\in\{1,...,K_t-1\},\nonumber\\
  \overline{a}_{l,\alpha\beta_n}[k]\hspace{-7pt}&=&\hspace{-7pt}{a}_{l,\alpha\beta_n}[k],\quad\forall l\in\mathcal{S}_{t}\backslash\{i,j\}, n\in\{1,...,m_l\},\nonumber\\
  \hspace{-7pt}&&\hspace{-7pt}\qquad\qquad\qquad k\in\{0,...,K_t-1\},\nonumber\\
  \overline{\alpha}_i[k]\hspace{-7pt}&=&\hspace{-7pt}{\alpha}_i[k]=\overline{\alpha}_j[k]={\alpha}_j[k]=\frac{1}{M},\quad k\in\{1,...,K_t-1\},\nonumber\\
   \overline{\alpha}_l[k]\hspace{-7pt}&=&\hspace{-7pt}{\alpha}_l[k]=\frac{1}{M},\,\,\,\,\forall l\in\mathcal{S}_{t}\backslash\{i,j\}, k\in\{0,...,K_t\hspace{-1pt}-\hspace{-1pt}1\}.\nonumber\\
   &&
\end{eqnarray}

Under \eqref{bar1} and \eqref{bar2}, it can be verified with some calculations that the information set accessible to $\mathcal{A}_h$ is the same as $\mathcal{I}_{t}^{\mathcal{A}_h}(0,K_t)$. Meanwhile, it implies that $\overline{\mathbf{w}}_{t}^i=\mathbf{w}_{t}^i+\mathbf{e}\in\mathcal{S}(\mathcal{I}_{t}^{\mathcal{A}_h}(0,K_t),i)$, which further yields
\begin{equation*}
\begin{aligned}
  \text{Diam}(\mathcal{I}_{t}^{\mathcal{A}_h}(0,K_t))&\geq\sup\|\overline{\mathbf{w}}_{t}^i-\mathbf{w}_{t}^i\|=\sup\limits_{\mathbf{e}\in^{*}\mathbb{R}^d}\|\mathbf{e}\|=\infty.
\end{aligned}
\end{equation*}

By Definition 1, one has that the privacy of the local model of any client $i\in\mathcal{S}_{t}$ is insensitive to any honest-but-curious adversary in $\mathcal{A}_h$, if the difference between $\mathbf{w}_{t}^i$ and $\overline{\mathbf{w}}_{t}^i$ which is denoted by $\mathbf{e}$ can be transmitted via the trusted authority to $j$, and both initial visible submodels of client $i$ and $j$ are transmitted via the trusted authority to each other (i.e., the trust between client $i$ and $j$ is mutual via the trusted authority). Equipped with Lemma 2 which indicates the nondeterminacy of current local model under MSP-FL, the insensitivity of current information set accessible to $\mathcal{A}_h$ shows that the privacy of $\mathbf{w}_{t}^i$, $\forall i\in\mathcal{S}_{t}$, is preserved.
\end{proof}

The key idea of the proof presented above can be interpreted as follows: Firstly, the first equation in \eqref{bar1} ensures that the sum of local models, denoted as $\sum_{i\in\mathcal{S}_{t}}\overline{\mathbf{w}}_{t}^i$, remains unchanged and equals to the original value $\sum_{i\in\mathcal{S}_{t}}\mathbf{w}_{t}^i$. Then, by maintaining the initial visible part invariant, any changes in local models are transferred to the initial invisible part, represented by $\overline{\mathbf{w}}_{t}^{i,\beta_p}[0]$ and $\overline{\mathbf{w}}_{t}^{j,\beta_q}[0]$. Simultaneously, the weight settings in \eqref{bar2} completely compensate for these changes from initial states. Consequently, the subsequent information set $\mathcal{I}_{t}^{\mathcal{A}_h}(0,K_t)$ accessible to $\mathcal{A}_h$ remains unchanged.
\begin{remark}
  It shall be pointed out that the coefficient of aggregation in \eqref{aggregation3} is still $1/M$, even though $\overline{\alpha}_i[0]$ and $\overline{\alpha}_j[0]$ are changed in the weight mechanism \eqref{bar2} for the local update of client $i$ and $j$, respectively.
\end{remark}
\begin{remark}
  One can observe that the difference of local model at client $i$ denoted by $\mathbf{e}$ together with $\mathbf{w}_t^{i,\alpha}[0]$ is transmitted via the trusted authority to legitimate client $j$ to reconstruct a new set $\{\overline{\mathbf{w}}_{t}^{j,\alpha}[k],\overline{\mathbf{w}}_{t}^{j,\beta_n}[k],\overline{\alpha}_j[k],\overline{a}_{j,\alpha\beta_n}[k]|n\in\{1,...,m_j\}\}$. Nonetheless, it is impossible to leakage the information of local model, since the other transmitted information between clients and the server involves only visible part.
\end{remark}

%
%
%

\subsection{The Sufficient Condition for Model Splitting Rule to Obtain Differential Privacy Noise}
One may wonder how much noise resulting from model splitting rule is equivalent to differential privacy noise when assigning initial values for submodels? For clarity of exposition, we take Laplacian mechanism to achieve $\epsilon_{DP}$-differential privacy for illustration. Let $Laplace(\mu,b)$ be a Laplace distributed noise with location $\mu$ and scale $b$.
Main result is presented in the following proposition. It shall be pointed out that, the discussed sufficient condition on the splitting rule does not conflict with the derivation of Lemma 8 (specifically, the equation \eqref{lemma8_2}), because $\mathbb{E}_{SS}\{\mathbf{w}_t^{i,\alpha}[0]\}=\mathbf{w}_t^i$ still holds under the random selection scheme for $\mathbf{w}_t^{i,\alpha}[0]$ over $Laplace(\mathbf{w}_t^i,\Delta f/\epsilon_{DP})$. Similarly, the noise $\xi_{t,a}^i$ can be Gaussian noise to achieve $(\epsilon_{DP},\delta)$-DP.

     \begin{proposition}
       For any function $f: \mathbb{D}\rightarrow \mathbb{R}^d$ where $\mathbb{D}$ is a domain of datasets $\mathcal{D}$, if the initial value of $\mathbf{w}_t^{i,\alpha}[0]$ is randomly drawn as the distribution $Laplace(\mathbf{w}_t^i,\Delta f/\epsilon_{DP})$ where $\Delta f=\max_{\mathcal{D}_i^{(1)},\mathcal{D}_i^{(2)}}\|f(\mathcal{D}_i^{(1)})-f(\mathcal{D}_i^{(2)})\|_1$, then the injected noise $\xi_{t,a}^{i}$ resulting from stochastic splitting is equivalent to the random noise injected via data-oriented $\epsilon_{DP}$-DP mechanism.
     \end{proposition}
         \begin{proof}
            Laplace mechanism is given as follows:
            \begin{definition}
              For any function $f: \mathbb{D}\rightarrow \mathbb{R}^d$ where $\mathbb{D}$ is a domain of datasets $\mathcal{D}$, the Laplace mechanism defined as
              \begin{equation}\label{Lap}
                \mathcal{M}^{Laplace}(f(x),\epsilon_{DP})=f(x)+[\varsigma_1,...,\varsigma_d]^{\top},
              \end{equation}
              is $\epsilon_{DP}$-DP. In \eqref{Lap}, $\varsigma_i\overset{i.i.d.}{\sim}$ Laplace($0,\Delta f/\epsilon_{DP}$), i.e., its probability density function at a realization $x\in\mathbb{R}$ is given by
              \begin{equation}\label{pdf}
                Laplace(x|0,\Delta f/\epsilon_{DP})=\frac{1}{2\Delta f/\epsilon_{DP}}\exp \left( -\frac{|x|}{\Delta f/\epsilon_{DP}} \right),
              \end{equation}
              where $\Delta f\overset{\triangle}{=}\max\nolimits_{x,y\in\mathcal{D}}\|f(x)-f(y)\|_1$.\vspace{5pt}
            \end{definition}

            Recalling that the initial value of each visible submodel, i.e., $\mathbf{w}_t^{i,\alpha}[0]$, can be viewed as a noise-distorted variant of $\mathbf{w}_t^i$, hence, we rewrite $\mathbf{w}_t^{i,\alpha}[0]$ as
            \begin{equation}\label{R.7}
              \mathbf{w}_t^{i,\alpha}[0]=\mathbf{w}_t^{i}+\xi_{t,a}^i,
            \end{equation}
             wherein, the additive noise $\xi_{t,a}^i$ is determined by the splitting rule when assigning initial values for visible submodels. Obviously, the condition for the noise $\xi_{t,a}^i$ to satisfy the Laplace privacy-preserving noise as defined in Definition 4 is that
             \begin{equation}
               \mathbf{w}_t^{i,\alpha}[0]\sim Laplace(\mathbf{w}_t^i,\Delta f/\epsilon_{DP}).
             \end{equation}
            That is, the initial value of $\mathbf{w}_t^{i,\alpha}[0]$ is randomly drawn as the Laplace distribution $Laplace(\mathbf{w}_t^i,\Delta f/\epsilon_{DP})$. It shall be pointed out that, this requirement on the splitting rule does not conflict with the derivation of Lemma 8 (specifically, the equation \eqref{lemma8_2}), because $\mathbb{E}_{SS}\{\mathbf{w}_t^{i,\alpha}[0]\}=\mathbf{w}_t^i$ still holds under the random selection scheme for $\mathbf{w}_t^{i,\alpha}[0]$ over $Laplace(\mathbf{w}_t^i,\Delta f/\epsilon_{DP})$.
            
            This completes the proof.
         \end{proof}

\subsection{Convergence Analysis}
To proceed, we analyze the convergence property of MSP-FL. For ease of presentation and analysis, the number of invisible local submodels is assumed to be $m_i=1$, $\forall i\in\mathcal{S}_t$. Let $a_{\alpha\beta_1}^{\max}=\max\{a_{i,\alpha\beta_1}[k]|{i\in\mathcal{S}_{t},k\in\{0,...,K_t-1\}}\}$.

The following customary assumption is listed for analysis.
\begin{assumption}
 (Uniformly bounded local models) There exists a vector $\mathbf{w}_{\max}\in\mathbb{R}^d$ such that after multi steps of stochastic gradient descent, each local model is uniformly bounded, i.e.,
    $\|\mathbf{w}_{t}^i\|\leq \|\mathbf{w}_{\max}\|$ $\forall i\in\mathcal{N}$ and $t=1,2,...,T$.
\end{assumption}

Now we give the convergence result of MSP-FL.

\begin{theorem}
  Let $\vartheta=\max\{\frac{8L}{\mu},E\}-1$, $\epsilon\in(0,\frac{M}{M-1})$, $K_t=\lceil\log_\lambda^{2/\mu(\vartheta+t)}\rceil$ and ${ {a_{\alpha\beta_1}^{\max}}\in (0, \frac{\lambda_{\min,\mathbf{U}}}{1+\lambda_{\min,\mathbf{U}}})}$ where $\mathbf{U}$ is defined in \eqref{notation4}. Choose the diminishing learning rate as $\eta_t=\frac{2}{\mu(\vartheta+t)}$. Suppose that Assumptions 1-2 hold, then MSP-FL has the following convergence property:\par
  \vspace{-10pt}
  \begin{small}
  \begin{eqnarray}\label{theorem3}
   \hspace{-10pt}&&\hspace{-9pt} \mathbb{E}\left\{F(\overline{\mathbf{w}}_{t})\right\}-F^* \hspace{-1pt}\leq \hspace{-1pt} \frac{8L}{\mu(\vartheta\hspace{-1pt}+\hspace{-2pt}t)}\left(\hspace{-1pt}\frac{2D_2}{\mu}\hspace{-1pt}+\hspace{-1pt}\frac{8L\hspace{-2pt}+\hspace{-2pt}\mu E}{2}\mathbb{E}\left\{\|\overline{\mathbf{w}}_0\hspace{-1pt}-\hspace{-1pt}\mathbf{w}^*\|^2\right\}\hspace{-2pt}\right)\nonumber\\
    \hspace{-10pt}&&\hspace{-9pt}
  \end{eqnarray}
  \end{small}%
where $D_2=\frac{(2\epsilon^2-4\epsilon+8)}{3}C^2\|\mathbf{w}_{\max}\|^2+\sum_{i=1}^{N}p_i^2\sigma_i+6L\Gamma+8(E-1)^2G+\frac{4}{M}E^2G$.
\end{theorem}

\begin{proof}
  The proof is given in Appendix A.
\end{proof}

In Theorem 2, we derive a bound on the objective value of the model learned via MSP-FL and that of the optimal model. The behavior of the bound in \eqref{theorem3} implies the asymptotic convergence of MSP-FL with a rate of $\mathcal{O}(1/t)$. That is, as the number of learning rounds $t$ progresses, the learned model converges to the optimal one with a difference decaying as $1/t$. The convergence rate is at the same order of FedAvg in \cite{fedavg} and FL with privacy and communication efficiency requirements in \cite{NIPS2022_SoteriaFL:A-Unified-Framework-for-Private-Federated,joint}, and it is better than the $\mathcal{O}(\log(t)/\sqrt{t})$ rate observed in \cite{Shuffled-Model-of-Differential-Privacy-in-Federated-Learning_2021_pmlr_Girgis}. One can observe from \eqref{theorem3} that, for privacy enhancement considerations, the model splitting procedure indeed yields an additive term in $D_2$. Specifically, $D_2$ contains several additive terms, wherein, the first term $\frac{(2\epsilon^2-4\epsilon+8)}{3}C^2\|\mathbf{w}_{\max}\|^2$ is a derivative part resulting from privacy enhancement via model splitting mechanism, which is dictated by the splitting factor $\varepsilon$ and the upper bound of model parameters. It is worth noting that $\varepsilon\in[0,\frac{1+m_i}{2})$, and the splitting factor $\varepsilon$ affects the selection of initial values of the visible submodels after local training process of each learning round.
Therewith, it implies that, on one hand, a smaller splitting factor brings about a smaller impact of model splitting mechanism on the learning profile if the number of iterations in \eqref{decomposition1} is finite and fixed. On the other hand, a smaller splitting factor means that the initial random values selected for visible submodels are closer to the true model obtained at learning round $t$. $\sum_{i=1}^{N}p_i^2\sigma_i$, $6L\Gamma$, $8(E-1)^2G$ and $\frac{4}{M}E^2G$ in $D_2$ reveal the impacts of stochastic gradients, heterogeneity of data distribution and model parameter, and partial clients participation on the learning profile, respectively. Detailed analysis about these terms can be found in the earlier research \cite{niid}. To further mitigate the impact of these heterogeneities, e.g., `client-drift', please refer to methods including SCAFFOLD in \cite{SCAFFOLD}.


\section{Model-Splitting Privacy-preserving with Dynamic Quantization Federated Learning}
This section considers both privacy enhancement and quantization in FL, and proposes a scheme named MSPDQ-FL. Compared to vanilla FedAvg, MSP-FL results in totally $\frac{\sum_{t=1}^{T} (K_t+1)}{T}$ times communication burden for privacy enhancement. To tackle this challenge, we utilize tools from quantization theory to reduce the usage of communication bandwidth. Simultaneously, the distortion induced by the compression of transmitted submodel is enforced to vanish as the dynamic quantization interval shrinks. Note that the dynamic quantization mechanism can be independently designed in such a joint scheme. The reason is that MSP-FL actually ensures exact convergence guarantee in expectation as long as the communication round $K_t$ after each learning round $t$ is larger than a dynamic lower bound, whereas, the communication round is independent of the design for quantization. The details of MSPDQ-FL are presented below.
\subsection{The Proposed Scheme: MSPDQ-FL}
For the split submodel $\mathbf{w}_{t}^{i,\alpha}[k]$ of the $i$-th client during the $k$-th communication round, it is assumed that any element of $\mathbf{w}_{t}^{i,\alpha}[k]$ is bounded in a closed interval, i.e., $w_{t}^{i,j,\alpha}[k]\in [{w}_{\min}^{i,\alpha},{w}_{\max}^{i,\alpha}]$, $\forall j\in \{1,2,...,d\}$, $\forall
k$. Denote the quantization interval as $\mathcal{R}_i\triangleq [{w}_{\min}^{i,\alpha}\mathbf{1},{w}_{\max}^{i,\alpha}\mathbf{1}]$. Given a quantization level ${l} $ with the quantization bits $B=\lceil\log_{2}(l)\rceil$, the knobs uniformly spaced in the interval $\mathcal{R}_i$ are denoted as $\{c_0,c_1,...,c_{l-1}\}$, where
$c_\tau=w_{\min}^{i,\alpha}+\tau \frac{w_{\max}^{i,\alpha}-w_{\min}^{i,\alpha}}{l-1}, \forall \tau=0,1,...,l-1$.

Suppose that the $j$-th element of submodel $\mathbf{w}_{t}^{i,\alpha}[k]$ is located in $[c_\tau,c_{\tau+1})$. Then, $w_{t}^{i,j,\alpha}[k]$ is quantized by the following stochastic quantizer:
\begin{equation}\label{sq}
    \begin{aligned}
       \hspace{-5.5pt} \mathcal{Q}(w_{t}^{i,j,\alpha}[k])=\left\{
          \begin{aligned}
            &\text{sign}(w_{t}^{i,j,\alpha}[k]) c_\tau, \quad\text{w.p.}\frac{c_{\tau+1}-|w_{t}^{i,j,\alpha}[k]|}{c_{\tau+1}-c_\tau} \\
            &\text{sign}(w_{t}^{i,j,\alpha}[k]) c_{\tau+1},  \text{w.p.} \frac{|w_{t}^{i,j,\alpha}[k]|-c_\tau}{c_{\tau+1}-c_\tau}
          \end{aligned}\right.
    \end{aligned}
\end{equation}
where, for brevity, we denote $q_{t}^{i,j}[k]= \mathcal{Q}(w_{t}^{i,j,\alpha}[k])$ and $\mathbf{q}_t^{i}[k]=[{q}_t^{i,1}[k],{q}_t^{i,2}[k],...,{q}_t^{i,d}[k]]^{\top}$. ‘w.p.’ represents ‘with probability’. Additionally, one has the quantization error given by
\begin{equation}\label{qe}
  \|\mathbf{\Delta}_t^{i}[k]\|=\|\mathbf{q}_t^{i}[k]-\mathbf{w}_t^{i,\alpha}[k]\|\leq\frac{\sqrt{d}(w_{\max}^{i,\alpha}-w_{\min}^{i,\alpha})}{l-1}.
\end{equation}

It can be observed from \eqref{qe} that the quantization error is directly relevant to the magnitude of the quantization interval $[w_{\min}^{i,\alpha},w_{\max}^{i,\alpha}]$, when the quantized level $l$ is finite and fixed. In this work, our motivation that follows with \cite{adapq} is to adaptively refine the quantization interval for such a stochastic quantizer as a dynamic one at each communication round, i.e., $\mathcal{R}_t^i[k]$. The magnitude of $\mathcal{R}_t^i[k]$ is designed to shrink as $k$ goes. Specifically, given the current time instant as $k$, the quantization interval for the next communication round is determined by the current information as
\begin{eqnarray}\label{interval}
 \hspace{-13pt} \mathcal{R}_{t}^i[k+1]\hspace{-9pt}&=&\hspace{-9pt}\left[\mathbf{q}_{t}^i[k]\hspace{-1pt}-\hspace{-1pt}\frac{\pi_t a_{\alpha\beta_1}^{\max}[k]}{2}\mathbf{1},\mathbf{q}_{t}^i[k]\hspace{-1pt}+\hspace{-1pt}\frac{\pi_t a_{\alpha\beta_1}^{\max}[k]}{2}\mathbf{1}\right],
\end{eqnarray}
where $\pi_t$ is a parameter to be designed that governs the interval bounds during the communication rounds $k\in\{0,...,K_t\}$. $a_{\alpha\beta_1}^{\max}[k]\overset{\triangle}{=}\max\{a_{i,\alpha\beta_1}[k]|i\in\mathcal{S}_t\}$ mainly controls the rate at which the dynamic quantization interval shrinks. As such, the variance of quantization error decays asymptotically as the magnitude of quantization interval shrinks. We denote such stochastic dynamic quantizer as $\mathcal{Q}_{k}^i(\cdot)$.

Equipped with the aforementioned stochastic dynamic quantizer for model aggregation, the updating rule for submodels in MSPDQ-FL is written as follows:
\begin{subequations}\label{decomposition2}
\begin{eqnarray}
      \hspace{-12pt}\mathbf{w}_{t}^{i,\alpha}[k+1] \hspace{-7pt}&=&\hspace{-7pt} \mathbf{w}_{t}^{i,\alpha}[k]+\epsilon(\mathbf{w}_{t}[k]-\mathbf{q}_{t}^{i}[k])\nonumber\\
     &&\hspace{-7pt}+\sum_{n=1}^{m_i}a_{i,\alpha\beta_n}[k]\left(\mathbf{w}_{t}^{i,\beta_n}[k]-\mathbf{w}_{t}^{i,\alpha}[k]\right),\label{decomposition2a}\\
    \hspace{-12pt} \mathbf{w}_{t}^{i,\beta_n}[k+1] \hspace{-7pt}&=& \hspace{-7pt} \mathbf{w}_{t}^{i,\beta_n}[k]+a_{i,\alpha\beta_n}[k](\mathbf{w}_{t}^{i,\alpha}[k]-\mathbf{w}_{t}^{i,\beta_n}[k]),\nonumber\\
     \hspace{-7pt}&&\hspace{-7pt}\qquad\qquad\qquad\qquad n=1,2,...,m_i, \label{decomposition2b}
\end{eqnarray}
\end{subequations}
where the aggregated model at the server is given by
\begin{eqnarray}\label{aggregation4}
   \mathbf{w}_{t}[k+1]\hspace{-7pt}&=&\hspace{-7pt}\sum_{i\in S_{t}}\frac{1}{M}\mathbf{q}_{t}^{i}[k+1].
\end{eqnarray}
Detailed procedure of MSPDQ-FL is shown in Algorithm 1.
\begin{algorithm}[t]
  \caption{MSPDQ-FL}
  \label{DPPD}
  \begin{algorithmic}[1]
    \State
    Initialize global model $\mathbf{w}_0$, $\epsilon\in(0,\frac{M}{M-1})$, $\tilde{\epsilon}=\max\{1,\epsilon\}$, $\pi_{t}=\frac{8(\epsilon+\tilde{\epsilon}+\tilde{\epsilon}\widetilde{W}_{t,k}^{\max})}{1-\lambda_{2,\mathbf{U}}}$ and $a_{i,\alpha\beta_n}[k]=\frac{\gamma_i}{k+1}$ $\forall n\in\{1,2,...,m_i\}$ and $\gamma_i \in (0, \frac{\lambda_{\min,\mathbf{U}}}{1+\lambda_{\min,\mathbf{U}}})$;
    \For{$t=1,2,...,T$}
        \State Server selects $M$ clients randomly with replacement \Statex\quad\, by
       probabilities $\{p_1,p_2,...,p_N\}$, denoted by $\mathcal{S}_{t}$;
        \State Server broadcasts $\mathbf{w}_{t-1}$;
    \For {client $i\in \mathcal{S}_{t}$} \textbf{in parallel}
        \State $\mathbf{w}_{t,0}^i\leftarrow \mathbf{w}_{t-1}$;\\
     \Statex$\qquad \quad K_t=\max\left\{\left\lceil\log_{\lambda}^{\frac{2}{\mu(\vartheta+t)}}\right\rceil, \left\lceil\frac{\mu(\vartheta+t)}{2}\right\rceil \right\}$;
            \State Client updates local model as in \eqref{decomposition2};
    \EndFor  \State\textbf{end for}
            \State Each client $i$ sets $\mathbf{w}_{t}^i\leftarrow\mathbf{w}_{t,E}^i$;
            \State Each client $i$ randomly splits $\mathbf{w}_{t}^{i,\alpha}[0]$ from a uniform
            \Statex \quad\, distribution on the interval $[\epsilon\mathbf{w}_{t}^i,(1+m_i-\epsilon)\mathbf{w}_{t}^i]$ with
            \Statex \quad\, $\mathbf{w}_{t}^{i,\alpha}[0]+\sum_{n=1}^{m_i}\mathbf{w}_{t}^{i,\beta_n}[0]=(1+m_i)\mathbf{w}_{t}^{i}[0]$, and
            \Statex \quad \,
             $\mathbf{w}_{t}^{i,\alpha}[0]=\mathbf{w}_{t}^{j,\alpha}[0]$, $\forall \{i,j\}\in\mathcal{S}_{t}$;
        \State Each client $i\in\mathcal{S}_{t}$ initializes
             \begin{itemize}
          \item Uniformly Divide $[w_{\min}^{i,\alpha}\mathbf{1},w_{\max}^{i,\alpha}\mathbf{1}]$ into $l^d$ rectangular bins as $w_{\min}^{i,\alpha}\mathbf{1}={\tau}_1\leq\cdots\leq\tau_{l^d}\leq w_{\max}^{i,\alpha}\mathbf{1}$;
          \item Let $\mathbf{w}_{t}^{i,\alpha}[0]=\mathbf{q}_{t}^i[0]=\tau_m$ for some $m\in[0,l^d]$. Compute $\mathbf{q}_{t}^{i|b}[0]$ using the quantization scheme $\mathcal{Q}_0(\mathbf{w}_{t}^{i,\alpha}[0])$ over $[w_{\min}^{i,\alpha}\mathbf{1},w_{\max}^{i,\alpha}\mathbf{1}]$;
          \item A sequence of positive and non-increasing step sizes $\{a_{i,\alpha\beta_n}[k]|n=1,2,...,m_i\}$;
        \end{itemize}
        \State Each client $i\in\mathcal{S}_{t}$ uploads $\mathbf{q}_{t}^{i|b}[0]$;
        \State Server receives $\mathbf{q}_{t}^{i|b}[0]$ from clients, recovers $\mathbf{q}_{t}^i[0]$
        \Statex \;\;\,\,\,\, from $\mathbf{q}_{t}^{i|b}[0]$ over $\mathcal{R}_{t}^i[0]=[w_{\min}^{i,\alpha}\mathbf{1},w_{\max}^{i,\alpha}\mathbf{1}]$, performs
        \Statex \;\,\,\,\,\, aggregation as \eqref{aggregation4} and broadcasts $\mathbf{w}_{t}[0]$;
        \For {$k=0,1,...,K_t-1$}
            \For {client $i\in \mathcal{S}_{t}$} \textbf{in parallel}
                \State Client $i$ executes \eqref{decomposition2} to update $\mathbf{w}_{t}^{i,\alpha}[k+1]$;
                \State Client $i$ computes $\mathbf{q}_{t}^{i}[k+1]$ and $\mathbf{q}_{t}^{i|b}[k+1]$ by
                \Statex \qquad \qquad\; $\mathcal{Q}_{k+1}^i(\cdot)$ over \eqref{interval};
                \State Client uploads $\mathbf{q}_{t}^{i|b}[k+1]$;
            \EndFor \State\textbf{end for}
            \State Server receives $\mathbf{q}_{t}^{i|b}[k+1]$ from clients, recovers \Statex \qquad\quad$\mathbf{q}_{t}^i[k+1]$ over $\mathcal{R}_{t+1}^i(k+1)$, performs aggregation
            \Statex \qquad\quad   as \eqref{aggregation4} and broadcasts $\mathbf{w}_{t}[k+1]$;
        \EndFor \State\textbf{end for}
        \State Server finds $\mathbf{w}_{t}\leftarrow\mathbf{w}_{t}[K_t]$;
    \EndFor  \State \textbf{end for}
  \end{algorithmic}
\end{algorithm}

For the execution of MSPDQ-FL, we give the following initial settings: 1) $\mathbf{w}_{t}^{i,\alpha}[0]=\mathbf{q}_{t}^i[0]=c_{\tau}$ for some arbitrarily index $\tau\in\{0,...,l-1\}$; 2) Each initially split visible part submodel is randomly generated from a uniform distribution on the interval $[\varepsilon \mathbf{w}_{t}^i,(1+m_i-\varepsilon) \mathbf{w}_{t}^i]$ satisfying $\mathbf{w}_{t}^{i,\alpha}[0]+\sum_{n=1}^{m_i}\mathbf{w}_{t}^{i,\beta_n}[0]=(1+m_i)\mathbf{w}_{t}^{i}$ for all $i\in\mathcal{S}_t$ and $t=1,2,...,T$. Moreover, it requires that $\mathbf{w}_{t}^{i,\alpha}[0]=\mathbf{w}_{t}^{j,\alpha}[0]$, $\forall \{i,j\}\in\mathcal{S}_{t}$.\footnote{The condition $\mathbf{w}_{t}^{i,\alpha}[0]=\mathbf{w}_{t}^{j,\alpha}[0]$, $\forall \{i,j\}\in\mathcal{S}_{t}$, can be achieved by the server to broadcast any uploaded $\mathbf{w}_{t}^{i,\alpha}[0]$ after each learning round.}

\begin{remark}
    The intuition of such a design for shrinking quantization interval is that, as the local updating rule in \eqref{decomposition2} proceeds after $t$-th learning round, client $i$ becomes more certain about where the current submodel $\mathbf{w}_{t}^{i,\alpha}[k]$ is and the subsequent submodel $\mathbf{w}_{t}^{i,\alpha}[k+1]$ will be on the whole quantization interval. Based on the current information, it determines a shrinking quantization interval $\mathcal{R}_{t}^i[k+1]$ for the next communication round.
\end{remark}

\subsection{Privacy and Convergence Analysis}
In this subsection, we present the analysis of MSPDQ-FL for its privacy preservation and convergence performance.

\begin{proposition}
  Under MSPDQ-FL, for any $i\in\mathcal{S}_{t}$, any adversary in $\mathcal{A}_h$ is unable to infer $\mathbf{w}_{t}^i$ in the sense of expectation with any guaranteed accuracy, if the following conditions are satisfied: 1. At least one legitimate client $j$ shares mutual trust with $i$ via the trusted authority;
 2. The number of invisible submodels $m_i$ is inaccessible to adversaries in the set $\mathcal{A}_h$.
\end{proposition}

\begin{proof}
  Note that the quantizer $\mathcal{Q}_{k}^i(\cdot)$ is unbiased as stated in Lemma 11, hence, the proof directly follows from the line of derivation in the proof of Theorem 1.
\end{proof}

For ease of analysis, an additional assumption is listed below, and the number of invisible submodels of each selected client is assigned to be $m_i=1$ $\forall i\in\mathcal{S}_t$, similarly as in Section \uppercase\expandafter {\romannumeral3}.C.
\begin{assumption}
  (Adapted from Assumption 2 in \cite{adapq}) Let $\pi_{t}=\frac{8(\epsilon+\tilde{\epsilon}+\tilde{\epsilon}\widetilde{W}_{t,k}^{\max})}{1-\lambda_{2,\mathbf{U}}}$ where $\tilde{\epsilon}=\max\{1,\epsilon\}$ and ${\widetilde{W}_{t,k}^{\max}}=\max\{\|\mathbf{W}_{t}^\alpha[k_1]-\mathbf{W}_{t}^{\beta_1}[k_1]\||k_1\in\{0,...,k\}\}$. The number of bits $B$ of the communication bandwidth satisfies
  \begin{equation}\label{assumption7}
    B\leq\log_2^{\sqrt{Md}\pi_{t}+1}, \quad \forall t=1,2,...,T.
  \end{equation}
\end{assumption}

Our main result for the convergence analysis of MSPDQ-FL is stated as follows.
\begin{theorem}
Let $\vartheta$, $\epsilon$ and $\eta_t$ be defined in Theorem 2. In addition, let $K_t=\max\big\{\lceil\log_{\lambda}^{2/\mu(\vartheta+t)}\rceil, \lceil\frac{\mu(\vartheta+t)}{2}\rceil \big\}$ and the diminishing step size $a_{i,\alpha\beta_1}[k]$ for all $i\in\mathcal{S}_{t}$ be designed as $\frac{\gamma_i}{k+1}$, where ${\gamma_i \in (0, \frac{\lambda_{\min,\mathbf{U}}}{1+\lambda_{\min,\mathbf{U}}})}$. If Assumptions 1-3 hold, then the convergence property of MSPDQ-FL with fixed quantization level $l$ satisfies\par
\vspace{-11pt}
\begin{small}
  \begin{eqnarray}\label{theorem4jieguo}
   \hspace{-23pt}&&\hspace{-10pt} \mathbb{E}\left\{F(\overline{\mathbf{w}}_{t})\right\}-F^* \nonumber\\
    \hspace{-23pt}&&\hspace{-10pt} \leq \frac{8L}{\mu(\vartheta\hspace{-1pt}+\hspace{-1pt}t)}\hspace{-1pt}\left(\hspace{-1pt}\frac{2(D_2\hspace{-1pt}+\hspace{-1pt}D_3)}{\mu}\hspace{-1pt}+\hspace{-1pt}\frac{8L\hspace{-1pt}+\hspace{-1pt}\mu E}{2}\mathbb{E}\left\{\|\overline{\mathbf{w}}_0\hspace{-1pt}-\hspace{-1pt}\mathbf{w}^*\|^2\right\}\hspace{-2pt}\right)
  \end{eqnarray}
\end{small}%
where $D_3=\frac{d\gamma_{\max}^2\tilde{\pi}^2}{4M(2^B-1)^2}$, $\gamma_{\max}=\max\{\gamma_1,...,\gamma_N\}$ and $\tilde{\pi}=\frac{8(\epsilon+\tilde{\epsilon}+\tilde{\epsilon}\widetilde{W}^{\max})}{1-\lambda_{2,\mathbf{U}}}$ with $\widetilde{W}^{\max}=\max \{\widetilde{W}_{t,k}^{\max}|t=1,2,...,T\}$.
\end{theorem}

\begin{proof}
  The proof is given in Appendix B.
\end{proof}

\begin{remark}
    In the proofs of Theorem 2 and Theorem 3, Assumption 2 (namely, uniformly bounded local models) and Assumption 3 (including the condition of bounded difference between visible and invisible submodels) are used as part of sufficient conditions to guarantee the asymptotic convergence. Actually, both of these two conditions are reasonable since the communication resources are limited. Hence, it is impossible to transmit an infinitely large model or submodel.
\end{remark}

In the above, one can see that, for both MSP-FL and MSPDQ-FL, $K_t+1$ communication rounds are introduced after each $t$-th local training. Hence, in what follows, we study the communication complexity. As the communication complexity of MSP-FL follows similarly with that of MSPDQ-FL, we focus on the latter one. In the setting, the communication complexity is defined as the total number of uploads over the selected clients with the expected convergence accuracy $\varrho$.

\begin{proposition}
  Given the conditions the same as those in Theorem 3, the communication complexity of MSPDQ-FL, denoted as $\mathbb{C}(\varrho)$, is upper bounded by
    \begin{eqnarray}
     \mathbb{C}(\varrho)  \hspace{-7pt}&<&\hspace{-7pt}M\lceil \mathbb{I}(\varrho) \rceil\left( 1+\mu\vartheta+\frac{\mu}{2} (1+\lceil \mathbb{I}(\varrho) \rceil)\right),
    \end{eqnarray}
  where $\mathbb{I}(\varrho)$ is the learning round complexity given by $\mathbb{I}(\varrho)=\frac{\nu_2}{\varrho \mathbb{E}\{\|\overline{\mathbf{w}}_{0}-\mathbf{w}^*\|^2\}}-\vartheta$, $\nu_2=\max\big\{\frac{\rho^2(D_2+D_3)}{\rho\mu-1},(\vartheta+1)\mathbb{E}\{\left\|\overline{\mathbf{w}}_{0}-\mathbf{w}^*\|^2\right\}\big\}$.
\end{proposition}

\begin{proof}
By \eqref{56}, it has that\par
  \vspace{-9pt}
  \begin{small}
  \begin{eqnarray}
   \frac{\mathbb{E}\{\|\overline{\mathbf{w}}_{t}-\mathbf{w}^*\|^2\}}{\mathbb{E}\{\|\overline{\mathbf{w}}_{0}-\mathbf{w}^*\|^2\}} \hspace{-7pt}&\leq&\hspace{-7pt} \frac{\nu_2}{(\vartheta+t)\mathbb{E}\{\|\overline{\mathbf{w}}_{0}-\mathbf{w}^*\|^2\}}\leq\varrho\nonumber\\
  \Longrightarrow \qquad t \hspace{-7pt}&\geq&\hspace{-7pt} \frac{\nu_2}{\varrho \mathbb{E}\{\|\overline{\mathbf{w}}_{0}-\mathbf{w}^*\|^2\}}-\vartheta.
  \end{eqnarray}
  \end{small}%
Then, we conclude that the learning round complexity is $\mathbb{I}(\varrho)=\frac{\nu_2}{\varrho \mathbb{E}\{\|\overline{\mathbf{w}}_{0}-\mathbf{w}^*\|^2\}}-\vartheta$. Building on the learning round complexity and recalling that, for each $t$, each client $i\in\mathcal{S}_t$ inherently engages in $K_t+1$ communication rounds, hence, the total communication complexity $\mathbb{C}(\varrho)$ is given by\par
\vspace{-10pt}
\begin{small}
\begin{eqnarray}\label{p2_22}
   \hspace{-19pt}\mathbb{C}(\hspace{-1pt}\varrho\hspace{-1pt}) \hspace{-9pt}&=&\hspace{-10pt} \sum_{i\in\mathcal{N}}{\hspace{-1pt}\text{Communication \hspace{-4pt} rounds \hspace{-4pt} of \hspace{-4pt} client $i$}}\hspace{-1pt}=\hspace{-1pt}M\hspace{-3pt}\sum_{t=1}^{\lceil\mathbb{I}(\varrho)\rceil}(K_t\hspace{-1pt}+\hspace{-1pt}1).
\end{eqnarray}
\end{small}%
By substituting $K_t=\max\big\{\lceil\log_{\lambda}^{2/\mu(\vartheta+t)}\rceil, \lceil\frac{\mu(\vartheta+t)}{2}\rceil \big\}$ into \eqref{p2_22}, and using $K_t\leq \lceil \log_{\lambda}^{2/\mu(\vartheta+t)}+\frac{\mu(\vartheta+t)}{2} \rceil$ together with the basic inequality $\log(1+x)\leq x, \forall x>-1$, it has that\par
 \vspace{-12pt}
\begin{small}
\begin{eqnarray*}
  \mathbb{C}(\varrho) \hspace{-7pt}&=&\hspace{-7pt}M\sum\nolimits_{t=1}\nolimits^{\lceil\mathbb{I}(\varrho)\rceil}(\lceil\mu(\vartheta+t)-1\rceil+1)\nonumber\\
  \hspace{-7pt}&<&\hspace{-7pt}M\lceil \mathbb{I}(\varrho) \rceil\left( 1+\mu\vartheta+\frac{\mu}{2} (1+\lceil \mathbb{I}(\varrho) \rceil)\right) .
\end{eqnarray*}
\end{small}

This completes the proof.
\end{proof}

Here, we further investigate the privacy performance of MSPDQ-FL based on the preliminary results of the proposed quantizer presented in Lemma 10. Quantization introduces additional distortion into model parameters, even though the magnitude of distortion gradually weakens over time. That is, MSPDQ-FL achieves stronger privacy protection compared to MSP-FL. However, the composite privacy enhancement cannot be trivially analyzed in a probabilistic expression, because the partial dynamics of model splitting mechanism is not randomized and the privacy notion in this paper is not in a probabilistic expression. At this stage, we analyze the inherent privacy of the proposed quantizer by differential privacy notion.
\begin{proposition}
  Suppose that the difference between any two adjacent model parameters $\mathbf{w}_t^{i,\alpha,(1)}$ and $\mathbf{w}_t^{i,\alpha,(2)}$ is bounded by a positive constant $C_4$, i.e., $\|\mathbf{w}_t^{i,\alpha,(1)}-\mathbf{w}_t^{i,\alpha,(2)}\|_1\leq C_4$. If Assumption 3 holds, and the parameters of the proposed dynamic quantizer are set as in Lemma 10, then the proposed quantizer ensures at least $(0,\delta)$-differential privacy with $\delta=\min\left\{\frac{C_4(l-1)}{\pi_t a_{\alpha\beta_1}^{\max}[k]},\frac{l-1}{2^B-1}\right\}$ for each element of the quantization input $\mathbf{w}_t^{i,\alpha}[k]$ $\forall i\in\mathcal{S}_t$.
\end{proposition}
\begin{proof}
    Since the quantization intervals are uniformly partitioned, we only need to analyze an arbitrary quantization interval between $c_{\tau}$ and $c_{\tau+1}$. According to the mechanism of stochastic quantization as in \eqref{sq},
      one has that the quantized value of an element in $\mathbf{w}_t^{i,\alpha}[k]$ has different distributions. When $w_{t}^{i,j,\alpha}[k]\geq 0$,
    \begin{equation}\label{distribution1}
        \left\{
        \begin{aligned}
        Pr(\mathcal{Q}(w_{t}^{i,j,\alpha}[k]) = c_\tau \mid \mathbf{w}_t^{i,\alpha}) & = \frac{c_{\tau+1}-|w_{t}^{i,j,\alpha}[k]|}{c_{\tau+1}-c_\tau} \\
        Pr(\mathcal{Q}(w_{t}^{i,j,\alpha}[k]) = c_{\tau+1} \mid \mathbf{w}_t^{i,\alpha}) & = \frac{|w_{t}^{i,j,\alpha}[k]|-c_\tau}{c_{\tau+1}-c_\tau}
        \end{aligned}
        \right.,
    \end{equation}
and when $w_{t}^{i,j,\alpha}[k]< 0$,
    \begin{equation}\label{distribution2}
        \left\{
        \begin{aligned}
        Pr(\mathcal{Q}(w_{t}^{i,j,\alpha}[k]) = -c_\tau \mid \mathbf{w}_t^{i,\alpha}) & = \frac{c_{\tau+1}-|w_{t}^{i,j,\alpha}[k]|}{c_{\tau+1}-c_\tau} \\
        Pr(\mathcal{Q}(w_{t}^{i,j,\alpha}[k]) = -c_{\tau+1} \mid \mathbf{w}_t^{i,\alpha}) & = \frac{|w_{t}^{i,j,\alpha}[k]|-c_\tau}{c_{\tau+1}-c_\tau}
        \end{aligned}
        \right..
    \end{equation}
    Given that the quantization of each element is independent of that of other elements, we can consider the privacy of different elements separately. For simplicity of notation, we denote the $j$-th element of any two adjacent model parameters, $w_{t}^{i,j,\alpha,(1)}[k]$ and $w_{t}^{i,j,\alpha,(2)}[k]$, by $x_t^{i,j}$ and $y_t^{i,j}$, respectively. Since the difference between any two adjacent model parameters is bounded, it has $\|x_t^i-y_t^i\|_1\leq C_4$ for some $C_4>0$ $\forall i\in\mathcal{N}$ and $\forall t\in\{1,2,...,T\}$.

    According to the definition of differential privacy, to prove that $(0,\delta)$-differential privacy is achieved for some $\delta>0$, i.e., $| Pr(\mathcal{Q}(x_t^{i,j})\in\mathcal{S}| x_t^{i,j})- Pr(\mathcal{Q}(y_t^{i,j})\in\mathcal{S}| y_t^{i,j})|\leq \delta$ for all $\mathcal{S}\in\{\pm c_{\tau},\pm c_{\tau+1}\}$ and all $x_t^{i,j}$ and $y_t^{i,j}$ with $\|x_t^{i}-y_t^{i}\|_1\leq C_4$, the derivation can be divided into two cases: 1) $x_t^{i,j}$ and $y_t^{i,j}$ are of the same sign; 2) $x_t^{i,j}$ and $y_t^{i,j}$ are of different signs.

    \textit{Case 1:} Considering that both $x_t^{i,j}$ and $y_t^{i,j}$ are nonnegative. One has that
    \begin{align}
        &\sup_{\|x_t^{i}-y_t^{i}\|_1\leq C_4} \left| Pr\left(\mathcal{Q}(x_t^{i,j}) = c_{\tau} \right) - Pr\left(\mathcal{Q}(y_t^{i,j})= c_{\tau} \right) \right| \nonumber\\
        & = \sup_{\|x_t^{i}-y_t^{i}\|_1 \leq C_4} \left| \frac{c_{\tau+1}-|x_{t}^{i,j}|}{c_{\tau+1}-c_\tau} - \frac{c_{\tau+1}-|y_{t}^{i,j}|}{c_{\tau+1}-c_\tau} \right|\nonumber\\
        &   \leq \frac{|x_{t}^{i}|-|y_{t}^{i}|}{c_{\tau+1}-c_\tau}\nonumber\\
        &   \leq \frac{C_4}{c_{\tau+1}-c_\tau}, \label{case1a}
    \end{align}
    and
    \begin{align}
        &\sup_{\|x_t^{i}-y_t^{i}\|_1\leq C_4} \left| Pr\left(\mathcal{Q}(x_t^{i,j}) = c_{\tau+1} \right) - Pr\left(\mathcal{Q}(y_t^{i,j})= c_{\tau+1} \right) \right| \nonumber\\
        & = \sup_{\|x_t^i-y_t^i\|_1 \leq C_4} \left| \frac{|x_{t}^{i,j}|-c_\tau}{c_{\tau+1}-c_\tau} - \frac{|y_{t}^{i,j}|-c_\tau}{c_{\tau+1}-c_\tau} \right|\nonumber\\
        &   \leq \frac{C_4}{c_{\tau+1}-c_\tau} \label{case1b}.
    \end{align}
    Similarly, one can obtain the same relation when both $x_t^{i,j}$ and $y_t^{i,j}$ are negative.

    \textit{Case 2:} Considering that $x_t^{i,j}$ and $y_t^{i,j}$ are of different signs. Assume $x_t^{i,j}\geq 0$ and $y_t^{i,j}<0$. Based on the rule of stochastic quantization, one obtains that
    \begin{align}
        &\sup_{\|x_t^{i}-y_t^{i}\|_1\leq C_4} \left| Pr\left(\mathcal{Q}(x_t^{i,j}) = -c_{\tau} \right) - Pr\left(\mathcal{Q}(y_t^{i,j})= -c_{\tau} \right) \right| \nonumber\\
        & = \sup_{\|x_t^i-y_t^i\|_1 \leq C_4} \left| 0 - \frac{c_{\tau+1}-|y_{t}^{i,j}|}{c_{\tau+1}-c_\tau} \right|\nonumber\\
        &   \leq \frac{ |\Delta_t^{i,j}[k]| }{c_{\tau+1}-c_\tau} \label{case2a},
    \end{align}
    and
    \begin{align}
        &\sup_{\|x_t^{i}-y_t^{i}\|_1\leq C_4} \left| Pr\left(\mathcal{Q}(x_t^{i,j}) \hspace{-1pt}=\hspace{-1pt} -c_{\tau+1} \right) \hspace{-1pt}-\hspace{-1pt} Pr\left(\mathcal{Q}(y_t^{i,j})\hspace{-1pt}=\hspace{-1pt} -c_{\tau+1} \right) \right| \nonumber\\
        & = \sup_{\|x_t^i-y_t^i\|_1 \leq C_4} \left| 0 - \frac{|y_{t}^{i,j}|-c_\tau}{c_{\tau+1}-c_\tau} \right|\nonumber\\
        &     \leq \frac{ |\Delta_t^{i,j}[k]| }{c_{\tau+1}-c_\tau}\label{case2b},
    \end{align}
    where $\Delta_t^{i,j}[k]$ is the quantization error of $j$-th element. The same result can be obtained if $x_t^{i,j}< 0$ and $y_t^{i,j}\geq 0$.

    Recalling that $c_{\tau+1}-c_\tau$ is determined by the dynamic quantization interval $\mathcal{R}_{t}^i[k+1]=\left[\mathbf{q}_{t}^i[k]-\frac{\pi_t a_{\alpha\beta_1}^{\max}[k]}{2}\mathbf{1},\mathbf{q}_{t}^i[k]+\frac{\pi_t a_{\alpha\beta_1}^{\max}[k]}{2}\mathbf{1}\right]$, and there are $l$ knobs located in the interval $\mathcal{R}_{t}^i[k]$. Hence, we have
    \begin{equation}\label{interval2}
      c_{\tau+1}-c_{\tau}=\frac{ \pi_t a_{\alpha\beta_1}^{\max}[k] }{ l-1 },
    \end{equation}
    where $a_{\alpha\beta_1}^{\max}[k]\overset{\triangle}{=}\max\{a_{i,\alpha\beta_1}[k]|i\in\mathcal{S}_t\}$, and $a_{i,\alpha\beta_1}[k]$ is the diminishing step size (e.g., $a_{i,\alpha\beta_1}[k]=\frac{\gamma_i}{k+1}$ with ${\gamma_i \in (0, \frac{\lambda_{\min,\mathbf{U}}}{1+\lambda_{\min,\mathbf{U}}})}$). Specific settings of parameters should satisfy the conditions as in Lemma 10 such that the quantization input is always located in the shrinking quantization interval. Meanwhile, from Lemma 10, we have that the quantization error of each element in the model parameter vector is bounded by
    \begin{equation}\label{quan_error}
      |{\Delta}_{t}^{i,j}[k]|\leq \frac{\pi_{t}}{2^B-1} a_{\alpha\beta_1}^{\max}[k],
    \end{equation}
    where $B=\lceil\log_{2}(l)\rceil$.

    Inserting \eqref{interval2} and \eqref{quan_error} into \eqref{case1a}-\eqref{case2b} in Cases 1 and 2, we always have at least $(0,\delta)$-differential privacy with $\delta=\min\left\{\frac{C_4(l-1)}{\pi_t a_{\alpha\beta_1}^{\max}[k]},\frac{l-1}{2^B-1}\right\}$ for each element of the quantization input $\mathbf{w}_t^{i,\alpha}[k]$ for each client $i\in\mathcal{S}_t$.

    This completes the proof.
\end{proof}

\begin{remark}
  Regarding the tuning rule of the dynamic shrinking interval in $\mathcal{R}_t^i[k+1]$, two key parameters are involved, i.e., $a_{\alpha\beta_1}^{\max}[k]$ and $\widetilde{W}_{t,k}^{\max}$ in $\pi_t$. For the tuning rule of the decaying factor $a_{\alpha\beta_1}^{\max}[k]$, we can set $a_{i,\alpha\beta_1}[k]=\frac{\gamma_{\max}}{k+1}$ for all $i\in\mathcal{S}_t$ where $\gamma_{\max}\in (0, \frac{\lambda_{\min,\mathbf{U}}}{1+\lambda_{\min,\mathbf{U}}})$. For the tuning rule of $\widetilde{W}_{t,k}^{\max}$, we can iteratively compare the value of $\|\mathbf{W}_{t}^\alpha[k]-\mathbf{W}_{t}^{\beta_1}[k]\|$ with that of $\|\mathbf{W}_{t}^\alpha[k-1]-\mathbf{W}_{t}^{\beta_1}[k-1]\|$, and select the larger one as $\widetilde{W}_{t,k}^{\max}$, because $\|\mathbf{W}_{t}^\alpha[k-1]-\mathbf{W}_{t}^{\beta_1}[k-1]\|$ is the largest value up to the $k$-th communication round.
\end{remark}

\section{Discussion}
In this section, we give extended discussions on the strategy for nonconvex settings, and the impact of the scale of invisible submodels on the global model quality.
{\subsection{Extension to nonconvex settings}
In nonconvex settings for FL, SGD is a widely used and efficient method. Under the Polyak–Łojasiewicz condition or general nonconvex assumption for objective functions, previous studies have established a rigorous theoretical foundation for SGD with almost sure convergence. For the linear speedup property of SGD variants, accelerated techniques, e.g., momentum methods \cite{nonconvex_pmlr}, can be incorporated into SGD, which also increases the probability of optimization algorithms escaping from the local minimum. In our proposed algorithms (MSP-FL and MSPDQ-FL), SGD is still the primary method in the local training stage. Our designed modules for privacy enhancement and quantization lie in the communication process after the training stage. Consequently, the variants of our proposed algorithms in nonconvex settings can be naturally extended based on those strategies for nonconvex settings, e.g., momentum SGD and proximal approach \cite{nonconvex_pmlr,proximal_tsp_nonconvex}. However, for the variants of our proposed algorithms in nonconvex settings, the analysis framework needs to be re-designed. Related literature usually uses the (average) expected squared gradient norm to characterize the convergence rate, because establishing the convergence rate for objective values, as in Theorems 2 and 3, is generally impossible.}
\subsection{Impact of the scale of invisible submodels}
In the above, we consider the model splitting with $m_i=1$ for all $i$ for ease of convergence analysis. Here, we discuss the impact of the number of invisible submodels $m_i$ for $i\in\mathcal{S}_t$ on the global model quality when $m_i\geq 2$. Denote $m=\max\{m_i|i\in\mathcal{S}_t\}$, $\mathbf{A}_n[k]=\text{diag}(a_{i,\alpha\beta_n}[k]|{i\in\mathcal{S}_{t}})$ for $n=1,2,...,m$, and $\hat{\mathbf{A}}[k]=\sum_{n=1}^{m}\mathbf{A}_n[k]$. We rewrite $\mathbf{P}[k]$ in \eqref{notation3} in the following formulation that is with $m\geq 2$:\par
\vspace{-8pt}
\begin{small}
\begin{eqnarray}\label{P_2[K]}
      \mathbf{P}[k]\hspace{-8pt}&=&\hspace{-10pt}\left[
        \begin{array}{ccccc}
          \mathbf{U}-\hat{\mathbf{A}}[k] & \mathbf{A}_1[k] & \mathbf{A}_2[k] &\cdots &  \mathbf{A}_m[k]\\
           \mathbf{A}_1[k]  & \mathbf{I}-\mathbf{A}_1[k] & \mathbf{0} & \cdots & \mathbf{0}\\
           \mathbf{A}_2[k]  & \mathbf{0} & \mathbf{I}-\mathbf{A}_2[k] & \vdots & \vdots\\
           \vdots & \vdots &\cdots &\ddots & \mathbf{0} \\
           \mathbf{A}_m[k] & \mathbf{0} & \cdots & \mathbf{0} & \mathbf{I}-\mathbf{A}_m[k]
        \end{array}
    \right].\nonumber\\
    &&
\end{eqnarray}
\end{small}
\par\noindent Noting that $m$ is the maximal number of local invisible submodels in the selected clients set $\mathcal{S}_t$, therewith some of diagonal entries in $\mathbf{A}_n[k]$, for $n=2,3,...,m$, are $0$, i.e., $a_{i,\alpha\beta_n}[k]=0$ for some $i\in\mathcal{S}_{t}$ and $n=2,3,...,m$.

By observing the state transition matrix $\mathbf{P}[k]$ in \eqref{P_2[K]}, we have that the conditions that suffice for the positive semi-definiteness of the matrix $\mathbf{P}[k]$ are $\epsilon\in(0,\frac{M}{M-1})$ and $\sum_{n=1}^{m} {a_{\alpha\beta_n}^{\max}}\in (0, \frac{\lambda_{\min,\mathbf{U}}}{1+\lambda_{\min,\mathbf{U}}})$, where $a_{\alpha\beta_n}^{\max}=\max \{a_{i,\alpha\beta_n}^{\max}|i\in\mathcal{S}_t\}$. Clearly, if $m_i$ is quite large, each non-zero ${a_{i,\alpha\beta_n}}$ for $n\in\{1,2,...,m_i\}$ is assigned a small value for $i\in\mathcal{S}_t$. This, on one hand, leads to a slow rate for the update process in \eqref{decomposition1}. On the other hand, given an irreducible stochastic matrix $\mathbf{P}[k]$ and a valid value for $a_{\alpha\beta_n}^{\max}$, $n=1,2,...,m$ that satisfies $\sum_{n=1}^{m} {a_{\alpha\beta_n}^{\max}}\in (0, \frac{\lambda_{\min,\mathbf{U}}}{1+\lambda_{\min,\mathbf{U}}})$, it is known that reducing diagonal elements (or equivalently increasing off-diagonal elements) can decrease the second largest eigenvalue of $\mathbf{\Phi}(k,0)\overset{\triangle}{=}\mathbf{P}[k]\mathbf{P}[k-1]\ldots \mathbf{P}[0]$, denoted by $\lambda_{2,\mathbf{\Phi}(k,0)}$, as the number of invisible portions $m$ increases. Consequently, the designed steps for the communication rounds $K_t$ (e.g., $K_t=\lceil\log_\lambda^{2/\mu(\vartheta+t)}\rceil$ required in Theorem 2) may not be sufficient to suppress the accumulated model parameter errors due to model splitting where $\lambda\in(\lambda_{2,\mathbf{\Phi}(k,0)},1)$. This leads to a significant degradation of the global model quality. The above provides a qualitative analysis of the impact of the number of invisible portions on the global model quality.

\section{Numerical Experiments}
To validate the theoretical results, we present the simulations of the proposed MSP-FL and MSPDQ-FL. All the tests were conducted on the Ubuntu 18.04 operating system, with an Intel i7-10700 CPU running at 2.90 GHz and an Nvidia RTX3090 graphics card with 32 GB of video memory.
\subsection{Experimental Setup}
On the basis of python environment and pytorch platform, we build up a FL framework on image classification tasks. We train CNN models on the MNIST and CIFAR-10 datasets in a non-i.i.d fashion.

\emph{1) Models, Datasets and Parameters Settings:}
\begin{itemize}
  \item MNIST Dataset: The data, divided into $60000$ training and $10000$ test samples, is evenly distributed among $N=100$ clients, from which $M=20$ clients are selected to participate in aggregation. Major training parameters are: local mini-batch size $s_i=10$, local epoch $E=10$ and learning rate $\eta_t=\frac{1}{5(t+1)}$.

  \item CIFAR-10 Dataset: The CIFAR-10 dataset has $50000$ images for training and $10000$ images for testing. The distribution of data samples is similar to the MNIST. Major training parameters are: local mini-batch size $s_i=40$, local epoch $E=5$ and learning rate $\eta_t=\frac{1}{10(t+1)}$.
\end{itemize}

\emph{2) Baselines:} We compare MSP-FL and MSPDQ-FL with the following baselines:
\begin{itemize}
  \item FedAvg \cite{fedavg}: Vanilla FL.
  \item FedAvg+LDP \cite{kangwei_fl+dp_tifs_2020}: We deploy LDP technique into FedAvg as a baseline of MSP-FL to compare the testing accuracy.
  \item JoPEQ: We deploy JoPEQ from \cite{joint} as a baseline to compare the performance of MSPDQ-FL.
\end{itemize}

\subsection{Numerical Results}
\textit{1) Learning performance of MSP-FL:} We begin by only focusing on evaluating the learning performance of MSP-FL and MSPDQ-FL in terms of training loss and testing accuracy under CIFAR-10 datasets. As a baseline, FedAvg+LDP is equipped with a privacy budget $\epsilon_{DP}=10$. For MSP-FL, the communication round after each learning round is set as $K_t=4$. While for MSPDQ-FL, the communication round is set as $K_t=10$. The communication bits for MSPDQ and JoPEQ are both $B=8$.
We observe from Fig. \ref{performance-comparison-MSP-MSPDQ-FL-on-cifar} that MSP-FL achieves similar performance to FedAvg. Both MSP-FL and MSPDQ-FL show better performance than that of FedAvg+LDP or JoPEQ across all epochs under the CIFAR-10 dataset. It indicates that our proposed privacy-preserving mechanism in MSP-FL and MSPDQ-FL sacrifices less model precision during the aggregation process. If FedAvg+LDP combines with SDQ for communication efficiency with a carefully designed noise (i.e., JoPEQ), it presents a relatively poor learning performance compared to our proposed MSPDQ-FL. This is not surprising because the power of noise introduced by quantization in MSPDQ-FL significantly decreases compared to that in JoPEQ, as the quantization error asymptotically decays in the sense of expectation with the shrinking quantization interval.


\begin{figure}[t]
    \subfloat[CIFAR-10 / Loss]{
    \begin{minipage}[t]{0.5\linewidth}
    \centering
    \vspace{-0pt}\hspace{-11.5pt}\includegraphics[scale=0.295]{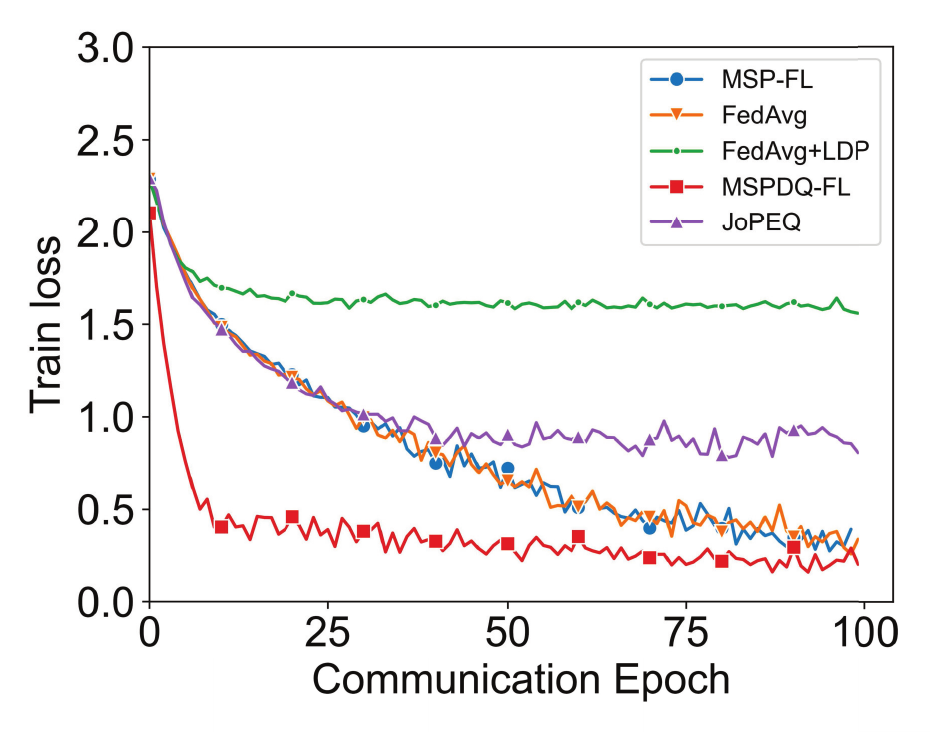}
    \label{mnist-trainloss}
    \end{minipage}%
    }%
    \subfloat[CIFAR-10 / Accuracy]{
    \begin{minipage}[t]{0.5\linewidth}
    \centering
    \vspace{-0pt}\hspace{-10pt}\includegraphics[scale=0.295]{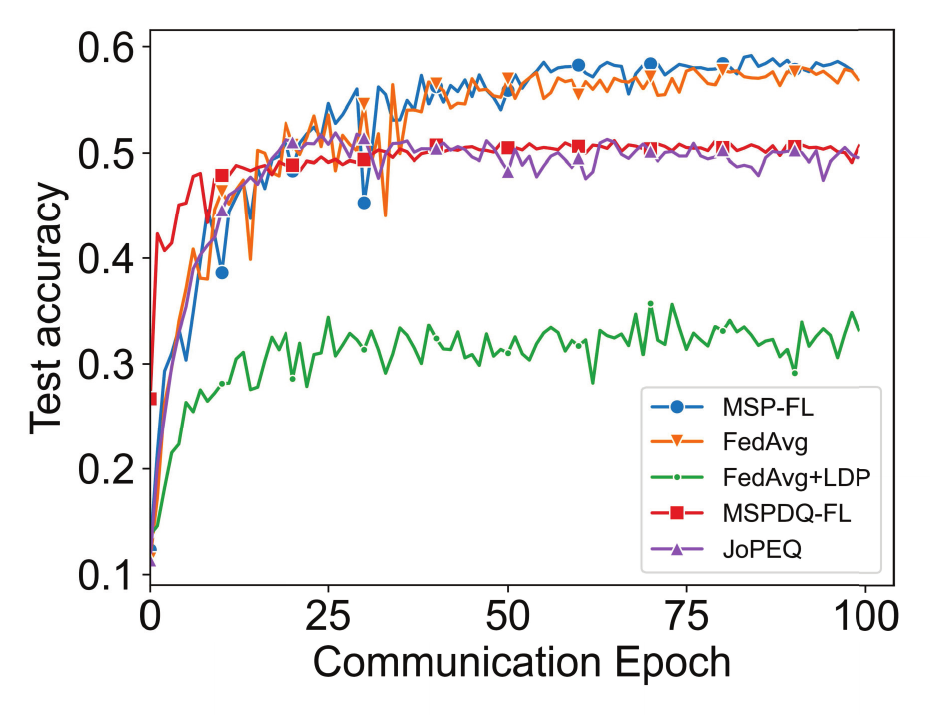}
    \label{mnist-testaccuracy}
    \end{minipage}
    }
    \caption{Comparing the learning performance of MSP-FL, MSPDQ-FL and baselines on CIFAR-10 dataset.}\label{performance-comparison-MSP-MSPDQ-FL-on-cifar}
\end{figure}

\begin{figure}[t]
    \subfloat[$K_t=20$ / MNIST]{
    \begin{minipage}[t]{0.5\linewidth}
    \centering
    \vspace{-0pt}\hspace{-12.8pt}\includegraphics[scale=0.29]{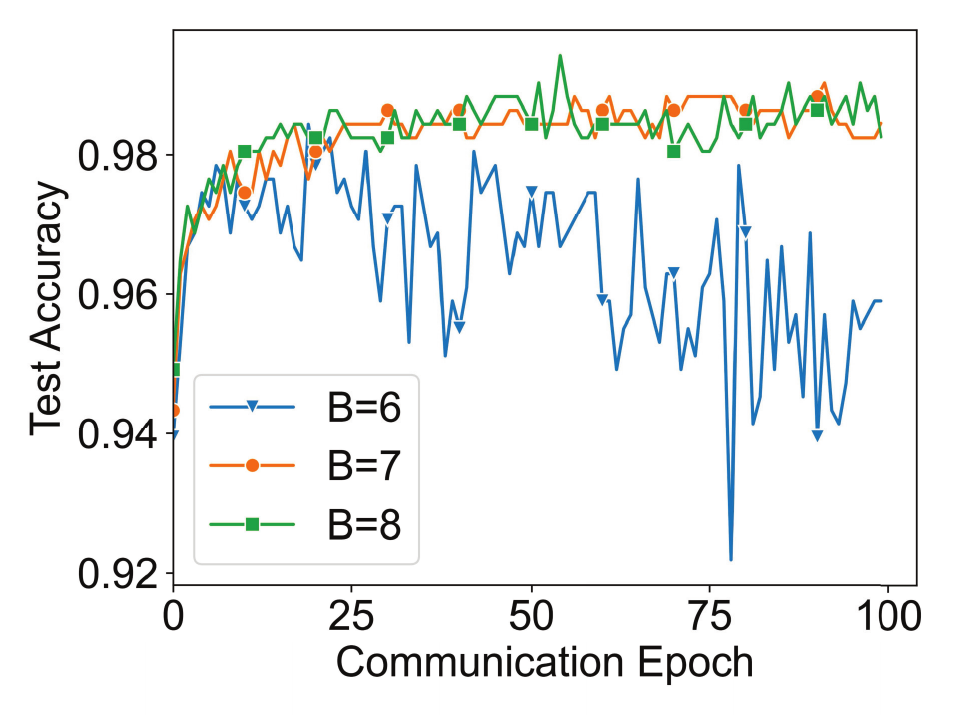}
    \label{K20-mnist}
    \end{minipage}%
    }%
    \subfloat[$K_t=30$ / MNIST]{
    \begin{minipage}[t]{0.5\linewidth}
    \centering
    \vspace{-0pt}\hspace{-12pt}\includegraphics[scale=0.29]{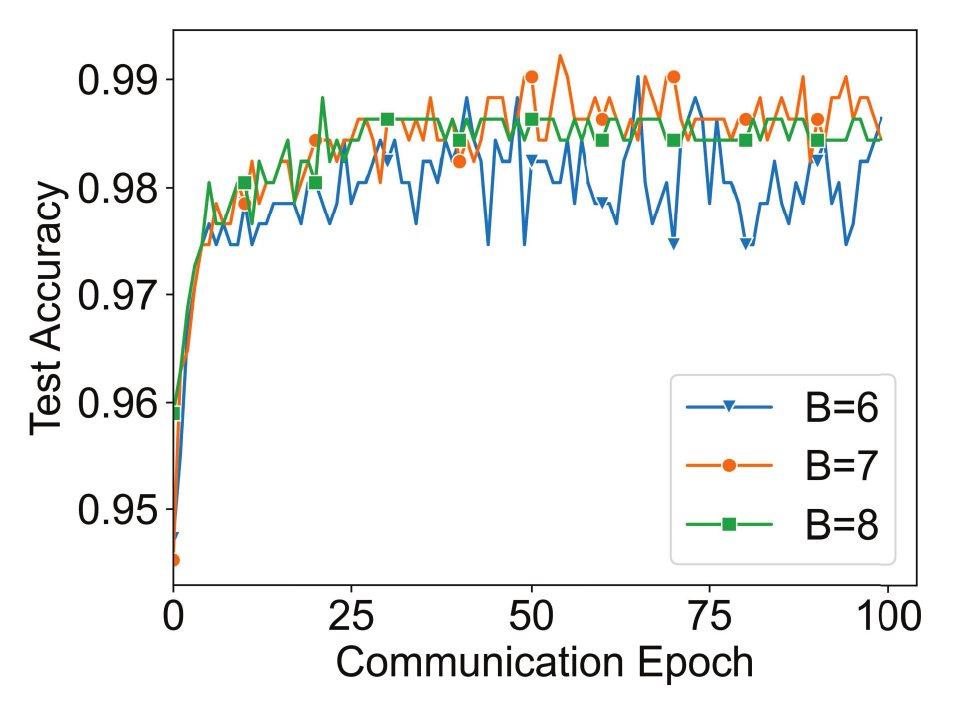}
    \label{K30-mnist}
    \end{minipage}
    }\\
    \vspace{-0pt}\subfloat[$K_t=20$ / CIFAR-10]{
    \begin{minipage}[t]{0.5\linewidth}
    \centering
    \vspace{-0pt}\hspace{-10pt}\includegraphics[scale=0.29]{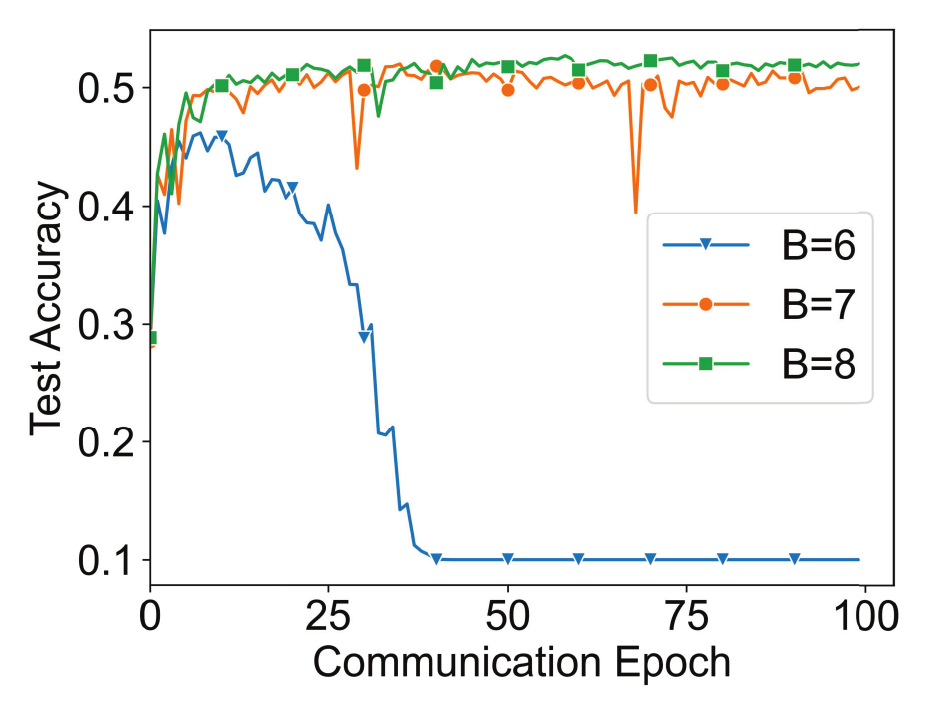}
    \label{K20-cifar}
    \end{minipage}
    }
    \subfloat[$K_t=30$ / CIFAR-10]{
    \begin{minipage}[t]{0.5\linewidth}
    \centering
    \vspace{-0pt}\hspace{-13pt}\includegraphics[scale=0.29]{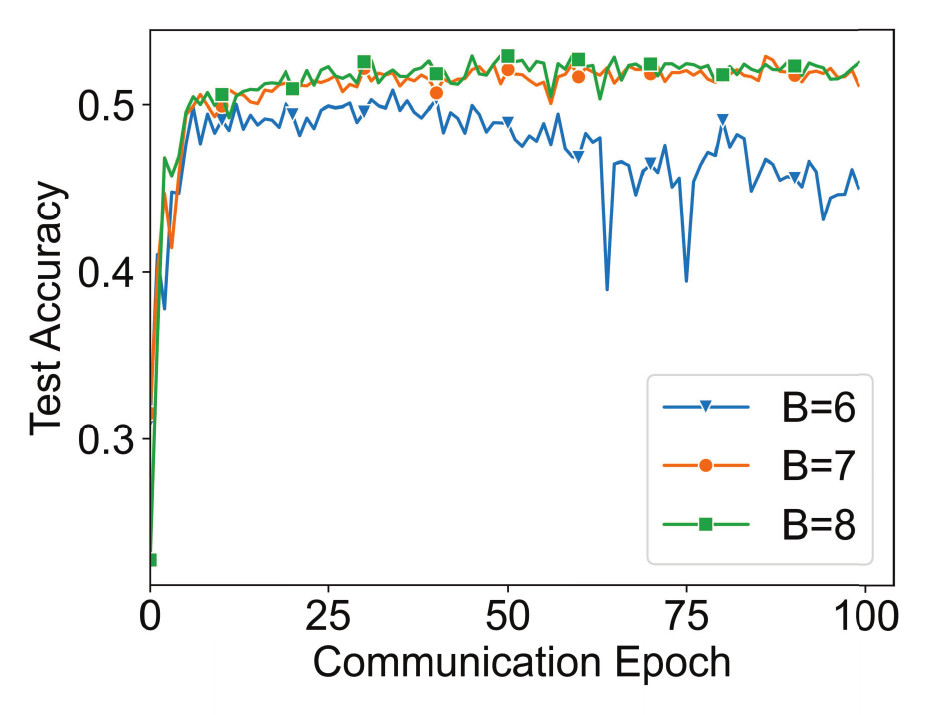}
    \label{K30-cifar}
    \end{minipage}
    }
    \caption{Comparing the learning accuracy of MSPDQ-FL with different communication rounds and quantization bits.}\label{performance-comparison-different-k-b}
\end{figure}
\begin{figure}[t]
\centering
    \includegraphics[scale=0.4]{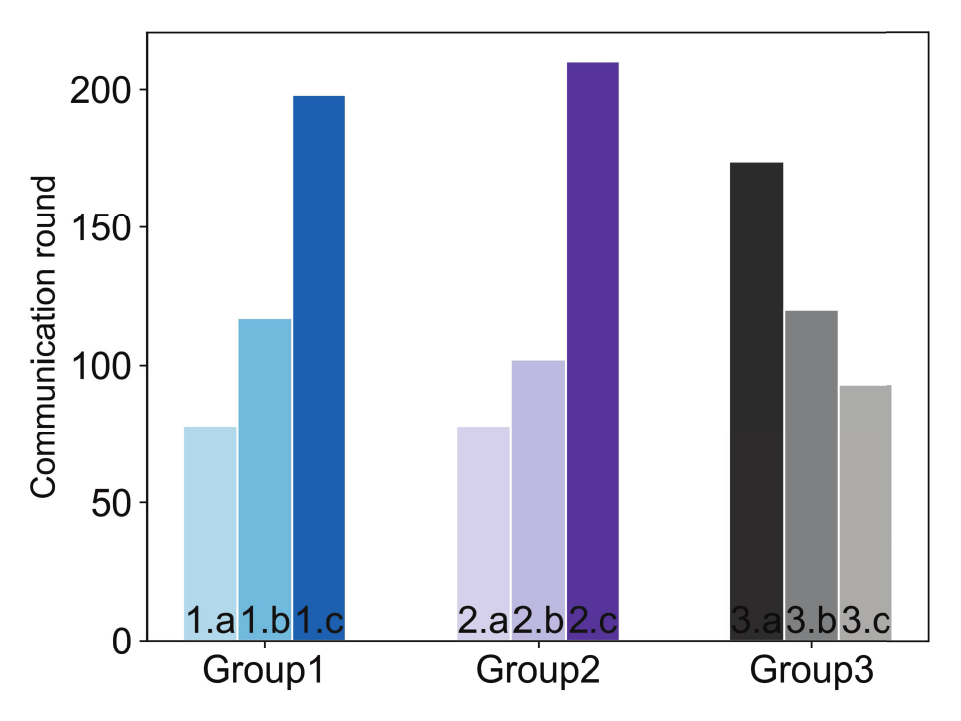}
    \vspace{0pt}\caption{The impact of different dominated coefficients on the communication complexity of MSPDQ-FL. Group 1 is dominated by prescribed learning error $\|\mathbf{w}_t-\mathbf{w}^*\|$, where the error ratio in 1.a, 1.b and 1.c is $4:3:2$. Group 2 is dominated by communication bits $B$, where $B_{2.a}=8$, $B_{2.b}=7$ and $B_{2.c}=6$. Group 3 is dominated by model-splitting coefficient $\varepsilon$, where $\varepsilon_{3.a}=0.3$, $\varepsilon_{3.b}=0.5$ and $\varepsilon_{3.c}=0.7$.}\label{com-complexity}
\end{figure}

             \begin{figure}[t]
                \subfloat[$a_{i,\alpha\beta_1}{[}k{]}=\frac{\gamma_{\max}}{k+1}$ / Accuracy ]{
                \begin{minipage}[t]{0.5\linewidth}
                \centering
                \hspace{-10pt}\includegraphics[width=1.81in,height=1.3in]{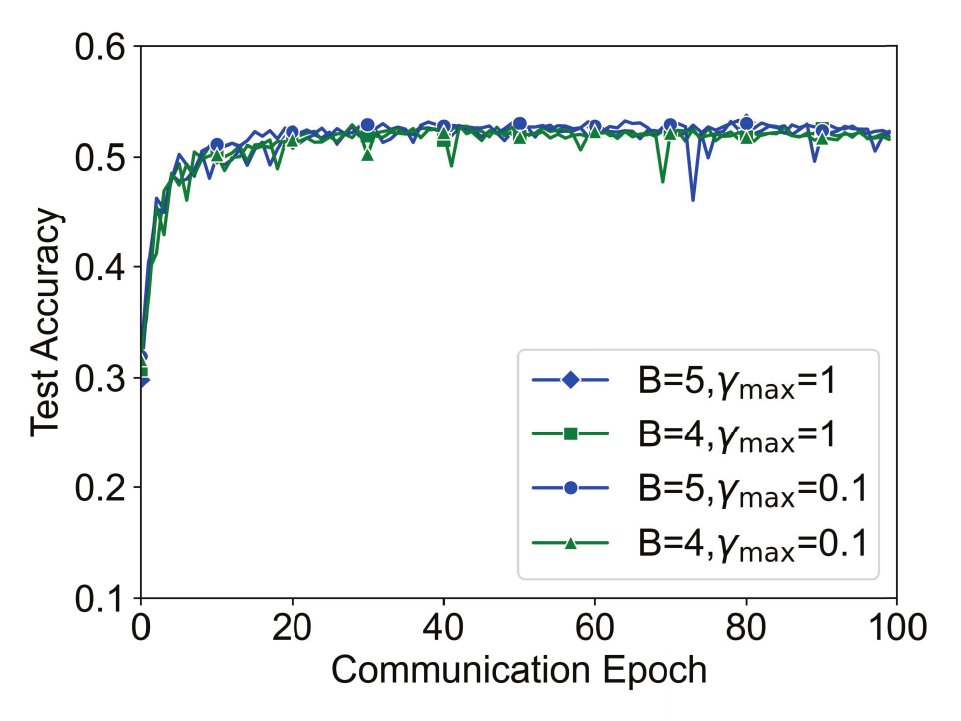}
                \label{hyperparameter_1/k_accuracy}
                \end{minipage}%
                }%
                \subfloat[$a_{i,\alpha\beta_1}{[}k{]}=\frac{\gamma_{\max}}{k+1}$ / Loss ]{
                \begin{minipage}[t]{0.5\linewidth}
                \centering
                \hspace{-10pt}\includegraphics[width=1.81in,height=1.3in]{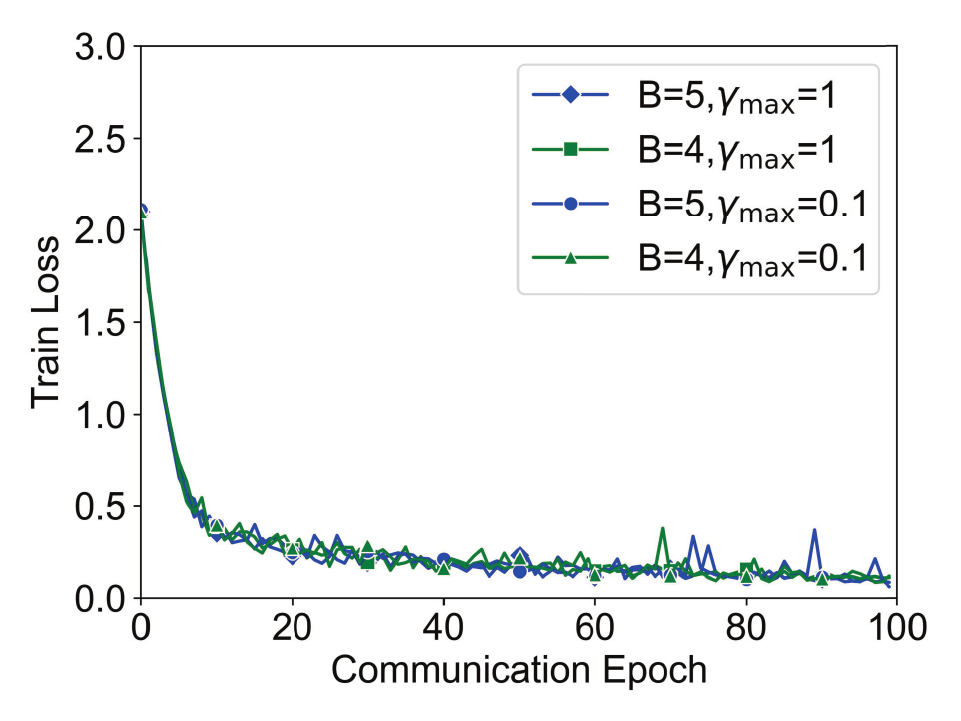}
                \label{hyperparameter_1/k_loss}
                \end{minipage}
                }\\
                \subfloat[$a_{i,\alpha\beta_1}{[}k{]}=\frac{\gamma_{\max}}{\sqrt{k+1}}$ / Accuracy ]{
                \begin{minipage}[t]{0.5\linewidth}
                \centering
                \hspace{-10pt}\includegraphics[width=1.81in,height=1.3in]{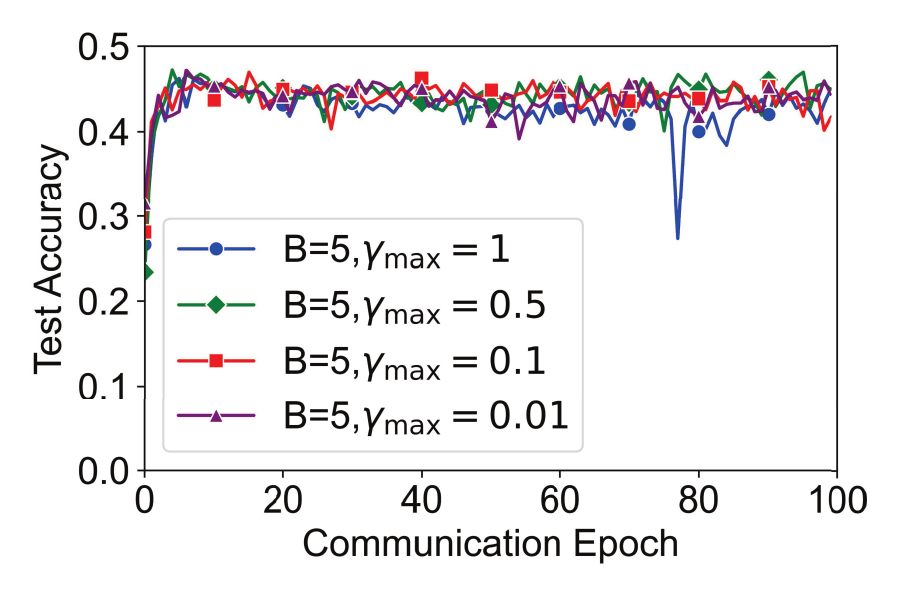}
                \label{hyperparameter_sqrt_accuracy}
                \end{minipage}%
                }%
                \subfloat[$a_{i,\alpha\beta_1}{[}k{]}=\frac{\gamma_{\max}}{\sqrt{k+1}}$ / Loss]{
                \begin{minipage}[t]{0.5\linewidth}
                \centering
                \hspace{-10pt}\includegraphics[width=1.81in,height=1.3in]{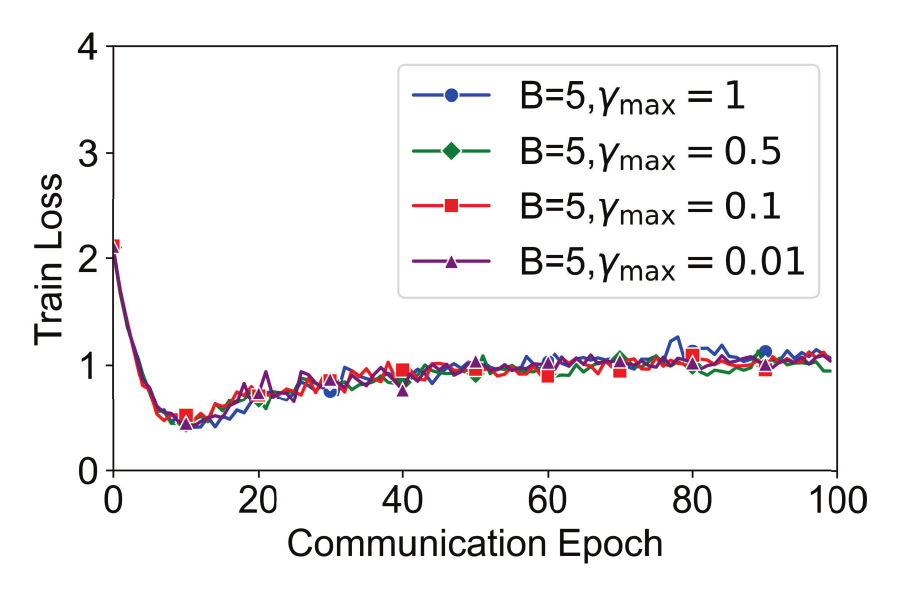}
                \label{hyperparameter_sqrt_loss}
                \end{minipage}
                }
                \caption{Learning performance of MSPDQ-FL on CIFAR-10 dataset with $K_t=20$ and different hyperparameters.}\label{hyperparameter}
            \end{figure}

\textit{2) Learning performance and communication complexity of MSPDQ-FL:} In Fig. \ref{performance-comparison-different-k-b}, we show the convergence performance of MSPDQ-FL under different $K_t$ and $B$. Since the performance of MSP-FL under different $K_t$ is similar to that of MSPDQ-FL, it is omitted. In the experiment, we notice that the value of $K_t$ is mainly governed by a small strongly-convex coefficient $\mu$ as shown in Theorem 3 (e.g., set as $10^{-2}$), hence, we set most $K_t$ as constant. We observe in Fig. \ref{K20-mnist} and Fig. \ref{K20-cifar} that the test accuracy is normal in all cases during initial epochs. At around the $10$th epoch, the test accuracy degrades under $B=6$ in Fig. \ref{K20-cifar}, resulting from the fact that $K_t=20$ is insufficient for the case under 6 bits as $t$ increases. The explanation for such a fact is that, on one hand, it requires sufficient communication rounds for aggregation to reduce the error caused by model splitting; on the other hand, it needs a sufficiently large $K_t$ to wait for the quantization interval to shrink into a small range, especially when the number of quantization bits is small.

In Fig. \ref{com-complexity}, we pick a few intuitively dominated coefficients to see their impacts on the communication complexity of MSPDQ-FL. The results are reasonable and in accordance with the derived upper bound of communication complexity in Proposition 3.

{In Fig. \ref{hyperparameter}, we further investigate how sensitive the learning performance of MSPDQ-FL is to the major hyperparameters in shrinking  quantization intervals. We set $\pi_{t}=\frac{8(\epsilon+\tilde{\epsilon}+\tilde{\epsilon}\widetilde{W}_{t,k}^{\max})}{1-\lambda_{2,\mathbf{U}}}$, where $\widetilde{W}_{t,k}^{\max}\overset{\triangle}{=}\max\{\|\mathbf{W}_{t}^\alpha[k_1]-\mathbf{W}_{t}^{\beta_1}[k_1]\||k_1\in\{0,...,k\}\}$
             iteratively finds its maximum value. The constant parameter $\gamma_{\max}$ in $a_{i,\alpha\beta_1}[k]\overset{\triangle}{=}\frac{\gamma_{\max}}{k+1}$ is modified to show the impact of the dynamic quantization interval sizes on the learning performance. Moreover, considering the potential impact of the shrinking rate of the quantization intervals, we add a comparison result by setting $a_{i,\alpha\beta_1}[k]=\frac{\gamma_{\max}}{\sqrt{k+1}}$. Given a fixed rule for $a_{i,\alpha\beta_1}[k]$, we can observe from Fig. \ref{hyperparameter} that the parameter $\gamma_{\max}$, that influences the quantization interval sizes, has little impact on the learning performance in each subfigure. By comparing Fig. \ref{hyperparameter_1/k_accuracy} with Fig. \ref{hyperparameter_sqrt_accuracy} or comparing Fig. \ref{hyperparameter_1/k_loss} with Fig. \ref{hyperparameter_sqrt_loss}, one can see a decline of learning performance with the decaying factor $a_{i,\alpha\beta_1}[k]=\frac{\gamma_{\max}}{\sqrt{k+1}}$ for shrinking quantization intervals. This is because, under fixed quantization bits, quantization intervals $\mathcal{R}_t^i[k+1]$ with a slower decaying rate result in relatively larger quantization errors.
\begin{figure}[t]
    \subfloat[The proposed methods]{
    \begin{minipage}[t]{0.5\linewidth}
    \centering
    \hspace{-10pt}\includegraphics[width=1.67in,height=1.25in]{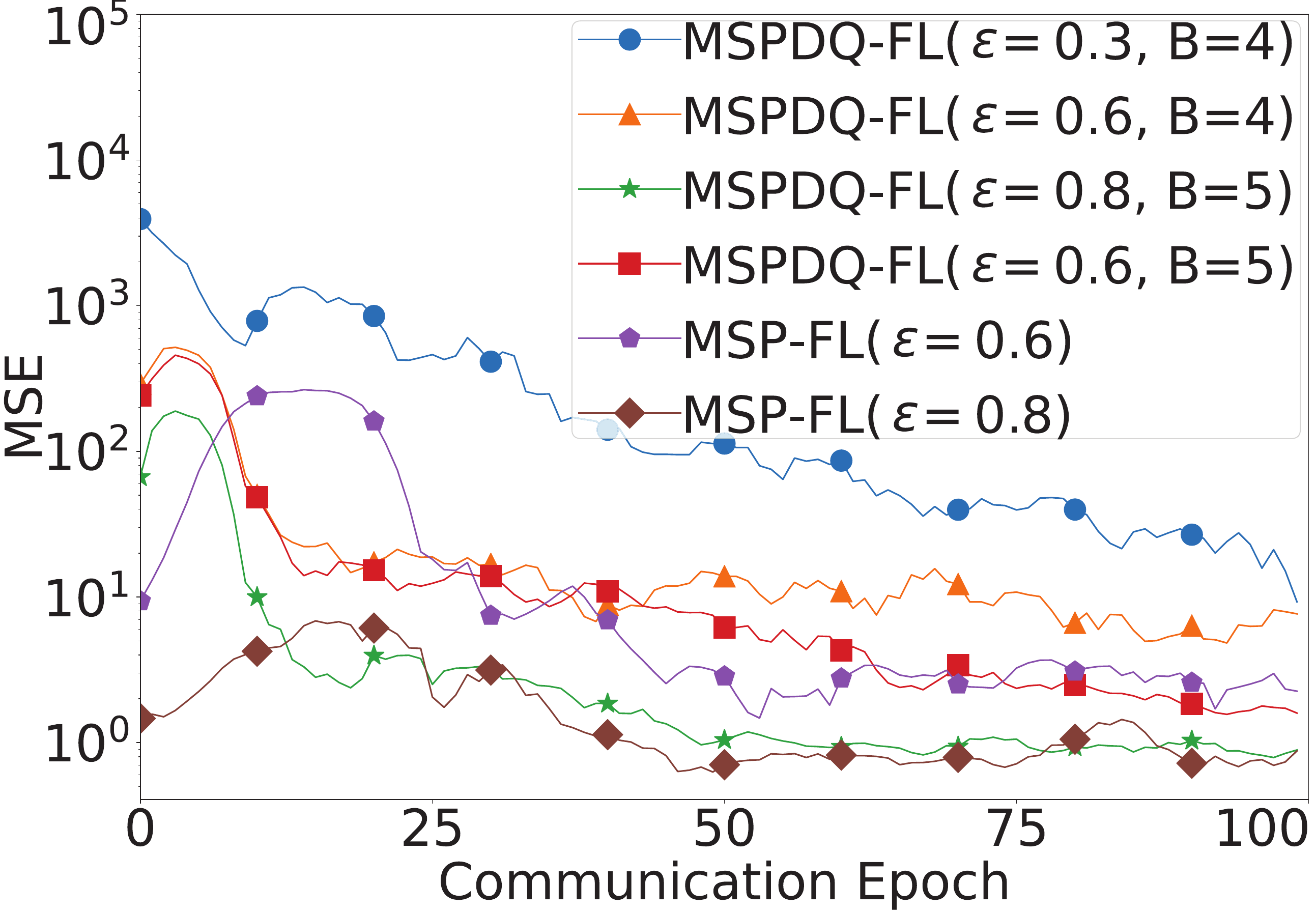}
    \label{msp_mse}
    \end{minipage}%
    }%
    \subfloat[DP-guaranteed methods]{
    \begin{minipage}[t]{0.5\linewidth}
    \centering
     \hspace{-12pt}\includegraphics[width=1.67in,height=1.25in]{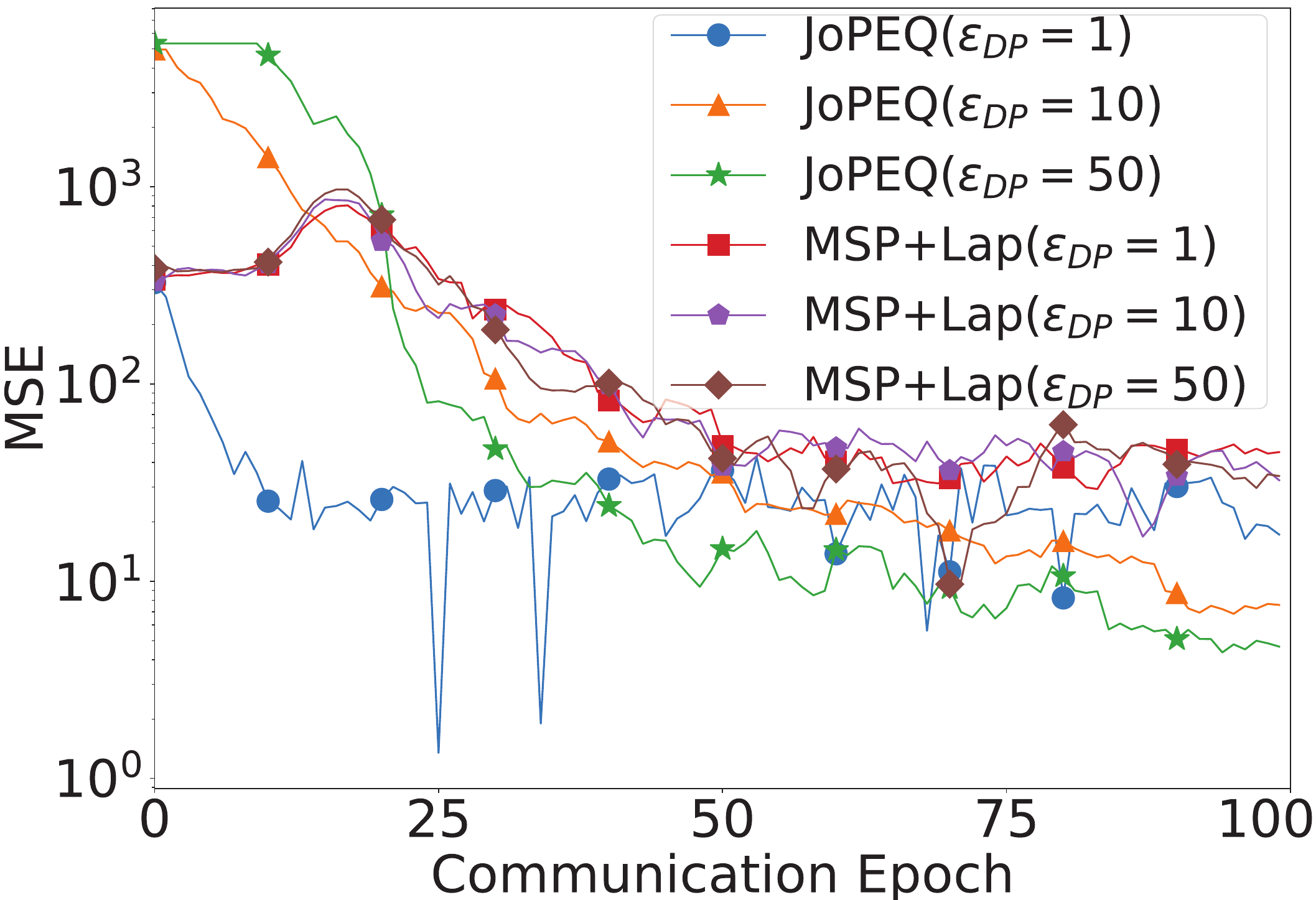}
    \label{jopeq_mse}
    \end{minipage}
    }\\
    \subfloat[The proposed methods]{
    \begin{minipage}[t]{0.5\linewidth}
    \centering
    \hspace{-10pt}\includegraphics[width=1.67in,height=1.25in]{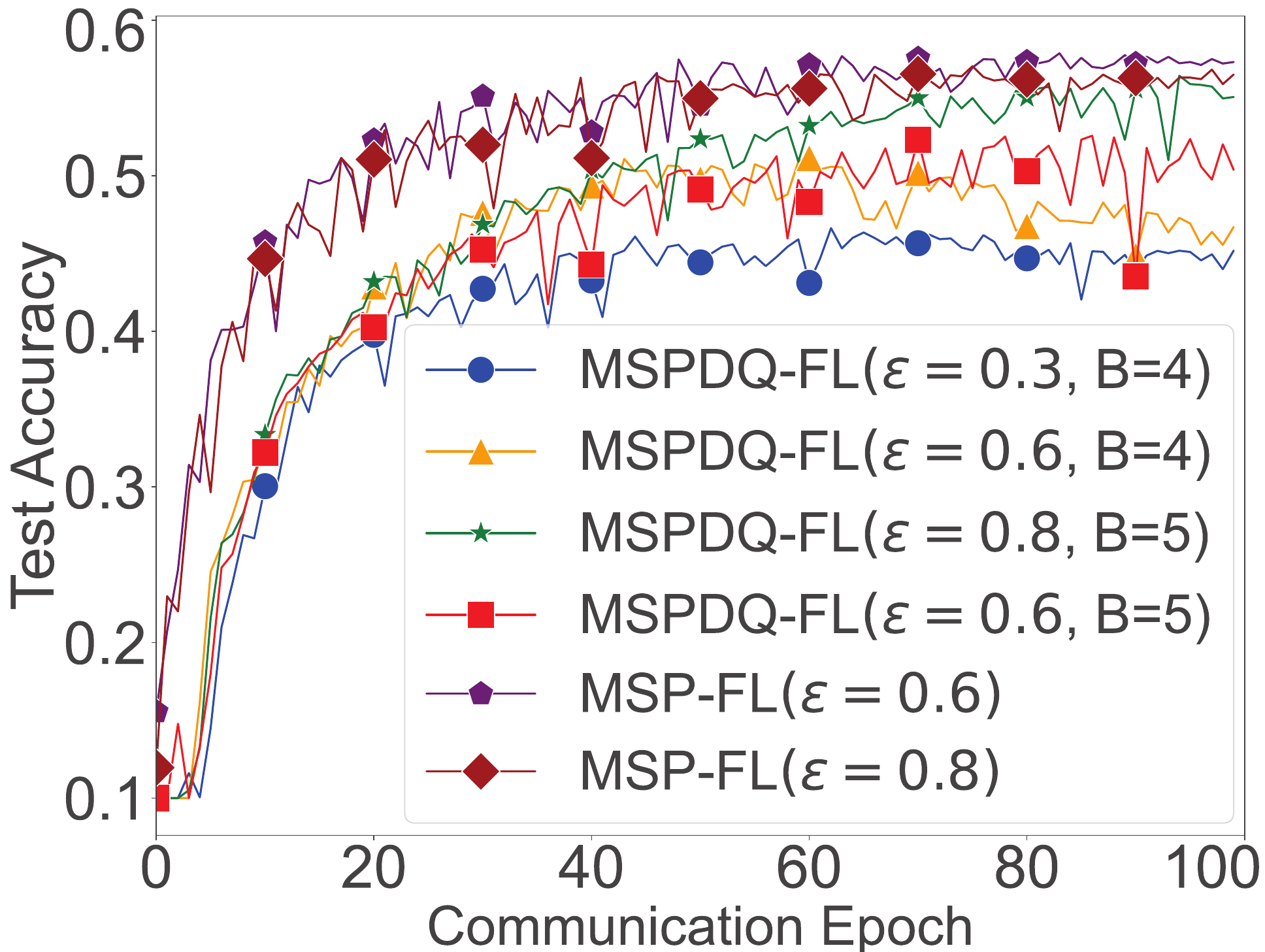}
    \label{msp_acc}
    \end{minipage}%
    }%
    \subfloat[DP-guaranteed methods]{
    \begin{minipage}[t]{0.5\linewidth}
    \centering
     \hspace{-12pt}\includegraphics[width=1.67in,height=1.25in]{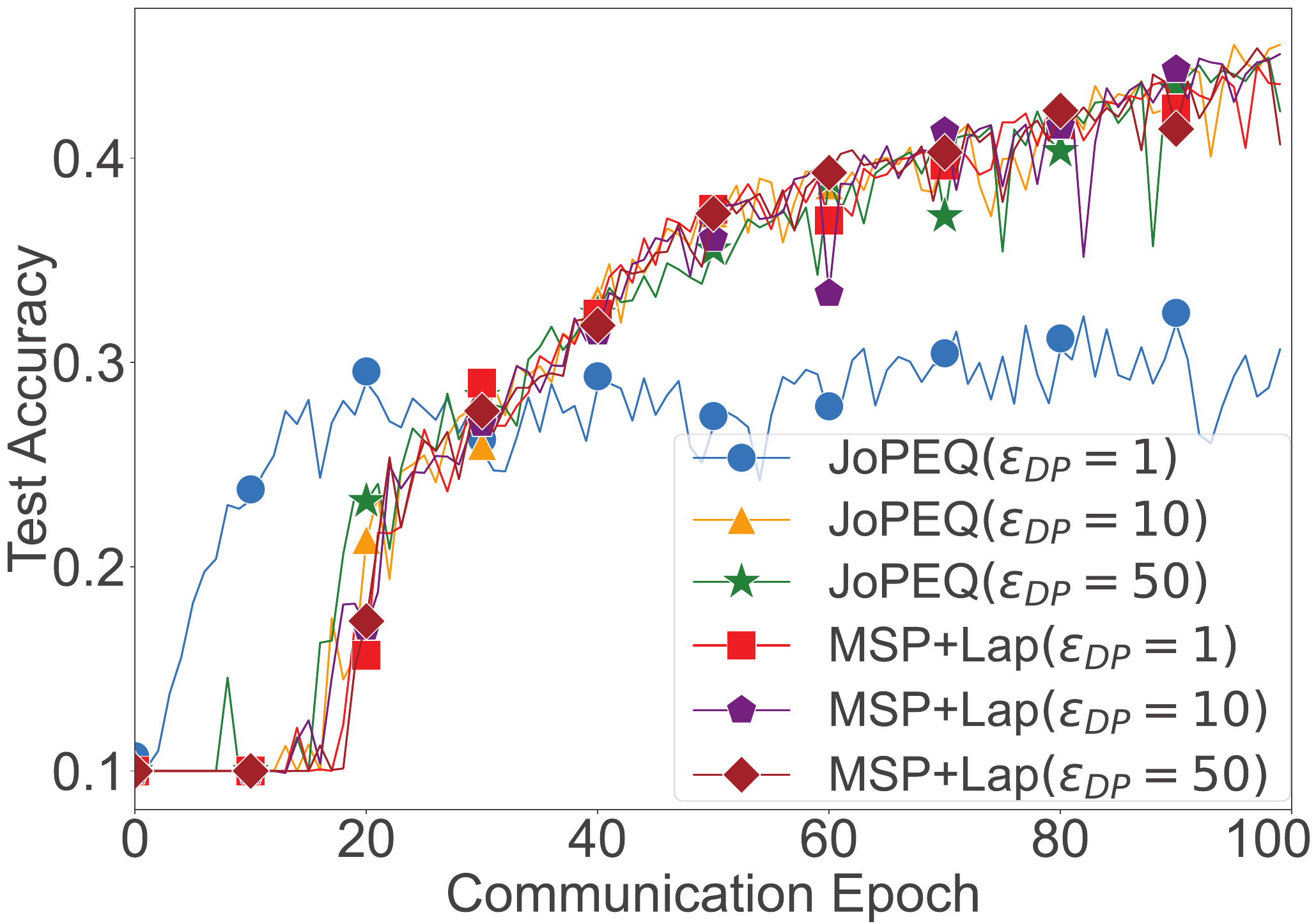}
    \label{jopeq_acc}
    \end{minipage}
    }
    {\caption{MSE defense score and the corresponding learning accuracy under MSP-FL, MSPDQ-FL, MSP+Lap and JoPEQ.}}\label{mse_score}
\end{figure}

\textit{3) Privacy leakage mitigation evaluation:} Here, we empirically verify the privacy guarantees of MSP-FL and MSPDQ-FL by assessing the defense performance against model inversion attacks. The attacks are implemented based on iDLG mechanism proposed in \cite{idlg}. The setup including the image classification task, model architecture, dataset and attack parameters, follows the one in \cite{idlg}. To evaluate the privacy leakage mitigation, mean squared error (MSE) is used to measure the discrepancy between the reconstructed images and truth label as the communication epoch progresses. Regarding the parameter setting, note that the initial values of submodels in the model splitting procedure are influenced by the splitting factor $\epsilon$, hence we treat this factor as analogous to the privacy budget $\epsilon_{DP}$ in differential privacy although this has not yet been fully substantiated through rigorous theoretical analysis. Therewith, Fig. \ref{msp_mse} presents the MSE defense score achieved by MSP-FL and MSPDQ-FL against iDLG under different splitting factors and quantization bits. As a comparison, Fig. \ref{jopeq_mse} reports the MSE defense scores achieved by DP-based method (JoPEQ) under different privacy budgets $\epsilon_{DP}$. The results indicate that, under the current parameter settings, our proposed methods provide a level of privacy protection comparable to that of DP-based method. Fig. \ref{msp_mse} further demonstrates that, due to quantization noise, MSPDQ-FL brings stronger privacy protection than MSP-FL. As the number of quantization bits decreases, the MSE naturally increases. This is consistent with the existing conclusion, i.e., the quantization enhances privacy preservation. 

Considering that the comparison between the proposed model splitting methods and the differential privacy method is somewhat unfair, even though our intention is only to demonstrate that the proposed methods provides a certain level of privacy protection rather than greater privacy strength than differential privacy methods, we supplement the MSE defense scores of MSP-FL under the splitting rule that guarantees differential privacy property for initial values of submodels in Fig. 6b. As discussed in Section \uppercase\expandafter {\romannumeral3}-C, there exists a splitting rule that ensures the noise injected into the initial submodel parameters follows a Laplace distribution. For simplicity, we denote the MSP-FL method with Laplace distribution for the injected noise as MSP+Lap. MSE scores of MSP+Lap in Fig. 6b demonstrate a matched privacy preservation level to DP-based method (JoPEQ). It is seen that the MSE scores of MSP+Lap exhibit consistency, and are slightly higher than MSE scores of JoPEQ. Intuitively, this is because the Laplace noise injected via model splitting can be amplified by the discrepancies among local submodels during the dynamics as in \eqref{decomposition1a} and \eqref{decomposition1b}.
 
We further test the learning accuracy of the abovementioned privacy-preserving methods under identical hyper-parameter settings to see the privacy-accuracy trade-off. The simulation results presented in Fig. 6a-6d, especially the blue curves (MSPDQ-FL with $\epsilon=0.3, B=4$ and JoPEQ with $\epsilon_{DP}=1$), show that our proposed model splitting method does not suffer as severe learning accuracy degradation as traditional noise-adding methods while maintaining comparable privacy protection effects. For a fair comparison, the results in Fig. 6d show that MSP+Lap also achieves similar or better learning accuracy to JoPEQ across various $\epsilon_{DP}$ values. Notably, as illustrated in Fig. 6d, learning accuracy of MSP+Lap remains largely unaffected even when noise intensity increases to a privacy budget of $\epsilon_{DP}=1$. This robustness stems from the zero-sum noise requirement in model splitting mechanism and the well-designed dynamics for the interaction of submodels, which mitigates the privacy-utility trade-off in the conventional noise-adding methods.

\begin{figure}[t]
    \subfloat[$K_t=10, \gamma_{\max}=0.1$ ]{
    \begin{minipage}[t]{0.5\linewidth}
    \centering
    \hspace{-12pt}\includegraphics[width=1.81in,height=1.3in]{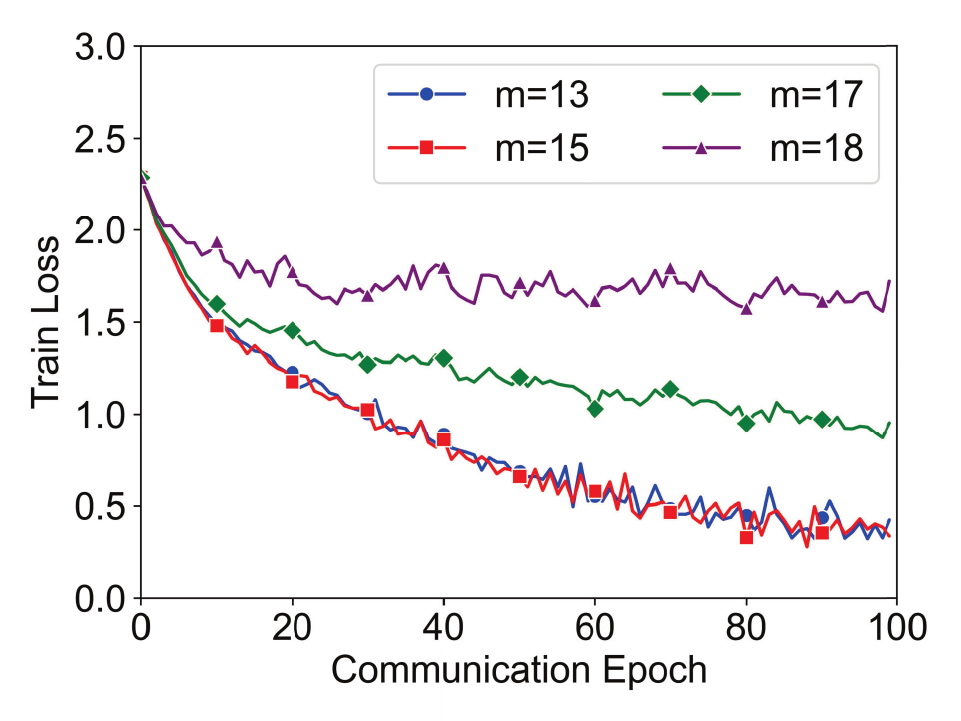}
    \label{k10_0.1varying_m_loss}
    \end{minipage}%
    }%
    \subfloat[$K_t=10, \gamma_{\max}=0.1$ ]{
    \begin{minipage}[t]{0.5\linewidth}
    \centering
    \hspace{-12pt}\includegraphics[width=1.81in,height=1.3in]{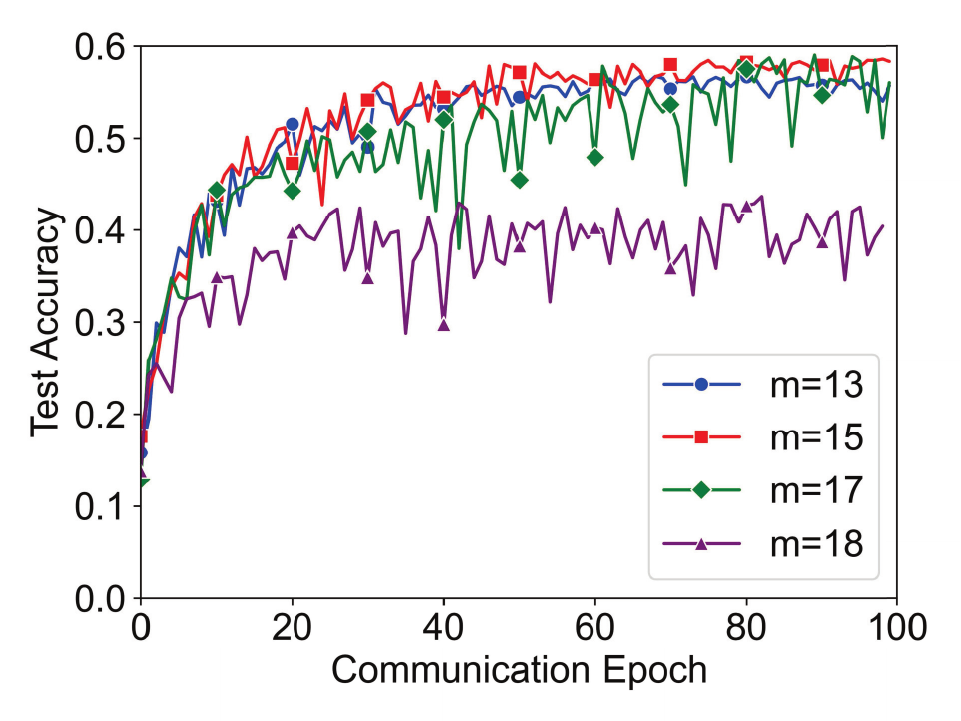}
    \label{k10_0.1varying_m_accuracy}
    \end{minipage}
    }\vspace{0pt}\\
    \subfloat[$K_t=4, \gamma_{\max}=0.1$ ]{
    \begin{minipage}[t]{0.5\linewidth}
    \centering
    \hspace{-12pt}\includegraphics[width=1.81in,height=1.31in]{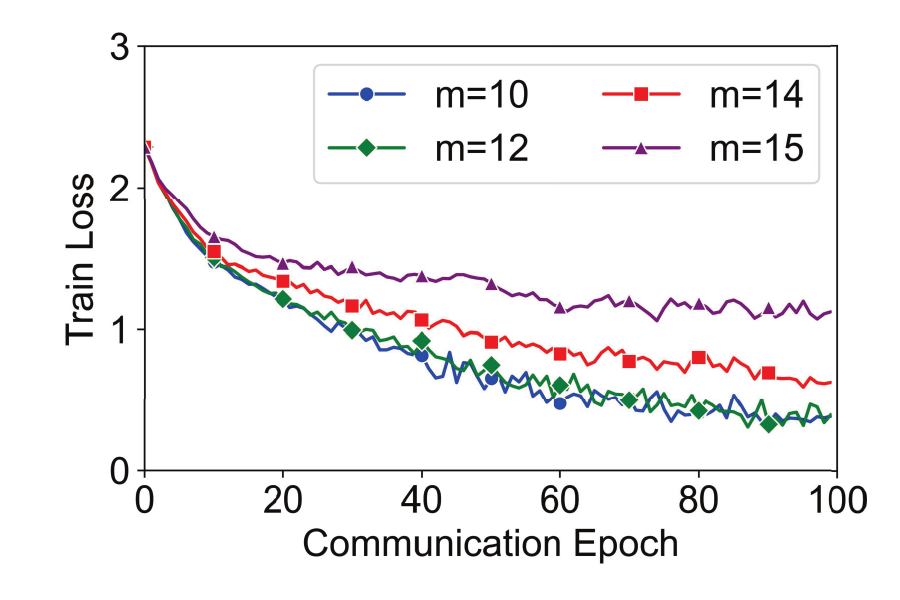}
    \label{k4_0.1varying_m_loss}
    \end{minipage}%
    }%
    \subfloat[$K_t=4, \gamma_{\max}=0.1$]{
    \begin{minipage}[t]{0.5\linewidth}
    \centering
    \hspace{-12pt}\includegraphics[width=1.81in,height=1.31in]{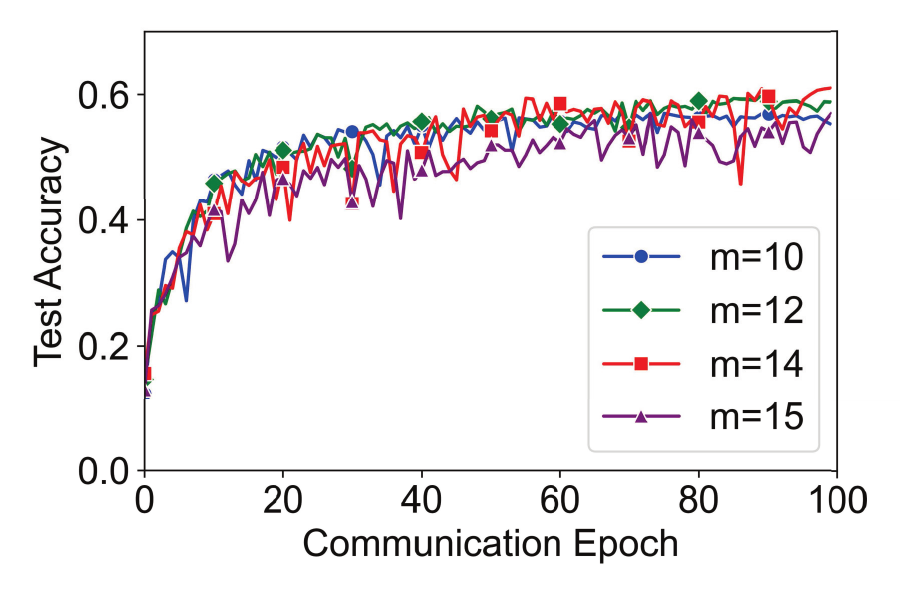} \label{k4_0.1varying_m_accuracy}
    \end{minipage}
    }\\
    \subfloat[$K_t=10, \gamma_{\max}=0.3$ ]{
    \begin{minipage}[t]{0.5\linewidth}
    \centering
    \hspace{-12pt}\includegraphics[width=1.81in,height=1.3in]{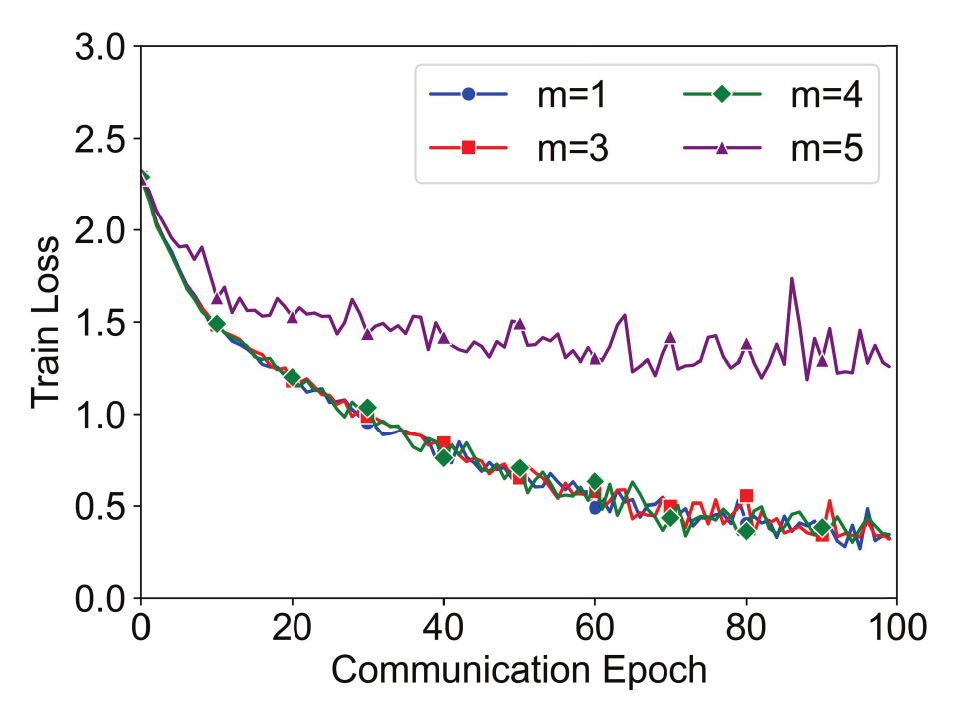}
    \label{k10_0.3varying_m_loss}
    \end{minipage}%
    }%
    \subfloat[$K_t=10, \gamma_{\max}=0.3$ ]{
    \begin{minipage}[t]{0.5\linewidth}
    \centering
    \hspace{-12pt}\includegraphics[width=1.81in,height=1.3in]{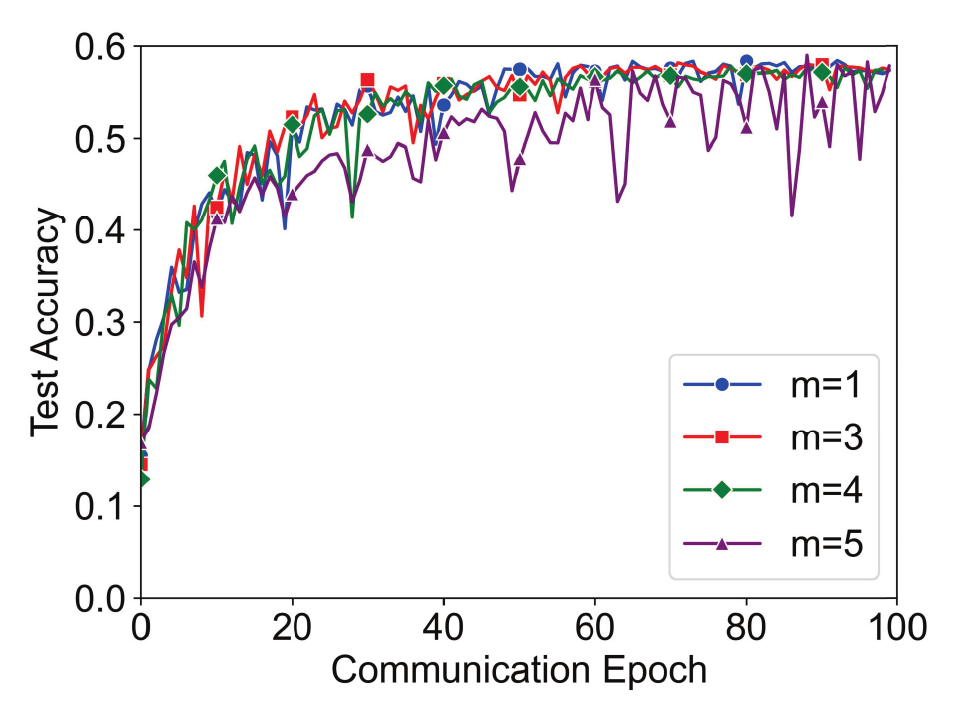}
    \label{k10_0.3varying_m_accuracy}
    \end{minipage}
    }
    \caption{Comparing the learning performance of MSP-FL on CIFAR-10 dataset with varying sizes of invisible submodels.}\label{varying_m}
\end{figure}
{\textit{4) Impact of the number of invisible submodels $m_i$:} Our subjective speculation for the impact is that, with the premise of ensuring the learning accuracy, the value of $m$ affects the least number of communication rounds $K_t$ and the interaction weight $a_{i,\alpha\beta_n}[k]$ for all $i\in\mathcal{S}_t$, $n\in\{1,2,...,m_i\}$. Given a fixed communication round $K_t$ and the parameter $\gamma_{\max}\overset{\triangle}{=}\max\{\gamma_1,...,\gamma_N\}$ (which is a key parameter in the formulation of interaction weights $a_{i,\alpha\beta_n}[k]$), Fig. \ref{k10_0.1varying_m_loss} and Fig. \ref{k10_0.1varying_m_accuracy} report the learning performance of MSP-FL with varying sizes of invisible submodels. As the proportion of invisible submodels increases, the results indicate an apparent reduction in learning performance when the maximal number of invisible submodels reaches a critical point at $m=18$. That is, the global model quality is overly degraded if $m\geq 18$ under the current settings. Setting $K_t=10$ and $\gamma_{\max}=0.1$ as a benchmark, we further explore the impact of varying invisible submodel sizes $m$ on the communication rounds $K_t$ and interaction weight $a_{i,\alpha\beta_n}[k]$. We can observe from Fig. \ref{k4_0.1varying_m_loss} and Fig. \ref{k4_0.1varying_m_accuracy} that the maximal number of invisible submodels currently allowed is close to that in Fig. \ref{k10_0.1varying_m_loss} and Fig. \ref{k10_0.1varying_m_accuracy} with $K_t=10$. This is because $K_t=4$ is nearly the least number of communication rounds that suffices for the aggregation error to be suppressed. In Fig. \ref{k10_0.3varying_m_loss} and Fig. \ref{k10_0.3varying_m_accuracy}, we increase the key parameter $\gamma_{\max}$ in the interaction weight $a_{i,\alpha\beta_n}[k]$ for $i\in\mathcal{S}_t$ to $0.3$, while keeping $K_t=10$. In this case, the maximal number of invisible submodels allowed significantly decreases to $m=4$. The result is consistent with the discussion in Section \uppercase\expandafter {\romannumeral5-C}. Similar results can be obtained for MSPDQ-FL with varying invisible submodel sizes.

\section{Conclusion}
This paper introduced an MSPDQ algorithm to solve both of privacy and communication efficiency issues in FL. Due to the incompletely randomized property of model splitting method, a new privacy notion is used to evaluate the privacy degree. Model splitting method averts the privacy-utility tradeoff via dedicated designs for correlated noises and interaction protocol for submodels. Moreover, in contrast to LDP mechanism (e.g., JoPEQ), our proposed scheme not only saves communication bandwidth benefiting from dynamic quantization, but also reduces the total communication rounds under the same learning accuracy. The theoretical results showed that both the proposed MSP-FL and MSPDQ-FL preserve the privacy, and ensure the convergence to the optimal model in the sense of expectation with an $\mathcal{O}(1/t)$ rate. In addition, we explored sufficient conditions for the DP guarantee under model splitting mechanism, and proved that the proposed dynamic quantizer in MSPDQ-FL achieves at least $(0,\delta)$-differential privacy. Future works include analyzing the tradeoff between the privacy level and communication cost, and extending the theoretical results without strong convexity and bounded gradients assumptions.

\appendices

\section{Proof of Theorem 2}\label{A}

\subsection{Notations}
The updating rule \eqref{decomposition1} is rewritten in a compact form as
\begin{equation}\label{notation1}
  \mathbf{W}_{t}[k+1]=\mathbf{P}[k]\mathbf{W}_{t}[k]
\end{equation}
where
\begin{equation*}\label{notation2}
  \mathbf{W}_{t}[k]
  \hspace{-2pt}=\hspace{-2pt}\left(\hspace{-3pt}
    \begin{array}{c}
     \mathbf{W}_{t}^1[k]^{\top} \\
       \vdots\\
      \mathbf{W}_{t}^{2M}[k]^{\top}
    \end{array}
  \right)
 \hspace{-2pt} =\hspace{-2pt}\left(
    \begin{array}{c}
     \text{col}( \mathbf{w}_{t}^{i,\alpha}[k]^{\top}|{i\in\mathcal{S}_{t}} )\\
     \text{col}( \mathbf{w}_{t}^{i,\beta_1}[k]^{\top}|{i\in\mathcal{S}_{t}} )
    \end{array}
  \right),
\end{equation*}

\begin{equation}\label{notation3}
    \mathbf{P}[k]=\left[
        \begin{array}{cc}
          \mathbf{U}-\mathbf{A}[k] & \mathbf{A}[k]  \\
           \mathbf{A}[k]  & \mathbf{I}-\mathbf{A}[k]
        \end{array}
    \right],
\end{equation}

  \begin{align}\label{notation4}
  {U}_{ij}=\left\{
  \begin{array}{cc}
    \frac{\epsilon}{M}, & \text{if} \; j\neq i, \\
    1-\frac{\epsilon(M-1)}{M}, & \text{otherwise,}
  \end{array}\right.
  \end{align}
$\mathbf{A}[k]=\text{diag}(a_{i,\alpha\beta_1}[k]|{i\in\mathcal{S}_{t}})$. In the above definitions, with a slight abuse of notation, we use $\mathbf{W}_{t}^i[k]^{\top}$ to denote a row vector in the $i$-th row of matrix $\mathbf{W}_{t}[k]$.

The overall timeline is set according to the SGD iteration step instead of the communication round. For each $i\in\mathcal{S}_t$, the procedure in \eqref{notation1}, that consists of a total of $K_t$ communication rounds between the server and $i$, is viewed as an integrated communication step and denoted as ${\text{Aggr}}(\mathbf{w}_{t}^i)$. Then, after such an integrated communication step, we rewrite the last aggregation process in \eqref{aggregation3} as $\mathbf{w}_{t}=\sum_{i\in\mathcal{S}_{t}}\frac{1}{M}{\text{Aggr}}(\mathbf{w}_{t}^i)$. Let $\mathcal{I}_E=\{nE|n=1,2,...\}$ denote the set of timestamps on the overall timeline for such an integrated communication step. Then, the dynamics under MSP-FL can be written as
\begin{subequations}\label{msp-fl}
    \begin{align}
        \mathbf{r}_{t+1}^i&=\mathbf{w}_{t}^i-\eta_{t}\nabla F_i(\mathbf{w}_{t}^i,\bm{\zeta}_{t}^i),\label{msp-fla}\\
        \mathbf{u}_{t+1}^i&=\left\{
          \begin{aligned}
            &\mathbf{r}_{t+1}^i, \qquad\qquad\qquad\qquad\text{if}\quad t+1\notin \mathcal{I}_E, \\
            &\sum\nolimits_{i\in S_{t+1}}\frac{1}{M}\mathbf{r}_{t+1}^i, \quad\,\,\,\;\qquad\text{otherwise},
          \end{aligned}\right.\label{msp-flb}\\
          \mathbf{w}_{t+1}^i&=\left\{
          \begin{aligned}
            &\mathbf{r}_{t+1}^i, \qquad\qquad\qquad\qquad\text{if}\quad t+1\notin \mathcal{I}_E, \\
            &\sum\nolimits_{i\in S_{t+1}}\frac{1}{M}\text{Aggr}(\mathbf{r}_{t+1}^i), \quad\,\text{otherwise},
          \end{aligned}\right.\label{msp-flc}
    \end{align}
\end{subequations}
where $\mathbf{r}_{t+1}^i$ records the result of one step SGD. Under partial clients selection, if $t+1\in\mathcal{I}_E$, $\mathbf{u}_{t+1}^i$ denotes the local model after one round communication for aggregation, and $\mathbf{w}_{t+1}^i$ denotes that after the integrated communication step. In the dynamics \eqref{msp-fl}, the randomness from stochastic gradients, random sampling of clients, stochastic splitting of model and stochastic quantization is represented by $\mathbb{E}_{SG}$, $\mathbb{E}_{RS}$, $\mathbb{E}_{SS}$ and $\mathbb{E}_{SQ}$, respectively. For the subsequent analysis, we additionally define $\overline{\mathbf{r}}_{t+1}=\sum_{i=1}^{N}p_i\mathbf{r}_{t+1}^i$, $\overline{\mathbf{u}}_{t+1}=\sum_{i=1}^{N}p_i\mathbf{u}_{t+1}^i$, $\overline{\mathbf{w}}_{t+1}=\sum_{i=1}^{N}p_i\mathbf{w}_{t+1}^i$,
$\overline{\mathbf{g}}_{t}=\sum_{i=1}^{N}p_i\nabla F_i(\mathbf{w}_{t}^i)$
and $\mathbf{g}_{t}=\sum_{i=1}^{N}p_i\nabla F_i(\mathbf{w}_{t}^i,\bm{\zeta}^i)$.
Due to the unbiased property of the clients selection scheme for model aggregation shown in Lemma 1, we have the following results at $t+1\in\mathcal{I}_E$:\par
\vspace{-10pt}
\begin{small}
\begin{eqnarray}
  \overline{\mathbf{u}}_{t+1}\hspace{-7pt}&=&\hspace{-7pt}\sum\nolimits_{i=1}\nolimits^{N}p_i\mathbf{u}_{t+1}^i=\mathbb{E}\Big\{\sum\nolimits_{i\in\mathcal{S}_{t+1}}\frac{1}{M}\mathbf{r}_{t+1}^i\Big\},\label{22_0}\\
  \overline{\mathbf{w}}_{t+1}\hspace{-7pt}&=&\hspace{-7pt}\sum\nolimits_{i=1}\nolimits^{N}p_i\mathbf{w}_{t+1}^i=\mathbb{E}\Big\{\sum\nolimits_{i\in\mathcal{S}_{t+1}}\frac{1}{M}\text{Aggr}(\mathbf{r}_{t+1}^i)\Big\}.\label{22_1}
\end{eqnarray}
\end{small}%

\subsection{Supporting Lemmas}
To present Theorem 3, it is crucial to display a few lemmas, where Lemmas 3-6 are established in \cite{niid}.

\begin{lemma}
  (Result of one step SGD) Suppose Assumptions 1.1-1.2 hold, and $\eta_t\leq\frac{1}{4L}$ for all $t$, then it has
  \begin{small}
  \begin{eqnarray*}
   \mathbb{E}_{SG}\left\{\|\overline{\mathbf{r}}_{t+1}-\mathbf{w}^*\|^2\right\}\hspace{-7pt}&\leq&\hspace{-7pt} (1-\eta_{t}\mu)\mathbb{E}_{SG}\left\{\|\overline{\mathbf{w}}_{t}-\mathbf{w}^*\|^2\right\}+6L\eta_{t}^2\Gamma\\
   \hspace{-7pt}&&\hspace{-75pt} +\eta_{t}^2\mathbb{E}_{SG}\left\{\|\mathbf{g}_{t}-\overline{\mathbf{g}}_{t}\|^2\right\}
   +2\mathbb{E}_{SG}\Big\{\frac{1}{N}\sum\nolimits_{i=1}\nolimits^{N}\|\overline{\mathbf{w}}_{t}-\mathbf{w}_{t}^i\|^2\Big\}.
\end{eqnarray*}
\end{small}%
\end{lemma}

\begin{lemma}
 (Bound on the variance of $\mathbf{g}_{t}$) Suppose Assumption 1.3 holds, it has
   $\mathbb{E}_{SG}\left\{\|\mathbf{g}_{t}-\overline{\mathbf{g}}_{t}\|^2\right\} \leq \sum_{i=1}^{N}p_i^2\sigma_i$.
\end{lemma}

\begin{lemma}
  (Bound on the divergence of $\mathbf{w}_{t}^i$) If Assumption 1.4 holds, $\eta_t$ is non-increasing and satisfies $\eta_t\leq2\eta_{t+E}$ for all $t\geq0$, it has
  $\mathbb{E}_{SG}\left\{\sum_{i=1}^{N}p_i\|\overline{\mathbf{w}}_{t}-\mathbf{w}_{t}^i\|^2\right\} \leq 4\eta_{t}^2(E-1)^2G$.
\end{lemma}

\begin{lemma}
    (Unbiased sampling and bound on the variance of $\overline{\mathbf{u}}_{t+1}$) If $\eta_t$ is non-increasing and satisfies $\eta_t\leq2\eta_{t+E}$ for all $t>0$, then for $t+1\in\mathcal{I}_E$, it has that
    \begin{subequations}\label{lemma4}
    \begin{eqnarray}
        \mathbb{E}_{RS}\left\{\overline{\mathbf{u}}_{t+1}\right\}\hspace{-7pt}&=&\hspace{-7pt}\overline{\mathbf{r}}_{t+1},\label{lemma4a}\\
        \mathbb{E}_{RS}\left\{\|\overline{\mathbf{u}}_{t+1}-\overline{\mathbf{r}}_{t+1}\|^2\right\}\hspace{-7pt}&\leq&\hspace{-7pt}\frac{4}{M}\eta_{t}^2E^2G.\label{lemma4b}
    \end{eqnarray}
    \end{subequations}
\end{lemma}

\begin{lemma}
    If $\epsilon\in(0,\frac{M}{M-1})$ and ${a_{\alpha\beta_1}^{\max}}\in (0, \frac{\lambda_{\min,\mathbf{U}}}{1+\lambda_{\min,\mathbf{U}}})$, then for some positive constant $C$ and $\lambda\in(|\lambda_{2,\mathbf{\Phi}(k,0)}|,1)$, it holds that
$
    \left|[\mathbf{\Phi}(k,0)]_{ij}-\frac{1}{2M}\right|\leq C\lambda^{k+1}
$ for all $i$, $j$ and $k$,
  where $\mathbf{\Phi}(k,0)=\mathbf{P}[k]\mathbf{P}[k-1]\ldots \mathbf{P}[0]$.
\end{lemma}

\begin{proof}
  Explicit statement of Proposition 1 in \cite{van}.
\end{proof}

\begin{lemma}
  (Unbiased stochastic splitting, and bounded error of aggregation with visible submodels) For $t+1\in\mathcal{I}_E$, if Assumption 2 holds, $\epsilon\in(0,\frac{M}{M-1})$, ${a_{\alpha\beta_1}^{\max}}\in (0, \frac{\lambda_{\min,\mathbf{U}}}{1+\lambda_{\min,\mathbf{U}}})$ and $K_t=\lceil\log_{\lambda}^{\eta_t}\rceil$, we have that
   $\mathbb{E}\left\{\overline{\mathbf{w}}_{t+1}\right\} = \overline{\mathbf{u}}_{t+1}$ and
   $\mathbb{E}\left\{\|\overline{\mathbf{w}}_{t+1}-\overline{\mathbf{u}}_{t+1}\|^2\right\}  \leq\eta_{t}^2 D_1$,
where $D_1=\frac{(2\epsilon^2-4\epsilon+8)}{3}C^2\|\mathbf{w}_{\max}\|^2$.
\end{lemma}

\begin{proof}
    Define the mean of the matrices $\mathbf{P}[k]$ as $\overline{\mathbf{P}}$. In terms of the independence of matrices $\mathbf{P}[k]$ for all $k$, it has that
$\mathbb{E}\left\{\mathbf{\Phi}(k,0)\right\}=\prod_{t=0}^{k}\mathbb{E}\left\{\mathbf{P}[t]\right\}=\overline{\mathbf{P}}^{k+1}$.
    It is not hard to see that the matrix $\overline{\mathbf{P}}$ is doubly stochastic, then $\lim_{k\rightarrow\infty}\overline{\mathbf{P}}^{k}=\frac{1}{2M}\mathbf{1}\mathbf{1}^{\top}$. That is to say,
 $\mathbf{\Phi}(k,0)$ converges in expectation to $\frac{1}{2M}\mathbf{1}\mathbf{1}^{\top}$. Accordingly, it yields\par
 \vspace{-11pt}
 \begin{small}
\begin{eqnarray}\label{lemma8_1}
  \mathbb{E}\left\{\text{Aggr}(\mathbf{r}_{t+1}^i)\right\} \hspace{-6pt}&=&\hspace{-6pt}\mathbb{E}\Bigg\{\Big(\sum_{j=1}^{2M}[\mathbf{\Phi}(k,0)]_{ij} \mathbf{W}_{t+1}^j[0]\Big)^{\top}\Bigg\}\nonumber\\
    \hspace{-6pt}&=&\hspace{-6pt}\frac{1}{2M}\mathbb{E}_{SS}\left\{\mathbf{W}_{t+1}[0]^{\top}\right\}\mathbf{1},
\end{eqnarray}
\end{small}%
where the second equality is obtained for the independence of expectation between $\mathbb{E}\left\{\mathbf{\Phi}(k,0)\right\}$ and $\mathbb{E}_{SS}\left\{\mathbf{W}_{t+1}[0]\right\}$.

Note that, at $t+1\in\mathcal{I}_E$, $\mathbf{w}_{t+1}^{i,\alpha}[0]$ is randomly drawn from a uniform distribution on the interval $[\varepsilon \mathbf{w}_{t+1}^i,(2-\varepsilon) \mathbf{w}_{t+1}^i]$ by the initial setting, where $\mathbf{w}_{t+1}^i$ is the local training model after multi-step SGD and represented by $\mathbf{r}_{t+1}^i$ as in \eqref{msp-fla}. Then, for any $i\in\mathcal{S}_{t+1}$, it has that\par
 \vspace{-9pt}\begin{small}
\begin{equation}\label{lemma8_2}
\begin{aligned}
   \mathbb{E}_{SS}\left\{\mathbf{W}_{t+1}^{[i]}[0]^{\top}\right\}=\hspace{1.5pt}&\mathbb{E}_{SS}\left\{\mathbf{W}_{t+1}^{[i]+M}[0]^{\top}\right\}= \mathbf{r}_{t+1}^{i}.
\end{aligned}
\end{equation}
\end{small}%
By plugging \eqref{lemma8_1} and \eqref{lemma8_2} into \eqref{22_1}, it yields
\begin{small}
\begin{eqnarray}\label{22}
  \mathbb{E}\{\overline{\mathbf{w}}_{t+1}\}
  \hspace{-7pt}&=& \hspace{-7pt}\mathbb{E}\bigg\{\sum_{i\in\mathcal{S}_{t+1}}\frac{1}{M}\text{Aggr}(\mathbf{r}_{t+1}^i)\bigg\}\nonumber\\
  \hspace{-7pt}&=& \hspace{-7pt}\mathbb{E}\bigg\{\frac{1}{M}\sum_{i\in\mathcal{S}_{t+1}}\frac{1}{M}\sum_{j\in\mathcal{S}_{t+1}}\mathbf{r}_{t+1}^j\bigg\}\nonumber\\
  \hspace{-7pt}&\overset{(a)}{=}& \hspace{-7pt}\mathbb{E}\bigg\{ \frac{1}{M}\sum_{i\in\mathcal{S}_{t+1}}\mathbf{u}_{t+1}^i \bigg\}\overset{(b)}{=}\overline{\mathbf{u}}_{t+1},
\end{eqnarray}
\end{small}%
where the equality (a) follows from \eqref{msp-flb}. The equality (b)  is obtained due to the fact of unbiased aggregation as in \eqref{22_0}, i.e., $\overline{\mathbf{u}}_{t+1}=\sum_{i=1}^{N}p_i\mathbf{u}_{t+1}^i=\mathbb{E}\{\sum_{i\in\mathcal{S}_{t+1}}\frac{1}{M}\mathbf{u}_{t+1}^i\}$.

 Let $\text{Aggr}^*(\cdot)$ denote the limit state of the integrated communication step $\text{Aggr}(\cdot)$ when $k\rightarrow\infty$. By \eqref{22_0} and \eqref{22_1}, and resorting to the fact that $\sum_{i\in\mathcal{S}_{t+1}}\frac{1}{M}\mathbf{r}_{t+1}^i=\sum_{i\in\mathcal{S}_{t+1}}\frac{1}{M}\text{Aggr}^*(\mathbf{r}_{t+1}^i)$, it has
    \begin{eqnarray}\label{lemma621}
    \hspace{-26pt}&&\hspace{-7pt}\mathbb{E}\left\{\|\overline{\mathbf{w}}_{t+1}-\overline{\mathbf{u}}_{t+1}\|^2\right\} \nonumber\\
    \hspace{-26pt}&=&\hspace{-7pt}\mathbb{E}\Bigg\{\Big\|\sum_{i\in\mathcal{S}_{t+1}}\frac{1}{M}\text{Aggr}(\mathbf{r}_{t+1}^i) \hspace{-2pt}-\hspace{-4pt}\sum_{i\in\mathcal{S}_{t+1}}\hspace{-2pt}\frac{1}{M}\text{Aggr}^*(\mathbf{r}_{t+1}^i)\Big\|^2\Bigg\}.
   \end{eqnarray}
Meanwhile, when $k\rightarrow\infty$ the dynamics \eqref{notation1} achieves average consensus, which implies that\par
\vspace{-11pt}\begin{small}
   \begin{eqnarray}\label{lemma622}
     \hspace{-27pt}&&\hspace{-7pt}\text{Aggr}^*(\mathbf{r}_{t+1}^i)=\frac{1}{2M}\mathbf{W}_{t+1}[0]^{\top}\mathbf{1}.
   \end{eqnarray}
\end{small}%
By virtue of the Schur Compliment, the matrix $\mathbf{P}[k]$ in \eqref{notation1} is positive semidefinite with conditions $\epsilon\in(0,\frac{M}{M-1})$ and ${a_{\alpha\beta_1}^{\max}}\in (0, \frac{\lambda_{\min,\mathbf{U}}}{1+\lambda_{\min,\mathbf{U}}})$. Then, with Assumption 2, inserting \eqref{lemma622} into \eqref{lemma621} leads to the result as shown in \eqref{bottom} at the top of the next page.
\newcounter{TempEqCnt}
\setcounter{TempEqCnt}{\value{equation}}
\begin{figure*}[htbp]
 \begin{small}
 \begin{eqnarray}\label{bottom}
    \hspace{-10pt} \mathbb{E}\left\{\left\|\overline{\mathbf{w}}_{t+1}-\overline{\mathbf{u}}_{t+1}\right\|^2\right\} \hspace{-7pt}&=&\hspace{-7pt}\mathbb{E}\left\{\bigg\| \sum_{i\in\mathcal{S}_{t+1}}\frac{1}{M}\sum_{j=1}^{2M}\left[\Phi(k,0)\right]_{ij}\mathbf{W}_{t+1}^j[0]^{\top}-\sum_{i\in\mathcal{S}_{t+1}}\frac{1}{M}\sum_{j=1}^{2M}\frac{1}{2M}\mathbf{W}_{t+1}^j[0]^{\top} \bigg\|^2\right\}\nonumber\\
    \hspace{-7pt}&=&\hspace{-7pt}\frac{1}{M^2}\sum_{i\in\mathcal{S}_{t+1}}\mathbb{E}\left\{\left\|\sum_{j=1}^{2M} \left(\left[\mathbf{\Phi}(k,0)\right]_{ij}-\frac{1}{2M}\right)\mathbf{W}_{t+1}^j[0]^{\top}\right\|^2\right\}\nonumber\\
    \hspace{-7pt}&&\hspace{-7pt} +\mathbb{E}\Bigg\{\sum_{h,i,j\in\mathcal{S}_{t+1}\atop h\neq i\neq j}\hspace{-3pt}\left\langle \frac{1}{M} \sum_{j=1}^{2M}\left(\left[\mathbf{\Phi}(k,0)\right]_{ij}-\frac{1}{2M}\right)\mathbf{W}_{t+1}^j[0]^{\top},
    \frac{1}{M}\sum_{i=1}^{2M} \Big([\mathbf{\Phi}(k,0)]_{hi}-\frac{1}{2M} \Big)\mathbf{W}_{t+1}^i[0]^{\top}\right\rangle\Bigg\}\nonumber\\
    \hspace{-7pt}&\leq&\hspace{-7pt}\frac{1}{M^2}\sum_{i\in\mathcal{S}_{t+1}}\sum_{j=1}^{2M}\left\|\left[\mathbf{\Phi}(k,0)\right]_{ij}-\frac{1}{2M}\right\|^2\mathbb{E}_{SS}\left\{\left\|\mathbf{W}_{t+1}^j[0]\right\|^2\right\}\leq \frac{(2\epsilon^2-4\epsilon+8)}{3}C^2\|\mathbf{w}_{\max}\|^2\lambda^{2(k+1)},
  \end{eqnarray}
  \end{small}%
  \hrulefill
  \end{figure*}
Further, if the lower bound of communication rounds for update in \eqref{notation1} is a dynamic value, and designed as $k+1=K_t=\lceil\log_\lambda^{\eta_t}\rceil$,
   we arrive at\par
   \vspace{-5pt}\begin{small}
   \begin{equation}\label{lemma6lasteq}
     \mathbb{E}\left\{\|\overline{\mathbf{w}}_{t+1}-\overline{\mathbf{u}}_{t+1}\|^2\right\}\leq \eta_t^2 D_1.
   \end{equation}
   \end{small}%

This completes the proof.
\end{proof}

\subsection{Proof of Theorem 2}
\begin{proof}
To see the gap between $\mathbb{E}\left\{F(\overline{\mathbf{w}}_t)\right\}-F^*$, we turn to analyze the variance of $\overline{\mathbf{w}}_{t+1}-\mathbf{w}^*$. Notice that\par
\vspace{-10pt}
\begin{small}
\begin{eqnarray}\label{start}
\|\overline{\mathbf{w}}_{t+1}-\mathbf{w}^*\|^2
   \hspace{-7pt}&=&\hspace{-7pt}\underbrace{\|\overline{\mathbf{w}}_{t+1}-\overline{\mathbf{u}}_{t+1}\|^2}_{\text{T}_{01}}+\underbrace{\|\overline{\mathbf{u}}_{t+1}-\mathbf{w}^*\|^2}_{\text{T}_{02}}\nonumber\\
   \hspace{-7pt}&&\hspace{-7pt}+\underbrace{2\langle  \overline{\mathbf{w}}_{t+1}-\overline{\mathbf{u}}_{t+1},\overline{\mathbf{u}}_{t+1}-\mathbf{w}^* \rangle}_{\text{T}_{03}}.
\end{eqnarray}
\end{small}%
If $t+1\notin \mathcal{I}_E$, it has $\overline{\mathbf{w}}_{t+1}=\overline{\mathbf{u}}_{t+1}=\overline{\mathbf{r}}_{t+1}$. Then, both $\text{T}_{01}$ and $\text{T}_{03}$ vanish in expectation, leading to\par
\vspace{-10pt}
\begin{small}
\begin{eqnarray}\label{t+1not}
   \hspace{-17pt}\mathbb{E}\left\{\|\overline{\mathbf{w}}_{t+1}-\mathbf{w}^*\|^2\right\}
   \hspace{-7pt}&\leq&\hspace{-7pt} (1-\eta_t\mu)\mathbb{E}\left\{\|\overline{\mathbf{w}}_t-\mathbf{w}^*\|^2\right\}\nonumber\\
   \hspace{-17pt} \hspace{-7pt}&&\hspace{-40pt} +\eta_t^2\Big[ \sum\nolimits_{i=1}\nolimits^{N}p_i^2\sigma_i+6L\Gamma+8(E-1)^2G\Big],
\end{eqnarray}
\end{small}%
where the inequality follows from Lemmas 3-5. Please refer to \cite{niid} for detailed proof.

If $t+1\in\mathcal{I}_E$, $\mathbb{E}\{\text{T}_{01}\}$ is bounded as Lemma 8 presents. $\text{T}_{02}$ can be rewritten as
\begin{eqnarray}\label{27_1}
   \|\overline{\mathbf{u}}_{t+1}-\mathbf{w}^*\|^2
   \hspace{-7pt}&=&\hspace{-7pt}\underbrace{\|\overline{\mathbf{u}}_{t+1}-\overline{\mathbf{r}}_{t+1}\|^2}_{\text{T}_{11}}+\underbrace{\|\overline{\mathbf{r}}_{t+1}-\mathbf{w}^*\|^2}_{\text{T}_{12}}\nonumber\\
   \hspace{-7pt}&&\hspace{-7pt}+\underbrace{2\langle \overline{\mathbf{u}}_{t+1}-\overline{\mathbf{r}}_{t+1},\overline{\mathbf{r}}_{t+1}-\mathbf{w}^* \rangle}_{\text{T}_{13}},
\end{eqnarray}
where $\mathbb{E}\{\text{T}_{12}\}$ can be bounded as in \eqref{t+1not}. $\mathbb{E}\{\text{T}_{13}\}$ vanishes due to the unbiased sampling, which is shown in Lemma 6. Inserting \eqref{lemma4} and \eqref{t+1not} into \eqref{27_1} gives rise to\par
\vspace{-11pt}
\begin{small}
\begin{eqnarray}\label{27}
  \hspace{-18pt} \mathbb{E}\{{\text{T}_{02}}\}\hspace{-7pt}&\leq&\hspace{-7pt}(1-\eta_t\mu)\mathbb{E}\left\{\|\overline{\mathbf{w}}_t-\mathbf{w}^*\|^2\right\}\nonumber\\
    \hspace{-7pt}&&\hspace{-30pt} +\eta_t^2\Big[ \sum\nolimits_{i=1}\nolimits^{N}p_i^2\sigma_i+6L\Gamma+8(E-1)^2G+ \frac{4}{M}E^2G\Big].
\end{eqnarray}
\end{small}%
With regard to $\text{T}_{03}$, it also vanishes in expectation since $\mathbb{E}\{\overline{\mathbf{w}}_{t+1}\}=\overline{\mathbf{u}}_{t+1}$. Now, by substituting \eqref{27} into \eqref{start} and invoking \eqref{lemma6lasteq}, it boils down to\par
 \vspace{-12pt}\begin{small}
\begin{eqnarray}\label{t+1in}
  \hspace{-30pt}&&\hspace{-7pt}\mathbb{E}\left\{\|\overline{\mathbf{w}}_{t+1}-\mathbf{w}^*\|^2\right\} \leq(1-\eta_t\mu)\mathbb{E}\{\|\overline{\mathbf{w}}_t-\mathbf{w}^*\|^2\}\nonumber\\
    \hspace{-30pt}&&\hspace{-7pt}\quad +\eta_t^2\Big[ D_1\hspace{-1pt}+\hspace{-1pt} \sum\nolimits_{i=1}\nolimits^{N}p_i^2\sigma_i\hspace{-1pt}+\hspace{-1pt}6L\Gamma\hspace{-1pt}+\hspace{-1pt}8(E-1)^2G\hspace{-1pt}+\hspace{-1pt}\frac{4}{M}E^2G\Big].
\end{eqnarray}
\end{small}%

At this stage, one can observe from \eqref{t+1not} and \eqref{t+1in} that no matter whether $t+1\in\mathcal{I}_E$ or not, it always holds that
\begin{eqnarray}\label{fixediteration}
  \hspace{-13pt}\mathbb{E}\left\{\|\overline{\mathbf{w}}_{t+1}\hspace{-1pt}-\hspace{-1pt}\mathbf{w}^*\|^2\right\} \hspace{-8pt}&\leq&\hspace{-8pt}(1\hspace{-1pt}-\hspace{-1pt}\eta_t\mu)\mathbb{E}\left\{\|\overline{\mathbf{w}}_{t}\hspace{-1pt}-\hspace{-1pt}\mathbf{w}^*\|^2\right\}
   \hspace{-1pt}+\hspace{-1pt}\eta_t^2D_2\hspace{-1pt},
\end{eqnarray}
where $D_2=D_1+\sum_{i=1}^{N}p_i^2\sigma_i+6L\Gamma+8(E-1)^2G+\frac{4}{M}E^2G$.
 Let $\rho>\frac{1}{\mu}$, $\vartheta\geq 0$ and design the learning rate $\eta_t$ as $\frac{\rho}{\vartheta+ t}$ such that $\eta_t<\min \{\frac{1}{\mu},\frac{1}{4L}\}$ and $\eta_t\leq 2\eta_{t+E}$. As a standard technique, we will show that  $\mathbb{E}\left\{\|\overline{\mathbf{w}}_{t+1}-\mathbf{w}^*\|^2\right\}\leq \frac{\nu_1}{\vartheta+t+1}$ via the mathematical induction method, where
$\nu_1=\max\left\{\frac{\rho^2D_2}{\rho\mu-1},(\vartheta+1)\mathbb{E}\{\left\|\overline{\mathbf{w}}_{0}-\mathbf{w}^*\|^2\right\}\right\}$.

First, it comes naturally when $t=0$ from the definition of $\nu_1$. Then, we assume that the claim still holds true for some $t>0$. Based on \eqref{fixediteration}, it has
\begin{eqnarray}
 \mathbb{E}\left\{\|\overline{\mathbf{w}}_{t+1}-\mathbf{w}^*\|^2\right\} \hspace{-7pt} &\leq& \hspace{-7pt} (1-\eta_t\mu)\mathbb{E}\left\{\|\overline{\mathbf{w}}_{t}-\mathbf{w}^*\|^2\right\}+\eta_t^2D_2\nonumber\\
  \hspace{-7pt} &\leq& \hspace{-7pt}(1-\frac{\rho\mu}{\vartheta+ t})\frac{\nu_1}{\vartheta +t}+\frac{\rho^2 D_2}{(\vartheta+t)^2}\nonumber \\
  \hspace{-7pt} \hspace{-7pt}&=&\hspace{-7pt}\frac{\vartheta+t-1}{(\vartheta+t)^2}\nu_1+\frac{\rho^2 D_2}{(\vartheta+t)^2}-\frac{\rho\mu-1}{(\vartheta+t)^2}\nu_1 \nonumber  \\
  \hspace{-7pt} &\leq&\hspace{-7pt}\frac{(\vartheta+t-1)\nu_1}{(\vartheta+t)^2-1}\nonumber\\
  \hspace{-7pt}&=&\hspace{-7pt}\frac{\nu_1}{\vartheta+t+1}
\end{eqnarray}
where the first equality follows by subtracting and adding $\frac{\nu_1}{(\vartheta+t)^2}$ on the right-hand side. The last inequality is derived from the definition of $\nu_1$. Evidently, these collaboratively affirm the claim that $\mathbb{E}\left\{\|\overline{\mathbf{w}}_{t+1}-\mathbf{w}^*\|^2\right\}\leq \frac{\nu_1}{\vartheta+t+1}$.

Next, let $\rho=\frac{2}{\mu}$ and $\vartheta=\max\{\frac{8L}{\mu},E\}-1$, and utilize the following inequalities, i.e.,
$\vartheta+1\leq\frac{8 L}{\mu}+E$ and $\nu_1\leq\frac{\rho^2D_2}{\rho\mu-1}+(\vartheta+1)\mathbb{E}\left\{\|\overline{\mathbf{w}}_{0}-\mathbf{w}^*\|^2\right\}$,
then one can proceed to the relationship of $\mathbb{E}\left\{F(\overline{\mathbf{w}}_{t})\right\}-F^*$ by the $L$-smoothness, i.e.,
\begin{eqnarray}\label{theorem3result}
  \hspace{-7pt}&&\hspace{-7pt}\mathbb{E}\left\{F(\overline{\mathbf{w}}_{t})\right\}-F^*\nonumber\\
  \hspace{-7pt}&\leq&\hspace{-7pt} \frac{L\nu_1}{2(\vartheta+t)}  \nonumber\\
   \hspace{-7pt}&\leq &\hspace{-7pt} \frac{8L}{\mu(\vartheta+t)}\left(\frac{2D_2}{\mu}+\frac{8L+\mu E}{2}\mathbb{E}\left\{\|\overline{\mathbf{w}}_0-\mathbf{w}^*\|^2\right\}\right).
\end{eqnarray}
This ends the proof.
\end{proof}

\section{Proof of Theorem 3}
\subsection{Additional Notations}
Some additional notations are given for the subsequent analysis.
Similarly as in \eqref{msp-fl}, we reformulate the dynamics under MSPDQ-FL and present it below. Unlike the reformulation of MSP-FL shown in \eqref{msp-fl}, $\mathbf{v}_{t+1}^i$ and $\mathbf{w}_{t+1}^i$ denote sequences of integrated communication step with and without quantized submodels, respectively.
\begin{subequations}
    \begin{align}
        \mathbf{r}_{t+1}^i&=\mathbf{w}_{t}^i-\eta_t\nabla F_i(\mathbf{w}_{t}^i,\bm{\zeta}_t^i),\\
        \mathbf{u}_{t+1}^i&=\left\{
          \begin{aligned}
            &\mathbf{r}_{t+1}^i, \qquad\qquad\qquad\qquad\text{if}\quad t+1\notin \mathcal{I}_E, \\
            &\sum\nolimits_{i\in S_{t+1}}\frac{1}{M}\mathbf{r}_{t+1}^i, \quad\,\,\,\;\qquad\text{otherwise},
          \end{aligned}\right.\\
        \mathbf{v}_{t+1}^i&=\left\{
          \begin{aligned}
            &\mathbf{r}_{t+1}^i, \qquad\qquad\qquad\qquad\text{if}\quad t+1\notin \mathcal{I}_E, \\
            &\sum\nolimits_{i\in S_{t+1}}\frac{1}{M}\text{Aggr}(\mathbf{r}_{t+1}^i), \quad\,\text{otherwise},
          \end{aligned}\right.\\
          \mathbf{w}_{t+1}^i&=\left\{
          \begin{aligned}
            &\mathbf{r}_{t+1}^i, \qquad\qquad\qquad\qquad\text{if}\quad t+1\notin \mathcal{I}_E, \\
            &\sum\nolimits_{i\in S_{t+1}}\frac{1}{M}\text{Aggr}(\mathbf{q}_{t+1}^i), \quad\,\text{otherwise}.
          \end{aligned}\right.
    \end{align}
\end{subequations}
Define $\mathbf{W}_{t+1}^{\alpha}[k]=\text{col}(\mathbf{w}_{t+1}^{i,\alpha}[k]^{\top}|{i\in\mathcal{S}_{t+1}})$ (doing likewise with $\mathbf{Q}_{t+1}[k]$ and $\mathbf{\Delta}_{t+1}[k]$), $\mathbf{W}_{t+1}^{\beta_1}[k]=\text{col}(\mathbf{w}_{t+1}^{i,\beta_1}[k]^{\top}|i\in\mathcal{S}_{t+1})$, $\overline{\mathbf{W}}_{t+1}^{\alpha}[k]=\sum_{i\in \mathcal{S}_{t+1}}\frac{1}{M}\mathbf{w}_{t+1}^{i,\alpha}[k]^{\top}$ and
 $\overline{\mathbf{q}}_{t+1}[k]=\sum_{i\in \mathcal{S}_{t+1}}\frac{1}{M}\mathbf{q}_{t+1}^{i}[k]$.

\subsection{Supporting Lemmas}
Some other lemmas are provided for the remaining proof.

\begin{lemma}
Let $\epsilon\in(0,\frac{M}{M-1})$, and $\{a_{i,\alpha\beta_1}[k]|i\in\mathcal{S}_{t}\}$ be  positive non-increasing sequences of step sizes. Under MSPDQ-FL, it has that
\begin{eqnarray}\label{lemma7}
  \hspace{-27pt} &&\hspace{-7pt}\left\| \mathbf{W}_{t}^{\alpha}[k+1]-\mathbf{1}\overline{\mathbf{W}}_{t}^\alpha[k+1] \right\|\nonumber \\
   \hspace{-27pt} &\leq& \hspace{-7pt} 2\sum\nolimits_{l=0}\nolimits^{k}\lambda_{2,\mathbf{U}}^{k-l}\|\mathbf{\Delta}_{t}[l] \|+ \widetilde{W}_{t}^{\max} \sum\nolimits_{l=0}\nolimits^{k}\lambda_{2,\mathbf{U}}^{k-l} a_{\alpha\beta_1}^{\max}[l].
\end{eqnarray}
\end{lemma}
\begin{proof}

For the $k$-th communication round after the $t$-th learning round, let  $\mathbf{X}_t[k]:=\mathbf{\Theta}\mathbf{W}_{t+1}^{\alpha}[k]$ denote the visible submodel deviation from the average state, where $\mathbf{\Theta}=\mathbf{I}-\frac{1}{M}\mathbf{1}\mathbf{1}^{\top}$. To see its upper bound, we firstly arrange $\mathbf{X}_{t}[k+1]$ as follows:
  \begin{eqnarray}\label{35_1}
\mathbf{X}_{t}[k+1]
\hspace{-7pt} &=& \hspace{-7pt}\mathbf{U}\mathbf{W}_{t}^\alpha[k]+(\mathbf{I}-\mathbf{U})\left(\mathbf{W}_{t}^\alpha[k] -\mathbf{Q}_{t}[k] \right)\nonumber\\
    \hspace{-7pt} && \hspace{-7pt}+\mathbf{A}[k]\left(\mathbf{W}_{t}^{\beta_1}[k]-\mathbf{W}_{t}^\alpha[k] \right)-\mathbf{1}\overline{\mathbf{W}}_{t}^\alpha[k] \nonumber\\
    \hspace{-7pt} && \hspace{-7pt} -\mathbf{1}\sum_{i\in\mathcal{S}_{t}}\frac{1}{M} a_{i,\alpha\beta_1}[k]\left( \mathbf{w}_{t}^{i,\beta_1}[k]-\mathbf{w}_{t}^{i,\alpha}[k] \right)^{\top}\nonumber\\
    \hspace{-7pt} &=& \hspace{-7pt} \mathbf{U}\mathbf{\Theta}\mathbf{W}_{t}^\alpha[k] + (\mathbf{I}-\mathbf{U})\mathbf{\Delta}_{t}[k]\nonumber\\
    \hspace{-7pt} && \hspace{-7pt}+\mathbf{A}[k]\mathbf{\Theta}\left(\mathbf{W}_{t}^{\beta_1}[k]-\mathbf{W}_{t}^\alpha[k]\right).
  \end{eqnarray}

Notice that the matrix $\mathbf{U}$ is doubly stochastic under the condition $\epsilon\in(0,\frac{M}{M-1})$. Accordingly, taking Frobenius norm on both sides of \eqref{35_1} and building upon the facts that $\|\mathbf{U}\|< \lambda_{2,\mathbf{U}}$, $\|\mathbf{I}-\mathbf{U}\|\leq2$ and $\|\mathbf{\Theta}\|<1$ yield \par
 \vspace{-12pt}\begin{small}
\begin{eqnarray*}
\hspace{-10pt} && \hspace{-9pt}\|\mathbf{X}_{t}[k+1]\|\\
   \hspace{-10pt} &\leq& \hspace{-9pt} \|\mathbf{U}\Theta \mathbf{W}_{t}^\alpha[k]\|\hspace{-1pt}+\hspace{-1pt}\|(\mathbf{I}\hspace{-1pt}-\hspace{-1pt}\mathbf{U})\mathbf{\Delta}_{t}[k]\|\hspace{-1pt}+\hspace{-1pt}\|\mathbf{A}[k]\mathbf{\Theta}(\mathbf{W}_{t}^{\beta_1}[k]\hspace{-1pt}-\hspace{-1pt}\mathbf{W}_{t}^\alpha[k])\| \\
   \hspace{-10pt} &\leq& \hspace{-9pt} \lambda_{2,\mathbf{U}}\|\mathbf{\Theta} \mathbf{W}_{t}^\alpha[k]\| +2\|\mathbf{\Delta}_{t}[k]\|+a_{\alpha\beta_1}^{\max}[k] \|\mathbf{W}_{t}^{\beta_1}[k]-\mathbf{W}_{t}^\alpha[k]\| \\
  \hspace{-10pt} &\leq& \hspace{-9pt} \lambda_{2,\mathbf{U}}^{k+1}\|\mathbf{X}_{t}[0]\| +2\sum_{l=0}^{k}\lambda_{2,\mathbf{U}}^{k-l}\|\mathbf{\Delta}_{t}[l]\|+ \widetilde{W}_{t}^{\max} \sum_{l=0}^{k}\lambda_{2,\mathbf{U}}^{k-l} a_{\alpha\beta_1}^{\max}[l].
\end{eqnarray*}
\end{small}%
By recalling the initial setting of visible submodels after each learning round $t$, specifically, $\mathbf{w}_{t}^{i,\alpha}[0]=\mathbf{w}_{t}^{j,\alpha}[0]$ $\forall \{i,j\}\in\mathcal{S}_{t}$, which implies that $\|\mathbf{X}_{t}[0]\|=0$, then the above inequality further turns into \eqref{lemma7}.
\end{proof}

\begin{lemma}
Let $\epsilon\in(0,\frac{M}{M-1})$ and $\tilde{\epsilon}=\max\{1,\epsilon\}$. For any $i\in\mathcal{S}_{t}$, define $a_{i,\alpha\beta_1}[k]=\frac{\gamma_i}{k+1}$ such that $a_{i,\alpha\beta_1}[k]\leq 2a_{i,\alpha\beta_1}[2k]$, where $\gamma_i>0$. If Assumption 3 holds, then it has the following properties of MSPDQ-FL: $\mathbf{w}_{t}^{i,\alpha}[k+1]\in\mathcal{R}_{t}^i[k+1]$
 and
$\|\mathbf{\Delta}_{t}^i[k]\|\leq \frac{\sqrt{d}\pi_{t}}{2^B-1} a_{\alpha\beta_1}^{\max}[k]$.
\end{lemma}

\begin{proof}
By subtracting $\mathbf{q}_{t}^i[k]$ on both sides of \eqref{decomposition2},
it has that\par
 \vspace{-12pt} \begin{small}
\begin{eqnarray}\label{lemma837_1}
  \hspace{-24pt}&&\hspace{-7pt}\| \mathbf{w}_{t}^{i,\alpha}[k+1]-\mathbf{q}_{t}^i[k] \|\nonumber\\
   \hspace{-24pt}&\leq &\hspace{-7pt}\left\|\mathbf{\Delta}_{t}^i[k] \right\|+\epsilon\Big\|\sum\nolimits_{j\in S_{t}}\frac{1}{M}\big(\mathbf{q}_{t}^j[k]-\mathbf{q}_{t}^{i}[k]\big) \Big\| \nonumber\\ 
   \hspace{-24pt}&&\hspace{-7pt} +a_{i,\alpha\beta_1}[k]\left\| \mathbf{w}_{t}^{i,\beta_1}[k]- \mathbf{w}_{t}^{i,\alpha}[k]\right\|\nonumber\\
    \hspace{-24pt}&\overset{(a)}{\leq}&\hspace{-7pt} (1+\epsilon)\left\| \mathbf{\Delta}_{t}[k] \right\|+\epsilon\left\| \mathbf{W}_{t}^\alpha[k]-\mathbf{1}\overline{\mathbf{W}}_{t}^\alpha[k] \right\|\nonumber\\
    \hspace{-24pt}&&\hspace{-7pt}+a_{\alpha\beta_1}^{\max}[k] \left\| \mathbf{W}_{t}^{\beta_1}[k]- \mathbf{W}_{t}^{\alpha}[k]\right\|\nonumber\\
    \hspace{-24pt}&\overset{(b)}{\leq}&\hspace{-7pt} (1+\epsilon)\left\|\mathbf{\Delta}_{t}[k]\right\|+2\epsilon\sum\nolimits_{l=0}\nolimits^{k-1}\lambda_{2,\mathbf{U}}^{k-1-l}\|\mathbf{\Delta}_{t}(l)\|\nonumber\\
    \hspace{-24pt}&&\hspace{-7pt}+\epsilon \widetilde{W}_{t}^{\max} \sum\nolimits_{l=0}\nolimits^{k-1}\lambda_{2,\mathbf{U}}^{k-l-1}a_{\alpha\beta_1}^{\max}[l]+\widetilde{W}_{t}^{\max} a_{\alpha\beta_1}^{\max}[k]\nonumber\\
    \hspace{-24pt}&{\leq}&\hspace{-7pt} (\epsilon+\tilde{\epsilon})\sum\nolimits_{l=0}\nolimits^{k}\lambda_{2,\mathbf{U}}^{k-l}\| \mathbf{\Delta}_{t}[l] \| + \tilde{\epsilon}\widetilde{W}_{t,k}^{\max}\sum\nolimits_{l=0}\nolimits^{k}\lambda_{2,\mathbf{U}}^{k-l}a_{\alpha\beta_1}^{\max}[l]
\end{eqnarray}
\end{small}%
where the inequality $(a)$ is obtained due to the inequality
  $\|\sum_{j\in S_{t}}\frac{1}{M}(\mathbf{q}_{t}^j[k]-\mathbf{q}_{t}^{i}[k]) \|  \leq\|\mathbf{W}_{t}^\alpha[k]-\mathbf{1}\overline{\mathbf{W}}_{t}^\alpha[k] \|+\|\mathbf{\Delta}_{t}[k]\|$.
By combining with the upper bound of $\|\mathbf{W}_{t}^\alpha[k]-\mathbf{1}\overline{\mathbf{W}}_{t}^\alpha[k]\|$ in Lemma 9, we obtain the inequality $(b)$.

It is clear that $\mathbf{w}_{t}^{i,\alpha}[0]\in \mathcal{R}_{t}^i[0]$ under initial settings as stated in Section \uppercase\expandafter {\romannumeral4}.A. Suppose that $\mathbf{w}_{t}^{i,\alpha}[k]\in \mathcal{R}_{t}^i[k]$ holds for some ${k}>0$. Then, for any $k_3\in[0,k]$, it has that
$ \|\mathbf{\Delta}_{t}^i[k_3]\|\leq \frac{\sqrt{d}\pi_{t}}{2^B-1}a_{\alpha\beta_1}^{\max}[k_3], \forall i\in\mathcal{S}_{t}
$,
which together with Assumption 3 yields $\|\mathbf{\Delta}_{t}[k_3]\| \leq a_{\alpha\beta_1}^{\max}[k_3]$. By substituting it into \eqref{lemma837_1}, we arrive at\par
\vspace{-10pt}
\begin{small}
\begin{eqnarray}
   \hspace{-12pt}&&\hspace{-10pt} \left\|\mathbf{w}_{t}^{i,\alpha}[k+1]-\mathbf{q}_{t}^i[k]\right\| \nonumber\\
   \hspace{-12pt}&\leq&\hspace{-10pt}(\epsilon+\tilde{\epsilon})\sum\nolimits_{l=0}\nolimits^{k}\lambda_{2,U}^{k-l}a_{\alpha\beta_1}^{\max}(l) + \tilde{\epsilon}\widetilde{W}_{t,k}^{\max}\sum\nolimits_{l=0}\nolimits^{k}\lambda_{2,\mathbf{U}}^{k-l}a_{\alpha\beta_1}^{\max}(l) \nonumber\\
   \hspace{-12pt}&\leq&\hspace{-10pt}(\epsilon+\tilde{\epsilon}+\tilde{\epsilon}\widetilde{W}_{t,k}^{\max})\Big[\sum\nolimits_{l=0}\nolimits^{\lfloor \frac{k}{2}\rfloor}\lambda_{2,\mathbf{U}}^{k-l}a_{\alpha\beta_1}^{\max}[l]+ \sum\nolimits_{l= \lceil \frac{k}{2}\rceil }\nolimits^{k}\lambda_{2,\mathbf{U}}^{k-l}a_{\alpha\beta_1}^{\max}[l]\Big]\nonumber\\
   \hspace{-12pt}&\overset{(a)}{<}&\hspace{-10pt}(\epsilon+\tilde{\epsilon}+\tilde{\epsilon}\widetilde{W}_{t,k}^{\max})\frac{\gamma_{\max}\lambda_{2,\mathbf{U}}^{\lceil \frac{k}{2}\rceil}+ a_{\alpha\beta_1}^{\max}[\lceil\frac{k}{2}\rceil]}{1-\lambda_{2,\mathbf{U}}}\nonumber\\
   \hspace{-12pt}&{\leq}&\hspace{-10pt}\frac{4\gamma_{\max}(\epsilon+\tilde{\epsilon}+\tilde{\epsilon}\widetilde{W}_{t,k}^{\max})}{(1-\lambda_{2,\mathbf{U}})(k+1)}\nonumber\\
   \hspace{-12pt}&=&\hspace{-10pt}\frac{\pi_{t}}{2}a_{\alpha\beta_1}^{\max}[k],
\end{eqnarray}
\end{small}%
where the inequality $(a)$ follows from the facts that $1-\lambda_{2,\mathbf{U}}\in (0,1)$ and $a_{\alpha\beta_1}^{\max}[0]=\gamma_{\max}>a_{\alpha\beta_1}^{\max}[k]$ for $k>0$. The last inequality is obtained by $\lambda_{2,\mathbf{U}}^k\leq\frac{1}{\gamma_{\max}}a_{\alpha\beta_1}^{\max}[k]$ and $a_{\alpha\beta_1}^{\max}[\lceil \frac{k}{2}\rceil]\leq 2a_{\alpha\beta_1}^{\max}[k]$. Recalling that $k$ can be an arbitrary number, therefore it always has that $\mathbf{w}_{t}^{i,\alpha}[k+1]\in\mathcal{R}_{t}^i[k+1]$ and the quantization error associated with the dynamic quantization interval is bounded by $\frac{\sqrt{d}\pi_{t}}{2^B-1}a_{\alpha\beta_1}^{\max}[k]$.
\end{proof}

\begin{lemma}
(Properties of stochastic quantization):
With the stochastic quantization as in \eqref{sq}, each split submodel of visible part is unbiasedly estimated if Assumption 3 holds, i.e.,
  $\mathbb{E}_{SQ}\{\mathbf{q}_{t}^i[k]\}=\mathbf{w}_{t}^{i,\alpha}[k]$,
and the associated quantization error is bounded by
$
  \mathbb{E}_{SQ}\left\{\|\mathbf{q}_{t}^i[k]-\mathbf{w}_{t}^{i,\alpha}[k]\|^2\right\}\leq\frac{d\pi_{t}^2(a_{\alpha\beta_1}^{\max}[k-1])^2}{4(2^B-1)^2}$.
\end{lemma}

\begin{proof}
  Proof of Lemma 11 is trivial, and has been discussed in the existing literature. In Lemma 10, it shows that $\mathbf{w}_{t+1}^{i,\alpha}[k]\in\mathcal{R}_{t+1}^i[k]$ for all $k$ after any $(t+1)$-th learning round, thereby it is reasonable to assume that any element of $\mathbf{w}_{t+1}^{i,\alpha}[k]$, i.e., ${w}_{t+1}^{i,j,\alpha}[k]$, is located in the interval $[c_\tau,c_{\tau+1})$ for some $\tau\in\{0,...,l-1\}$. With some calculations, it has $\mathbb{E}_{SQ}\left\{q_{t+1}^{i,j}[k]\right\}=w_{t+1}^{i,j,\alpha}[k]$.

For $\mathbf{w}_{t+1}^{i,\alpha}[k]=\big[w_{t+1}^{i,1,\alpha}[k],...,w_{t+1}^{i,d,\alpha}[k]\big]^{\top}$, it is clear that $\mathbb{E}_{SQ}\left\{\mathbf{q}_{t+1}^{i}[k]\right\} = \mathbf{w}_{t+1}^{i,\alpha}[k]$.

The quantization error is given by
\begin{eqnarray*}\label{lemma9a}
  \hspace{-7pt}&&\hspace{-7pt}\mathbb{E}_{SQ}\left\{\left\|q_{t+1}^{i,j}[k] - w_{t+1}^{i,j,\alpha}[k]\right\|^2\right\} \nonumber\\
  \hspace{-7pt}&=&\hspace{-7pt} (c_\tau-w_{t+1}^{i,j,\alpha}[k])^2 {Pr}\left\{q_{t+1}^{i,j}[k]=c_{\tau}\right\}\nonumber\\
        \hspace{-7pt}&&\hspace{-7pt}+(c_{\tau+1}-w_{t+1}^{i,j,\alpha}[k])^2{Pr}\left\{q_{t+1}^{i,j}[k]=c_{\tau+1}\right\}\nonumber\\
  \hspace{-7pt}&\leq&\hspace{-7pt}  (\frac{c_{\tau+1}-c_\tau}{2})^2\leq\frac{\pi_{t}^2(a_{\alpha\beta_1}^{\max}[k-1])^2}{4(2^B-1)^2}
\end{eqnarray*}
where the last inequality is obtained by substituting into the magnitude of dynamic interval at the $k$-th communication round, i.e.,
  $c_{\tau+1}-c_\tau=\frac{\pi_{t} a_{\alpha\beta_1}^{\max}[k-1]}{l-1}$.
Meanwhile, it has
\begin{eqnarray*}
  \hspace{-10pt}&&\hspace{-7pt}\mathbb{E}_{SQ}\left\{\|\mathbf{q}_{t+1}^i[k]-\mathbf{w}_{t+1}^{i,\alpha}[k]\|^2\right\}\\
  \hspace{-10pt}&=&\hspace{-7pt}\sum_{j=1}^{d}\mathbb{E}_{SQ}\left\{\|q_{t+1}^{i,j}[k]-w_{t+1}^{i,j,\alpha}[k]\|^2\right\}\\
  \hspace{-10pt}&\leq&\hspace{-7pt} \frac{d\pi_{t}^2 (a_{\alpha\beta_1}^{\max}[k-1])^2}{4(2^B-1)^2}.
\end{eqnarray*}
\end{proof}

\begin{lemma} (Unbiased and variance bounded quantization):
Let Assumption 3 hold, under MSPDQ-FL, we have
$\mathbb{E}_{SQ}\left\{\overline{\mathbf{w}}_{t+1}\right\}=\overline{\mathbf{v}}_{t+1}$ and
$\mathbb{E}_{SQ}\left\{\|\overline{\mathbf{w}}_{t+1}-\overline{\mathbf{v}}_{t+1}\|^2\right\}\leq \frac{d\pi_{t+1}^2 (a_{\alpha\beta_1}^{\max}[K_t-1])^2}{4M(2^B-1)^2}$.
\end{lemma}

\begin{proof}
Due to the unbiasedness of partial model aggregation, it has
$ \overline{\mathbf{v}}_{t+1}=\mathbb{E}_{RS}\{\sum\nolimits_{i\in\mathcal{S}_{t+1}}\frac{1}{M}\text{Aggr}(\mathbf{r}_{t+1}^i)\}$ and $ \overline{\mathbf{w}}_{t+1}=\mathbb{E}_{RS}\{\sum\nolimits_{i\in\mathcal{S}_{t+1}}\frac{1}{M}\text{Aggr}(\mathbf{q}_{t+1}^i)\}$ that are similar to \eqref{22_0} and \eqref{22_1}. Meanwhile, noting that after $K_t$ communication rounds, it has the following facts, i.e., $\text{Aggr}(\mathbf{q}_{t+1}^i)=\mathbf{q}_{t+1}^i[K_t]$ and $\text{Aggr}(\mathbf{r}_{t+1}^i)=\mathbf{w}_{t+1}^{i,\alpha}[K_t]$. Then, one has
\begin{eqnarray}
  \hspace{-7pt}\mathbb{E}_{SQ}\{\overline{\mathbf{w}}_{t+1}\}\hspace{-7pt}&= &\hspace{-7pt} \mathbb{E}_{SQ}\left\{\mathbb{E}_{RS}\Big\{\sum\nolimits_{i\in\mathcal{S}_{t+1}}\frac{1}{M}\mathbf{q}_{t+1}^i[K_t]\Big\}\right\}\nonumber\\
  \hspace{-7pt}\hspace{-7pt}&\overset{(a)}{=}&\hspace{-7pt} \mathbb{E}_{RS}\Big\{\sum\nolimits_{i\in\mathcal{S}_{t+1}}\frac{1}{M}\mathbf{w}_{t+1}^{i,\alpha}[K_t]\Big\}\nonumber\\
 \hspace{-7pt}\hspace{-7pt}& =&\hspace{-7pt} \mathbb{E}_{RS}\Big\{ \sum\nolimits_{i\in\mathcal{S}_{t+1}}\frac{1}{M}\text{Aggr}(\mathbf{r}_{t+1}^i)\Big\}=\overline{\mathbf{v}}_{t+1}
\end{eqnarray}
where the equality $(a)$ is obtained in light of Lemma 11 and the independence between $\mathbb{E}_{RS}$ and $\mathbb{E}_{SQ}$.

Following the similar arguments as in \eqref{bottom}, and using Lemma 11 lead to
  \begin{eqnarray}
     \hspace{-7pt}&&\hspace{-7pt} \mathbb{E}_{SQ}\left\{\left\|\overline{\mathbf{w}}_{t+1}-\overline{\mathbf{v}}_{t+1}\right\|^2\right\}\nonumber\\
     \hspace{-7pt}&=&\hspace{-7pt} \mathbb{E}_{SQ}\bigg\{\Big\|\sum\nolimits_{i\in\mathcal{S}_{t+1}} \frac{1}{M}\mathbf{q}_{t+1}^i[K_t]-\sum\nolimits_{i\in\mathcal{S}_{t+1}}\frac{1}{M}\mathbf{w}_{t+1}^{i,\alpha}[K_t] \Big\|^2\bigg\}\nonumber\\
     \hspace{-7pt}&=&\hspace{-7pt} \sum\nolimits_{i\in\mathcal{S}_{t+1}}\frac{1}{M^2}\mathbb{E}_{SQ}\left\{\left\|\mathbf{\Delta}_{t+1}^i[K_t] \right\|^2\right\}\nonumber\\
     \hspace{-7pt}&&\hspace{-7pt}+ \mathbb{E}_{SQ}\bigg\{\sum\nolimits_{i,j\in\mathcal{S}_{t+1},i\neq j}\left\langle \frac{1}{M}\mathbf{\Delta}_{t+1}^i[K_t],\frac{1}{M}\mathbf{\Delta}_{t+1}^j[K_t] \right\rangle\bigg\} \nonumber\\
     \hspace{-7pt}&\leq&\hspace{-7pt}  \frac{d\pi_{t+1}^2 (a_{\alpha\beta_1}^{\max}[K_t-1])^2}{4M(2^B-1)^2}
  \end{eqnarray}
where the inequality holds for $\mathbb{E}_{SQ}\left\{\mathbf{\Delta}_{t+1}^i[k]\right\}=0$, and $\mathbf{\Delta}_{t+1}^i[k]$ and $\mathbf{\Delta}_{t+1}^j[k]$ are independent for any $k$ if $i\neq j$.
\end{proof}

\subsection{Proof of Theorem 3}
\begin{proof}
Similar to the proof of Theorem 2, we rewrite the expectation over $\|\overline{\mathbf{w}}_{t+1}-\mathbf{w}^*\|^2$ as
\begin{eqnarray}\label{37_1}
 \hspace{-25pt}&&\hspace{-7pt} \mathbb{E}\left\{\|\overline{\mathbf{w}}_{t+1}-\mathbf{w}^*\|^2\right\}
  ={2\mathbb{E} \left\{\left\langle \overline{\mathbf{w}}_{t+1}-\overline{\mathbf{v}}_{t+1},\overline{\mathbf{v}}_{t+1}-{\mathbf{w}}^*  \right\rangle\right\}} \nonumber\\
  \hspace{-25pt}&&\hspace{-7pt}\qquad\qquad+\mathbb{E}\left\{\|\overline{\mathbf{w}}_{t+1}-\overline{\mathbf{v}}_{t+1}\|^2\right\}+\mathbb{E}\left\{\|\overline{\mathbf{v}}_{t+1}-{\mathbf{w}}^*\|^2\right\}.
\end{eqnarray}
If $t+1\notin \mathcal{I}_E$, it has $\overline{\mathbf{w}}_{t+1}=\overline{\mathbf{v}}_{t+1}$, leading to
\begin{eqnarray}\label{38_1}
\hspace{-25pt}&&\hspace{-7pt}\mathbb{E}\left\{\|\overline{\mathbf{w}}_{t+1}-\mathbf{w}^*\|^2\right\}
\leq (1-\eta_t\mu)\mathbb{E}\left\{\|\overline{\mathbf{w}}_t-\mathbf{w}^*\|^2\right\} \nonumber\\
    \hspace{-25pt}&&\hspace{-7pt}\qquad\qquad +\eta_t^2\Big[ \sum\nolimits_{i=1}\nolimits^{N}p_i^2\sigma_i+6L\Gamma+8(E-1)^2G\Big],
\end{eqnarray}
which is exactly the result as in \eqref{t+1not}.
If $t+1\in \mathcal{I}_E$, it has
\begin{eqnarray}\label{39_1}
 \hspace{-12pt}&&\hspace{-8pt}\mathbb{E}\left\{\|\overline{\mathbf{w}}_{t+1}-\mathbf{w}^*\|^2\right\}\nonumber\\
\hspace{-12pt}&\overset{(a)}{=}&\hspace{-8pt}\mathbb{E}_{SQ}\left\{\|\overline{\mathbf{w}}_{t+1}-\overline{\mathbf{v}}_{t+1}\|^2\right\}+\mathbb{E}\left\{\|\overline{\mathbf{v}}_{t+1}-\mathbf{w}^*\|^2\right\}\nonumber\\
 \hspace{-12pt}&\overset{(b)}{\leq}&\hspace{-8pt}\frac{d\pi_{t+1}^2 (a_{\alpha\beta_1}^{\max}[K_t-1])^2}{4M(2^B-1)^2}+\mathbb{E}\left\{\|\overline{\mathbf{v}}_{t+1}-\mathbf{w}^*\|^2\right\},
\end{eqnarray}
where the equality $(a)$ and the first term on the right-hand side of the inequality $(b)$ follow from Lemma 12.

Note that $K_t=\max\left\{\lceil\log_{\lambda}^{2/\mu(\vartheta+t)}\rceil, \lceil\frac{\mu(\vartheta+t)}{2}\rceil \right\}$ and $a_{\alpha\beta_1}^{\max}[K_t-1]=\frac{\gamma_{\max}}{K_t}$. Via a simple calculation, we have
$a_{\alpha\beta_1}^{\max}[K_t-1]\leq\gamma_{\max}\eta_t$
 and
$
  \mathbb{E}\left\{\|\overline{\mathbf{v}}_{t+1}-\overline{\mathbf{u}}_{t+1}\|^2\right\}\leq \eta_t^2D_1
$
which is obtained in the same way as that in \eqref{lemma6lasteq}. Next, by resorting to similar arguments used from \eqref{27_1} to \eqref{t+1in}, one can derive\par
 \vspace{-12pt}\begin{small}
 \begin{eqnarray}\label{41_1}
   \hspace{-13pt}\mathbb{E}\left\{\|\overline{\mathbf{v}}_{t+1}\hspace{-1pt}-\hspace{-1pt}\mathbf{w}^*\|^2 \right\} \hspace{-8pt}&\leq&\hspace{-8pt} (1\hspace{-1pt}-\hspace{-1pt}\eta_t\mu)\mathbb{E}\left\{\|\overline{\mathbf{w}}_{t}\hspace{-1pt}-\hspace{-1pt}\mathbf{w}^*\|^2\right\}
   \hspace{-1pt}+\hspace{-1pt}\eta_t^2D_2,
 \end{eqnarray}%
 \end{small}%
where $D_2$ is defined in \eqref{fixediteration}. Plugging the inequality $a_{\alpha\beta_1}^{\max}[K_t-1]\leq\gamma_{\max}\eta_t$ and \eqref{41_1} into \eqref{39_1} yields
\begin{eqnarray}\label{42_1}
   \hspace{-7pt}\mathbb{E}\left\{\|\overline{\mathbf{w}}_{t+1}-\mathbf{w}^*\|^2\right\} \hspace{-7pt}&\leq&\hspace{-7pt} (1-\eta_t\mu)\mathbb{E}\left\{\|\overline{\mathbf{w}}_{t}-\mathbf{w}^*\|^2\right\}\nonumber\\
   \hspace{-7pt}&&\hspace{-12pt} +\eta_t^2\Big(D_2+\frac{d\gamma_{\max}^2\tilde{\pi}^2}{4M(2^B-1)^2}\Big).
 \end{eqnarray}

By \eqref{38_1} and \eqref{42_1}, one can conclude that regardless of whether $t+1\in\mathcal{I}_E$ or $t+1\notin\mathcal{I}_E$, it always holds that
\begin{eqnarray*}
  \mathbb{E}\left\{\|\overline{\mathbf{w}}_{t+1}-\mathbf{w}^*\|^2\right\}\hspace{-7pt} &\leq& \hspace{-7pt} (1-\eta_t\mu)\mathbb{E}\left\{\|\overline{\mathbf{w}}_{t}-\mathbf{w}^*\|^2\right\}\\
  \hspace{-7pt} && \hspace{-7pt}+\eta_t^2(D_2+D_3),
\end{eqnarray*}
where, for brevity, let $D_3=\frac{d\gamma_{\max}^2\tilde{\pi}^2}{4M(2^B-1)^2}$. By virtue of the mathematical induction method again and invoking similar arguments shown in Theorem 2 from \eqref{fixediteration} to \eqref{theorem3result}, we derive \par
\vspace{-10pt}
\begin{small}
\begin{eqnarray}\label{56}
 \mathbb{E}\left\{\|\overline{\mathbf{w}}_{t+1}-\mathbf{w}^*\|^2\right\} \hspace{-7pt} &\leq& \hspace{-7pt} \frac{\nu_2}{\vartheta+t+1}
\end{eqnarray}
\end{small}%
together with the result as shown in \eqref{theorem4jieguo} by $L$-smoothness, i.e.,
  \begin{eqnarray*}
   \hspace{-23pt}&&\hspace{-10pt} \mathbb{E}\left\{F(\overline{\mathbf{w}}_{t})\right\}-F^* \nonumber\\
    \hspace{-23pt}&&\hspace{-10pt} \leq \frac{8L}{\mu(\vartheta\hspace{-1pt}+\hspace{-1pt}t)}\hspace{-1pt}\left(\hspace{-1pt}\frac{2(D_2\hspace{-1pt}+\hspace{-1pt}D_3)}{\mu}\hspace{-1pt}+\hspace{-1pt}\frac{8L\hspace{-1pt}+\hspace{-1pt}\mu E}{2}\mathbb{E}\left\{\|\overline{\mathbf{w}}_0\hspace{-1pt}-\hspace{-1pt}\mathbf{w}^*\|^2\right\}\hspace{-2pt}\right),
  \end{eqnarray*}
where $\nu_2=\max\big\{\frac{\rho^2(D_2+D_3)}{\rho\mu-1},(\vartheta+1)\mathbb{E}\{\left\|\overline{\mathbf{w}}_{0}-\mathbf{w}^*\|^2\right\}\big\}$.
\end{proof}


\end{document}